\documentclass[11pt,reqno]{amsart}

\usepackage[T1]{fontenc}
\usepackage[utf8]{inputenc}

\usepackage{hyperref}

\allowdisplaybreaks[4]
\usepackage{algorithm}
\usepackage{algpseudocode}

\usepackage{amsmath}
\usepackage{amssymb}
\usepackage{mathrsfs}
\usepackage{amsthm}
\usepackage{bm}
\usepackage{graphicx}
\usepackage{tikz}
\usepackage{booktabs}

\usepackage[caption=false]{subfig}
\usepackage{caption}

\usepackage{float}

\newfloat{longfigure}{htbp}{lof}
\floatname{longfigure}{Figure}
\newsubfloat{longfigure}

\usepackage{geometry}
\usepackage{color}

\usepackage{color}
\theoremstyle{plain}
\newtheorem{lemma}{Lemma}[section]
\newtheorem{theorem}[lemma]{Theorem}
\newtheorem{definition}[lemma]{Definition}

\newtheorem{remark}[lemma]{Remark}
\newtheorem{proposition}[lemma]{Proposition}
\newtheorem{assumption}{Assumption}
\newtheorem{example}{Example}

 \makeatletter
\@namedef{subjclassname@2020}{%
  \textup{2020} Mathematics Subject Classification}
\makeatother

\begin{document}
\title[Schemes for HJE on graphs]{First-Order Convergence of Monotone Schemes for Hamilton--Jacobi Equations on the Wasserstein Space on Graphs}
\author{Jianbo Cui, Tonghe Dang}
\address{Department of Applied Mathematics, The Hong Kong Polytechnic
University, Hung Hom, Kowloon, Hong Kong, SAR, China.}
\email{jianbo.cui@polyu.edu.hk;tonghe.dang@polyu.edu.hk(Corresponding author)}
\thanks{This work is supported by MOST National Key R\&D Program No. 2024FA1015900, the Hong Kong Research Grant Council GRF grant 15302823, GRF grant 15301025, NSFC/RGC Joint Research Scheme N$\_$PolyU5141/24, NSFC grant 12522119, NSFC grant 12301526, internal funds (P0041274, P0045336)  from Hong Kong Polytechnic University,  and the
CAS AMSS-PolyU Joint Laboratory of Applied Mathematics.}
\begin{abstract}

We prove first-order convergence of semi-discrete monotone finite difference schemes for Hamilton--Jacobi equations on the Wasserstein space over a finite graph. 
A central challenge is the boundary degeneracy of the Wasserstein simplex, 
which prevents the direct use of the standard 
 $L^1$ adjoint method and limits doubling-of-variables arguments to the suboptimal rate 
$\mathcal O(h^{\frac 12})$ \cite{CDM25}. We address this issue by introducing a weighted 
$L^1$
  framework with a boundary-vanishing weight and by analyzing the corresponding weighted adjoint equation for the linearized operator of the scheme, featuring a new
  geometric drift term.   Our proof relies on uniform bounds for the weighted adjoint variable and the mesh-parameter derivative of the numerical solution. 
  These estimates are derived from discrete gradient and semi-concavity bounds, obtained through a bootstrap argument for two classes of monotone Hamiltonians.



\end{abstract}
\maketitle
\section{Introduction}
Hamilton--Jacobi equations (HJEs) on metric spaces 
beyond the classical Euclidean setting have attracted 
increasing attention, driven in part by applications in mean field control~\cite{BCP,CDJS,Daudin}, 
mean field games~\cite{BFY,CardaliaguetPorretta,MFG_Caines,LasryLions,CarmonaDelarue1,CarmonaDelarue2}, 
and stochastic optimal control on structured state spaces~\cite{GNT,GT,CGKPR,approximation_HJB,approximation_HJB2,XiangY1,Daudin2}. 
One of the most natural and practically relevant settings is the Wasserstein space over a finite graph, 
where the state variable is a probability distribution 
over the vertices of a weighted connected graph and 
evolves according to mass transport dynamics encoded 
by the graph topology.
This discrete geometric framework, developed 
in~\cite{Maas, Chow2, Liwuchen}, lies at the 
intersection of discrete optimal transport, Markov 
chain theory, and finite-dimensional differential 
geometry. The probability simplex is endowed with a 
Riemannian-like metric tensor that encodes the 
combinatorial structure of the graph, while the 
 associated differential operators are tailored to this geometry. 
These features make the Wasserstein space on graphs 
a natural setting for modeling population dynamics, 
multi-agent systems, and reversible Markov chains on 
discrete domains~\cite{Chow2, Cui_Hamiltonian, 
Mielke, HJSK}.
At the same time, they lead to analytical and numerical difficulties that do not arise in the classical Euclidean setting.

The HJE on the Wasserstein space over a finite 
connected graph $G=(V,E,\omega)$ takes the form
\begin{align}\label{eq_intro}
  \partial_t u(t,\xi) 
  + \mathcal{H}(\xi,\nabla_{\mathcal{W}}u(t,\xi)) 
  + \mathcal{F}(\xi) = 0,
  \quad t\in(0,T),
  \qquad
  u(0,\xi) = \mathcal{U}_0(\xi),
\end{align}
where $\xi\in\mathcal{P}(G)$ is a probability vector 
on the vertex set $V$, with $d:=|V|\geq 2$ denoting the number of vertices, $\nabla_{\mathcal{W}}$ is the 
Wasserstein gradient on $(\mathcal{P}(G),\mathcal{W})$, and $\mathcal{U}_0, \mathcal{F}$ are initial datum and potential, respectively; see Section~\ref{eq_assp} for details. The well-posedness of \eqref{eq_intro} in the viscosity sense was established in~\cite{MCC}  under certain convexity and growth conditions on the Hamiltonian $\mathcal{H}$.    
This equation arises, for example, in mean field control on graphs, where $u$ describes  the value function of a population optimization problem in which agents minimize a running cost while their empirical distribution evolves according to a graph continuity equation~\cite{MCC}.
Despite this well-posedness result, the numerical approximation of \eqref{eq_intro} remains largely undeveloped.   
Recently, in \cite{CDM25}  the authors introduced a discretization framework based on monotone finite-difference schemes on a discretized probability simplex and proved a
convergence rate $\mathcal{O}(h^{1/2})$. 
Their analysis relies on a barrier function to keep the discrete solution away from  $\partial\mathcal{P}(G)$, 
{together with a structural boundary assumption on a compact subset of the simplex}. However,
numerical experiments in~\cite{CDM25} consistently 
showed first-order accuracy in both $L^1$ and 
$L^\infty$ norms.
Explaining this discrepancy between theory and computation, and establishing a first-order convergence result, was left open in \cite{CDM25}.

The main objective of this paper is to close, in a weighted 
$L^1$
  sense, the gap between the theoretical and observed convergence rates. 
Our main result (Theorem \ref{thm1})  
 establishes the first-order error estimate 
\[\int_{\mathcal{P}(G)}|u^h(T,\xi) - u(T,\xi)|\,w(\xi)\,\mathrm d\mathscr{H}^{d-1}(\xi) \leq Ch\]
for any {admissible} weight function $w$ (see Definition~\ref{def_truncation}), under natural assumptions on {the discrete Hamiltonian and on the discrete gradient and semi-concavity of} the numerical solution (see Assumption~\ref{ass:semi-concave}). Here, $\mathscr H^{d-1}$ denotes the $(d-1)$-dimensional Hausdorff measure on $\mathcal P(G)$ (see \eqref{hausdorff}).
To this end, we need to overcome 
two fundamental obstacles.
First, the doubling-of-variables technique is intrinsically limited  to the $\mathcal{O}(h^{1/2})$ 
regime for general viscosity 
solutions~\cite{Souganidis85, Crandall84, Barles91}.
Although first-order $L^1$ error estimate  can be obtained  in the 
Euclidean setting under the semi-concavity of the exact 
solution~\cite{LT2001}, establishing comparable {regularity on the Wasserstein space over graphs is largely open, due to the intricate interaction between the graph geometry and the probability simplex constraints.}
Second, the probability 
simplex $\mathcal{P}(G)$ has a degenerate boundary $\partial\mathcal{P}(G)$ corresponding to configurations in which some vertices carry 
zero probability mass. 
At such points, the Wasserstein metric tensor degenerates 
and the viscosity solution is naturally defined only on the interior  $\mathcal{P}^\circ(G)$.
The analysis in~\cite{CDM25} {bypasses this difficulty}
by restricting {the problem} to compact subsets 
$\mathcal{P}_\epsilon(G)\subset\mathcal{P}^\circ(G)$ at distance 
$\epsilon>0$ from $\partial\mathcal{P}(G)$. However, the resulting 
$\epsilon$-dependent error prevents a {sharper} estimate on {the full probability simplex $\mathcal{P}(G)$}. Obtaining a genuine first-order rate on  $\mathcal{P}(G)$ 
therefore requires an approach that is {compatible} with the boundary 
degeneracy and avoids any $\epsilon$-dependent domain restriction.

To this end, we develop a weighted 
 $L^1$ framework based on a discrete 
adjoint method.
The key idea is to introduce a weight function 
$w\colon\mathcal{P}(G)\to[0,\infty)$ that vanishes 
on $\partial\mathcal{P}(G)$, and to measure the 
numerical error in the weighted space 
$L^1_w(\mathcal{P}(G))$. We note that the adjoint method for $L^1$ error analysis was initially developed in the context of hyperbolic conservation laws  \cite{Tadmor2}, 
and later extended to HJEs in Euclidean domains (see e.g. \cite{LT2001,Cagnetti13}).
Our weighted adjoint framework has several advantages.
First, the vanishing weight regularizes the boundary behavior and leads to a duality argument adapted to the geometry of the Wasserstein simplex. In particular, the associated weighted adjoint equation yields a first-order error estimate on the full simplex, without any domain truncation.  
Second, the adjoint variable admits a direct probabilistic interpretation as the transition kernel of a discrete diffusion process on $\mathcal{P}(G)$,  generated by the linearization of the numerical Hamiltonian.
In contrast to the Euclidean case, the adjoint equation contains an additional geometric drift term $\mathcal S$ (see \eqref{adjoint1}), which reflects the nonuniformity of the weight on the curved simplex.
Third, the weighted $L^1$ setting is, in principle, flexible enough to accommodate non-smooth initial data, since the duality argument is formulated at the level of integral averages rather than pointwise error bounds.

The proof of Theorem \ref{thm1} proceeds in several steps. We begin by deriving a weighted discrete integration-by-parts formula (see Lemma~\ref{remark_IBP}). The key observation is that the weight $w$ vanishes on the boundary of the simplex, while the numerical Hamiltonian admits a natural zero extension there. 
As a result, 
all boundary contributions cancel, leading to a weighted duality relation between the linearization of the scheme and its adjoint. 
 Building on this identity, we show that the adjoint variable is nonnegative and satisfies weighted mass conservation (see Proposition~\ref{prop_sigma}), and obtain a uniform $L^{\infty}$-bound on the adjoint variable by exploiting the semi-concavity  of the numerical solution (see Proposition~\ref{bound_sigma}). 
By proving the first-order consistency estimates for the truncation errors, we derive an $L^1_w$-bound
on the mesh-parameter derivative of the numerical solution  
(Proposition~\ref{prop_regu}). 
Finally, a Cauchy-type argument based on this bound shows that the numerical solutions converge to the unique viscosity solution, with the claimed first-order rate.

A central part of the analysis is the derivation of two a priori estimates for the numerical solution: a gradient bound and a semi-concavity bound. These estimates are strongly coupled. 
In the Euclidean setting, analogous estimates for numerical approximations are available; see \cite{LT2001,Tadmor2}. In the present Wasserstein graph setting, however, establishing such bounds is more delicate because of the interaction between the graph structure and the simplex constraint on the probability space. 
To address these difficulties, we first derive evolution equations for the gradient and Hessian of 
the numerical solution via the linearized operator 
associated with the scheme.  We then exploit the edgewise decomposition of the Hamiltonian and combine it with the adjoint method to control these quantities. This allows us to establish the gradient and semi-concavity estimates simultaneously through a stopping-time bootstrap argument on  $[0,T]$ (see Section~\ref{sec:examples}).
Finally, we verify these estimates  
for two families of numerical Hamiltonians, including a Lax--Friedrichs 
type scheme and an Osher--Sethian type upwind scheme.   

The remainder of the paper is organized as follows. Section~\ref{sec_2} reviews the basic background on the Wasserstein space on graphs and the corresponding HJEs. Section~\ref{sec_4.1} introduces the finite difference operators, the semi-discrete scheme, and the weighted adjoint equation, and states the main convergence result. Section~\ref{sec_5} provides the proof of Theorem~\ref{thm1} by means of  the adjoint method. Section~\ref{sec:examples} specifies the numerical Hamiltonians and establishes the required gradient and semi-concavity estimates. 
 The appendix collects several auxiliary arguments. 

\section{Preliminaries}\label{sec_2}
In this section, we provide the preliminaries for the Wasserstein space of probability measures on a finite graph. We then 
specify the standing assumptions and collect key properties of the HJE in this setting.

\subsection{Notations}\label{notation}
Consider an undirected connected graph $G=(V,E,\omega)$ with no self-loops or multiple edges, where $V=\{1,\ldots,d\}$ denotes the 
set of vertices with $d \geq 2$ being the total number of vertices,   
and $E\subset V\times V$ is the set of edges.  Here, $\omega=(\omega_{i,j})_{1\le i,j\le d}$ is a $d\times d$ symmetric matrix with nonnegative entries, such that $\omega_{i,j}>0$  if $(i,j)\in E$, $\omega_{i,j}=0$ if $(i,j)\notin E$  and $\omega_{i,i}=0$ for all $i=1,\ldots,d.$ We denote the probability simplex by
\begin{align*}
\mathcal P(G)=\Big\{\xi=(\xi_1,\ldots,\xi_d)\in[0,1]^d:\sum_{i=1}^d\xi_i=1\Big\}.
\end{align*}
For a fixed $\epsilon\in(0,\frac1d),$ we define the truncated simplex $\mathcal P_{\epsilon}(G) :=\mathcal P(G)\cap [\epsilon,1)^d$. For notational simplicity, we will often write $\mathcal P_{\epsilon}:=\mathcal P_{\epsilon}(G)$. Let $\partial A$ and $A^{\circ}$ denote the boundary and interior of a Borel set $A$, respectively. In particular, $\mathcal P^{\circ}(G)$ denotes the interior of $\mathcal P(G)$ and $\mathcal P^{\circ}_{\epsilon}$  the interior of $\mathcal P_{\epsilon}$. Similarly, $\partial \mathcal P(G)=\mathcal P(G)\backslash \mathcal P^{\circ}(G)$ and $\partial\mathcal P_{\epsilon} =\mathcal P_{\epsilon} \backslash\mathcal P^{\circ}_{\epsilon} .$

 Let $\mathbb S^{d\times d}$ denote the set of $d\times d$ skew-symmetric matrices. We introduce a symmetric function $g:[0,\infty)^2 \to[0,\infty)$ such that $g(t,r)=g(r,t)$ for $t,r\in[0,\infty)$. For $\xi\in\mathcal P(G)$, we say that $\upsilon,\tilde\upsilon\in\mathbb S^{d\times d}$ are $\xi$-equivalent if $(\upsilon_{i,j}-\tilde\upsilon_{i,j})g_{i,j}(\xi)=0$ for all $(i,j)\in E,$ where $g_{i,j}(\xi):=g(\xi_i,\xi_j).$ This equivalence relation induces a quotient space on $\mathbb S^{d\times d}$ denoted by $\mathbb H_{\xi}$ \cite{MCC}.  Under conditions specified later (see Assumption \ref{assumption_g}), the function 
$g$ induces a metric tensor on $\mathcal P(G)$.
We endow $\mathbb H_{\xi}$ with the inner product and discrete norm:
\begin{align*}
(\upsilon,\tilde\upsilon)_{\xi}:=\frac12\sum_{(i,j)\in E}\upsilon_{i,j}\tilde\upsilon_{i,j}g_{i,j}(\xi),\quad \|\upsilon\|_{\xi}:=\sqrt{(\upsilon,\upsilon)_{\xi}},\quad \upsilon,\tilde\upsilon \in\mathbb S^{d\times d}.
\end{align*}The factor $\frac12$ compensates for the symmetry of the edge set: if $(i,j)\in E,$ then $(j,i)\in E$ as well.

For a mapping $\phi=(\phi_1,\ldots,\phi_d):V\to\mathbb R^d,$ we define its graph gradient by $\nabla_G\phi:=(\sqrt{\omega_{i,j}}(\phi_i-\phi_j))_{(i,j)\in E}.$ The adjoint of $\nabla_G$ with respect to the inner product $(\cdot,\cdot)_{\xi}$  is the divergence operator $-\mathrm{div}_{\xi}:\mathbb H_{\xi}\to\mathbb R^d,$ given by
\begin{align*}
\mathrm{div}_{\xi}(\upsilon):=\Big(\sum_{j=1}^d\sqrt{\omega_{i,j}}\upsilon_{j,i}g_{i,j}(\xi)\Big)_{i=1}^d,\quad \upsilon\in\mathbb S^{d\times d},
\end{align*}
such that the following integration-by-parts formula holds: $(\nabla_G\phi,\upsilon)_{\xi}=-(\phi,\mathrm{div}{\xi}(\upsilon)).$ Here, $(\upsilon,\tilde\upsilon):=\frac12 \sum_{(i,j)\in E}\upsilon_{i,j}\tilde\upsilon_{i,j}.$ We use $\|\cdot\|_{\infty}$ for the supremum norm and $\|\cdot\|_{l^2}$ for the Frobenius norm of a matrix.

For $\rho^0,\rho^1\in\mathcal P(G),$ the $L^2$-Monge--Kantorovich metric is defined as
\begin{align}\label{MK}
\mathcal W(\rho^0,\rho^1):=\Big(\inf_{(\sigma,\upsilon)}\Big\{\int_0^1(\upsilon,\upsilon)_{\sigma}\mathrm dt:\dot{\sigma}+\mathrm {div}_{\sigma}(\upsilon)=0,\;\sigma(0)=\rho^0,\sigma(1)=\rho^1\Big\}\Big)^{\frac12},
\end{align}
where the infimum is taken over pairs $(\sigma,\upsilon)$ such that $\sigma\in H^1(0,1;\mathcal P(G))$ and $\upsilon:[0,1]\to\mathbb S^{d\times d}$ is measurable. The probability simplex $\mathcal P(G)$, endowed with the metric $\mathcal W$, is  called the \textit{Wasserstein space on graphs}, and is denoted by $(\mathcal P(G),\mathcal W)$; see e.g. \cite{Liwuchen,MCC,Maas} for further background on this space.  
Throughout this paper, we assume the symmetric function $g$ satisfies the following conditions:  
\begin{assumption}\label{assumption_g}
\begin{itemize}
\item[(g-\romannumeral1)] $g$ is continuous on $[0,\infty)^2$ and is smooth on $(0,\infty)^2$;
\item[(g-\romannumeral2)] $t\wedge r\leq g(t,r)\leq t\vee r$ for any $t,r\in[0,\infty);$
\item[(g-\romannumeral3)] $g(\lambda t,\lambda r)=\lambda g(t,r)$ for any $\lambda,t,r\in[0,\infty);$
\item[(g-\romannumeral4)] $g$ is concave;   
\item[(g-\romannumeral5)] $\int_0^1\frac{1}{\sqrt{g(r,1-r)}}\mathrm dr<\infty.$
\end{itemize}
\end{assumption}
As shown in \cite[Proposition 3.7]{Liwuchen}, assumption (g-\romannumeral5) ensures that $\mathcal W(\rho^0,\rho^1)<\infty$ for any $\rho^0,\rho^1\in\mathcal P(G).$ 
Some common examples of symmetric function $g$ satisfying  (g-\romannumeral1)--(g-\romannumeral5) are given as follows. 
\begin{example}\label{ex_g}
(\romannumeral1) The average probability weight \cite{Chow1}: $g_1(t,r)=\frac{t+r}{2};$ 

(\romannumeral2) The logarithmic probability weight \cite{Chow2}: 
 $$g_2(t,r)=\frac{t-r}{\log t-\log r},\text{ if }t\neq r;\;g_2(t,r)=0,\text{ if }t=0\text{ or }r=0;\;g_2(t,r)=t,\text{ if }t=r;
$$ (\romannumeral3) The harmonic probability weight  \cite{Maas}: 
$$g_3(t,r)=0,\text{ if }t=0\text{ or }r=0;\;g_3(t,r)=\frac{2}{\frac1t+\frac1r},\text{ otherwise}.$$
\end{example}

Using the metric tensor $g$, one can define \textit{$\mathcal W$-differentiability} and the \textit{Wasserstein gradient} $\nabla_{\mathcal W}$ on $(\mathcal P(G),\mathcal W)$ \cite[Definition 3.9]{MCC}. For a function $f:\mathcal P(G)\to\mathbb R$ that is $\mathcal W$-differentiable at $\xi\in\mathcal P(G),$ we denote the Wasserstein gradient as $\nabla_{\mathcal W}f(\xi).$ Let \begin{align}\label{mathbbV}
\mathbb{V} := \{ \mathbf a \in \mathbb{R}^d : \|\mathbf a\|_{l^2}=1, \sum_{k=1}^d \mathbf a_k = 0 \}
\end{align} 
denote the set of unit tangent vectors to the probability simplex. 
For any $(i,j)\in E$ with $1\leq i<j\leq d,$ let $e_{i,j}\in\mathbb R^d$ be the vector with $1$ at the $i$th entry, $-1$ at the $j$th entry, and zeros elsewhere. For $f:\mathcal P(G)\to\mathbb R,$ we define
\begin{align*}
\nabla^{e_{i,j}}f(\xi):=\lim_{t\to0}\frac{f(\xi+te_{i,j})-f(\xi)}{t},\quad \xi\in\mathcal P^{\circ}(G).
\end{align*}
If the Fr\'echet derivative $\frac{\delta f}{\delta \xi}$ exists and is continuous at $\xi\in\mathcal P^{\circ}(G)$, then $\nabla_{\mathcal W}f(\xi)=\nabla_{G}\frac{\delta f}{\delta \xi}(\xi)\in\mathbb S^{d\times d},$ and the upper triangular entries of $\nabla_{\mathcal W}f(\xi)$ are given by $\sqrt{\omega_{i,j}}\nabla^{e_{i,j}}f(\xi)$ for $1\leq i<j\leq d.$ 

Throughout this paper, $C$ denotes a generic positive constant that may change from line to line. When needed, we write $C(a,b)$ or $C_{a,b}$ to emphasize dependence on parameters $a$ and $b.$ For a function $F\colon\mathbb{R}^d\to\mathbb{R}$, we denote the Euclidean gradient and Hessian matrix 
by $\nabla_{\xi}F$ and $\nabla^2_{\xi}F$, respectively.
Partial derivatives are denoted  by $\frac{\partial}{\partial{\xi_k}}F$ 
and $\frac{\partial^2}{\partial{\xi_k}\partial{\xi_i}}F$, and we sometimes abbreviate these as 
$\partial_{\xi_k} F$ and $\partial^2_{\xi_k\xi_i}F$ for $i,k=1,2,\ldots,d$.

\subsection{Hamilton--Jacobi equation on $(\mathcal P(G),\mathcal W)$}\label{eq_assp}
Consider the HJE on the Wasserstein space on graphs:
\begin{align}\label{HJeq}
\partial_tu(t,\xi)+\mathcal H(\xi,\nabla_{\mathcal W}u(t,\xi))+\mathcal F(\xi)=0,\quad u(0,\xi)=\mathcal U_0(\xi),
\end{align}
where $t\in(0,T)$ for some $T>0,$ $\xi\in\mathcal P^{\circ}(G),$ and the initial datum  $\mathcal U_0$ is continuous on $(\mathcal P(G),\mathcal W).$ In the following, we specify the assumptions on the Hamiltonian $\mathcal H.$

\begin{assumption}\label{ass_H}
Fix a constant $\kappa>1$ and assume there exist a constant $t_{*}>1$ and nonnegative functions $\gamma,\bar{\gamma}\in\mathcal C([0,\infty))$ such that for any $\xi,\eta\in\mathcal P^{\circ}(G)$ and $p\in\mathbb S^{d\times d},$ the following hold: 
\begin{itemize} 
\item[(H-\romannumeral1)] $\mathcal H\in\mathcal C(\mathcal P^{\circ}(G)\times \mathbb S^{d\times d})$ and $\mathcal H(\xi,\cdot)$ is convex. 
\item[(H-\romannumeral2)] $\lim_{t\to1^+}\bar{\gamma}(t)=1,\;\gamma(t)>1$ for $t\in(1,t_*),$ and $t\gamma(t)\mathcal H(\xi,p)\leq \mathcal H(\xi,tp)\leq \bar{\gamma}(t)\mathcal H(\xi,p)$ for all $t>0.$ 
\item[(H-\romannumeral3)] For every $\epsilon\in(0,1),$ there exists $\theta_{\epsilon}>0$ such that $\theta_{\epsilon}\|p\|^{\kappa}_{\xi}\leq \mathcal H(\xi,p)$ for all $\xi\in\mathcal P_{\epsilon}.$ 
\item[(H-\romannumeral4)] Let $\mathcal H(\xi,0)=0.$ For fixed $\epsilon\in(0,1),$ there exist a modulus $\mathfrak m_{\epsilon}$ (where $\mathfrak m_{\epsilon}(r)\leq C({\epsilon})r$) and a constant $C_{\epsilon}>0$ such that 
\begin{align*}
\mathcal H(\xi,p)-\mathcal H(\eta,p)\ge -\mathfrak m_{\epsilon}(\|\xi-\eta\|_{l^2})\|p\|^{\kappa}_{\xi}-C_{\epsilon}\big|\|p\|_{\xi}-\|p\|_{\eta}\big|\big(\|p\|^{\kappa-1}_{\xi}+\|p\|^{\kappa-1}_{\eta}\big).
\end{align*}
\item[(H-\romannumeral5)] Denote $\mathcal I(\xi):=\sum_{i=1}^d\frac{1}{\xi_i}.$ There exists $C_H>0$ such that $|\mathcal H(\xi,p)|\leq C_H\|p\|^{\kappa}_{\xi}\mathcal I^{-\kappa}(\xi)$. 
\end{itemize}
\end{assumption}
We now present an example of a Hamiltonian satisfying (H-\romannumeral1)--(H-\romannumeral5).

\begin{example}\label{exam1} 
Let $\mathcal{H}(\xi,p):=\mathfrak{a}(\xi)\|p\|^{\kappa}_{\xi}$ for $\xi\in\mathcal{P}^{\circ}(G)$ and $p\in\mathbb{S}^{d\times d}$, where $\mathfrak{a}(\xi)=\mathcal{I}^{-\kappa}(\xi)$ with $\kappa>1$. In this case, the parameters in (H-iii) and (H-iv) are given by $\theta_{\epsilon}=C\epsilon^{\kappa}d^{-\kappa}$ and $C_{\epsilon}=C\kappa d^{-\kappa}$. Furthermore, the modulus {in (H-iv)} if of the form $\mathfrak{m}_{\epsilon}(r)=C_{\epsilon}r$ for $r\ge 0$, as a consequence of the $l^2$-Lipschitz continuity of $\mathcal{I}^{-\kappa}$.  We refer the reader to \cite[Example 5.1]{MCC} for more details.
\end{example}

Under Assumption \ref{ass_H}, the HJE \eqref{HJeq} is well posed in the viscosity sense and admits a unique bounded  Lipschitz continuous solution.

\begin{proposition}\label{prop_exact}
\cite[Proposition 6.4]{MCC} Let Assumptions \ref{assumption_g} and \ref{ass_H} hold. Suppose in addition that $\mathcal U_0$ and $\mathcal F$ are $l^2$-Lipschitz continuous. Then there exists a unique bounded continuous viscosity solution $u$ to \eqref{HJeq} on $[0,T)\times \mathcal P^{\circ}(G)$ such that:  
\begin{itemize}
\item[(\romannumeral1)] There exists $L_1>0$ such that $|u(t,\xi)-u(r,\xi)|\leq L_1|t-r|$ for all $\xi\in\mathcal P^{\circ}(G), t,r\in[0,T);$
\item[(\romannumeral2)] For every $\epsilon\in(0,\frac{1}{d}),$ there exists $L_2:=L_2(\epsilon)>0$ such that $|u(t,\xi)-u(t,\eta)|\leq L_2\|\xi-\eta\|_{l^2}$ for all $t\in[0,T),\xi,\eta\in\mathcal P_{\epsilon}.$
\end{itemize}
\end{proposition}

\section{Main Results on the Numerical Scheme}\label{sec_4.1}
In this section, we introduce the finite difference operators, the discrete Hamiltonian function, and the weighted function spaces.  We then present the semi-discrete finite difference scheme together with  the associated   weighted adjoint equation. Finally, we state our main result on the first-order convergence of the numerical solution in the weighted 
 $L^1$
  space.

\subsection{Finite difference operators} 
For $\xi \in \mathcal{P}(G)$, we define the coordinate transformation 
$\Pi \colon \mathcal{P}(G) \to \widetilde{\mathcal{P}}(G)$ via the cumulative 
probability sums:
\[
  s^0 = 0, \quad s^k = \sum_{i=1}^k \xi_i \; (\text{for }k = 1,\ldots,d-1), 
  \quad s^d = \sum_{i=1}^d \xi_i = 1.
\]
 The image of this transformation is the Euclidean set 
\[
  \widetilde{\mathcal{P}}(G) 
  := \bigl\{ s = (s^1,\ldots,s^{d-1}) \in \mathbb{R}^{d-1} 
     : 0 = s^0 \leq s^1 \leq \cdots \leq s^{d-1} \leq s^d = 1 \bigr\}.
\] 
The inverse transformation recovers 
$\xi$ from $s$ is given by $\xi_k = s^k - s^{k-1}$ for 
$k = 1, \ldots, d$. For notational brevity, we write $\mathcal  P:=\mathcal P(G)$ and $\widetilde{\mathcal P} := \widetilde{\mathcal P}(G)$
 throughout this paper. 
Since $\Pi$ is affine with Jacobian equal 
to $1$, it is a measure-preserving bijection between $\mathcal P$ and $\widetilde{\mathcal P}$.  
Consequently, for any  function $f:\mathcal P\to\mathbb R$, integration with respect  to the $(d-1)$-dimensional Hausdorff measure 
$\mathscr{H}^{d-1}$ on $\mathcal{P}$  becomes as standard Lebesgue 
integrals over $\widetilde{\mathcal{P}}$: 
\begin{align}\label{hausdorff}
  \int_{\mathcal{P}} f(\xi)\,\mathrm{d}\mathscr H^{d-1}(\xi) 
  = \int_{\widetilde{\mathcal{P}}} f\bigl(\Pi^{-1}(s)\bigr)\,\mathrm{d}s, 
\end{align}  and all subsequent integrals over $\mathcal{P}$ are to be understood in this sense. Throughout the paper, we use $\xi \in \mathcal P$
to denote points in graph coordinates and and $x \in \widetilde{\mathcal P}$  to denote points in Euclidean coordinates.
Through the map $\Pi$, the function spaces on $\mathcal P$ and $\widetilde{\mathcal P}$ considered in this paper (e.g. $L^p(\mathcal {P}),\mathcal C(\mathcal P)$ and $L^p(\widetilde{\mathcal {P}}),\mathcal C(\widetilde{\mathcal P})$ with $p\in[1,+\infty]$) are naturally identified. We therefore use these spaces interchangeably, depending on the coordinate system being employed.

The transformation $\Pi$ 
connects Wasserstein directional derivatives on the graph
with standard partial derivatives in Euclidean coordinates; see \cite[Lemma 4.1]{CDM25}.

\begin{lemma}\label{lemma1}
Let $f\in\mathcal C^1(\mathcal P^{\circ}(G))$ and define $\tilde f(x)=f(\Pi^{-1}(x))$ for $x=\Pi(\xi)$, {$\xi \in \mathcal P^{\circ}(G)$}. Then the Wasserstein gradient $\nabla_{\mathcal W}f(\xi)$ can be expressed via the partial derivatives of $\tilde f$ at $x = \Pi(\xi)$ as follows:
\begin{align*}
\sqrt{\omega_{j,k}}\nabla^{e_{j,k}}f(\xi)=\sqrt{\omega_{j,k}}(\partial_{x_j}+\cdots+\partial_{x_{k-1}})\tilde f(\Pi(\xi)),\quad 1\leq j<k\leq d,\;\;\xi\in\mathcal P^{\circ}(G).
\end{align*}
\end{lemma}

Let $h\in(0,\frac{1}{d})$ be the mesh size. 
Denote the interior area and the boundary layer of thickness $h$ along the direction $e_{j,k}$ with $(j,k)\in E$, respectively, as 
\begin{align}\label{interior}
\mathcal P_{e_{j,k}}^{\rm Int}:=\{\xi\in\mathcal P :\xi\pm he_{j,k}\in\mathcal P \},\text{ and }
  \mathcal{B}_{h,e_{j,k}} := \mathcal{P} \setminus \mathcal{P}_{e_{j,k}}^{\rm Int}.
  \end{align} 
  Define the forward and backward differences respectively as \begin{align}\label{D_notation}
D^{+}_{e_{j,k}} u^h(\xi):=u^h(\xi+he_{j,k})-u^h(\xi),\quad D^{-}_{e_{j,k}} u^h(\xi):=u^h(\xi)-u^h(\xi-he_{j,k})
\end{align}
when $\xi\in \mathcal P^{\rm Int}_{e_{j,k}}
.$ 
For grid points 
$\xi \in \mathcal{P}$ that lie on the boundary in the 
$e_{j,k}$ direction
(that is, those $\xi$ such that $\xi \pm h e_{j,k} \notin \mathcal{P}$, i.e., $\xi \in \mathcal B_{h,e_{j,k}}$), we employ 
the constant extrapolation. Concretely, we set $u^h(\xi \pm he_{j,k}) := u^h(\xi)$ which implies $D^\pm_{e_{j,k}} u^h(\xi) = 0$ at the outflow boundaries. 
To  approximate the Wasserstein gradient, we define $d\times d$ skew-symmetric difference matrices 
 on ${\mathcal P}$: 
\begin{align}\label{differ_matrix}
&[D^{\pm} u^h]=(0,b_{1,2},b_{1,3},\ldots,b_{1,d};-b_{1,2},0,b_{2,3},\ldots,b_{2,d};\ldots;-b_{1,d},\ldots,-b_{d-1,d},0)\notag\\
&\text{ with entries }
 b_{j,k}=\frac{\sqrt{\omega_{j,k}}}{h}D^{\pm}_{e_{j,k}} u^h \text{ for } 1\leq j<k\leq d.
 \end{align}  
 
 In order to distinguish coordinates between the two mesh spaces, we 
use $\tilde u(t,\cdot)$ with initial value 
$\widetilde{\mathcal{U}}_0 := \mathcal{U}_0\circ\Pi^{-1}$ to denote 
the exact viscosity solution on $\widetilde{\mathcal{P}}$, 
which satisfies $\tilde u(t,x) = u(t,\Pi^{-1}(x)) = u(t,\xi)$ 
with $x = \Pi(\xi)$. The numerical approximations of $u$ 
(resp.\ $\tilde u$) are denoted by $u^h$ (resp.\ $\tilde u^h$), 
respectively.
 For notational brevity, we write $\tilde u^h(x)$ and $u^h(\xi)$ 
when the time variable is not emphasized.

Finite difference operators in the coordinates of 
$\widetilde{\mathcal{P}}$ are defined as follows. For fixed indices 
$j,k$ with $1\leq j < k \leq d$, we define the multi-index 
$\vec{m}_{j,k} := (m_1,\ldots,m_{d-1})$ by setting $m_l = 1$ for 
$l \in \{j,\ldots,k-1\}$ and $m_l = 0$ otherwise.
By virtue of the identity 
$\Pi^{-1}(x) \pm he_{j,k} = \Pi^{-1}(x \pm h\vec{m}_{j,k})$, 
one has 
\begin{align}\label{def_DtildeU}
D^+_{e_{j,k}} u^h(x)&=\tilde u^h(x+h\vec m_{j,k})-\tilde u^h(x)=:D^+_{\vec m_{j,k}}\tilde u^h(x),\notag\\
D^-_{e_{j,k}} u^h(x)&=\tilde u^h(x)-\tilde u^h(x-h\vec m_{j,k} )=:D^-_{\vec m_{j,k} }\tilde u^h(x).
\end{align}
Then the equivalent difference matrices of \eqref{differ_matrix} in the coordinate of $\widetilde{\mathcal P}$ are denoted by $[D^\pm\tilde u^h],$ whose entries are given by $b_{j,k}=\frac{\sqrt{\omega_{j,k}}}{h}D^{\pm}_{\vec m_{j,k}}\tilde u^h$ for $1\leq j<k\leq d.$

\subsection{Discrete Hamiltonian function and weighted space}  Denote the skew-symmetric matrices \begin{align*}&P:=(0,p_{1,2},p_{1,3},\ldots,p_{1,d};-p_{1,2},0,p_{2,3},\ldots,p_{2,d};\ldots;-p_{1,d},\ldots,-p_{d-1,d},0),\\
 &Q:=(0,q_{1,2},q_{1,3},\ldots,q_{1,d};-q_{1,2},0,q_{2,3},\ldots,q_{2,d};\ldots;-q_{1,d},\ldots,-q_{d-1,d},0).
 \end{align*}  
We rewrite the Hamiltonian as $\mathcal H(\xi,P)=\mathcal H(\xi,p_{1,2},p_{1,3},\ldots,p_{d-1,d}):\mathcal P^{\circ}(G)\times \mathbb R^{(d^2-d)/2}\to\mathbb R.$
We define the discrete Hamiltonian $\mathcal G(\xi,P,Q)$ acting on the skew-symmetric matrices $P,Q$:  
 $$\mathcal G(\xi,P,Q)=\mathcal G(\xi,p_{1,2},q_{1,2};p_{1,3},q_{1,3};\ldots;p_{d-1,d},q_{d-1,d}):\mathcal P^{\circ}(G)\times \mathbb R^{d^2-d}\to\mathbb R.$$ 
When $\mathcal G$ is evaluated at the pair of difference matrices in \eqref{differ_matrix}, we write, for brevity, 
$\mathcal G(\xi,[D^\pm u^h]):=\mathcal G(\xi,[D^+u^h],[D^-u^h]).$
We impose the following assumption on $\mathcal G$.

\begin{assumption}\label{ass_H2} There exists $R_0>0$ such that 
\begin{itemize}
\item[(\romannumeral1)] (Monotonicity) For each $R\in(0, R_0]$, $\mathcal G$ is non-increasing in $p_{k,l}$ and non-decreasing in $q_{k,l},1\leq k<l\leq d$, whenever  $\|P\|_{\infty}\vee \|Q\|_{\infty}\leq R.$
\item[(\romannumeral2)] (Consistency) $\mathcal G(\xi,P,P)=\mathcal H(\xi,P),\;\xi\in\mathcal P^{\circ}(G).$
\item[(\romannumeral3)] (Local Lipschitz property) For  each  $R\in (0,R_0],$ when $\|P\|_{l^2}\vee\|Q\|_{l^2}\vee\|\bar P\|_{l^2}\vee\|\bar Q\|_{l^2}\leq R$, we have \begin{align}\label{locally}|\mathcal G(\xi,P,Q)-\mathcal G(\xi,\bar P,\bar Q)|\leq C_{R}(\|P-\bar P\|_{l^2}+\|Q-\bar Q\|_{l^2}),\quad \xi\in \mathcal P^{\circ}(G)\end{align}for some $C_{R}>0$ independent of $\xi$.
\end{itemize}
\end{assumption} 

For the Hamiltonian $\mathcal{H}$ given by Example~\ref{exam1}, 
two important classes of numerical Hamiltonians satisfying Assumption
~\ref{ass_H2} are the Lax--Friedrichs 
 Hamiltonian and the Osher--Sethian Hamiltonian; see  Section \ref{exam:LF} for 
details.

To handle the geometric discrepancy between the domain $\mathcal{P}(G)$ 
and the shifted stencils $\xi \pm he_{i,j}$, which may fall outside  
$\mathcal{P}(G)$, we formulate the analysis  in the weighted Lebesgue space 
$L^p_w:=L^p_w(\widetilde{\mathcal{P}})$ for $1\le p \le 2$, equipped with the norm
\[
  \|v\|_{L^p_w} := \left(\int_{\widetilde{\mathcal{P}}} 
  |v(x)|^p\, w(x)\,\mathrm{d}x\right)^{1/p}.
\]
The weight function 
$w$ captures the boundary degeneracy of the Wasserstein simplex and plays a central role in the weighted discrete integration-by-parts (IBP)  formula established below. For $p \in [1,+\infty]$, let $L^p(\widetilde{\mathcal{P}})$ denote the 
Lebesgue space of $p$-integrable functions on $\widetilde{\mathcal{P}}$ 
with respect to the Lebesgue measure, with the convention that $L^\infty(\widetilde{\mathcal P})$ consists 
of essentially bounded measurable functions.

\begin{definition}\label{def_truncation}
A function $w: {\mathcal{P}} \to [0, \infty)$ is called an admissible weight function on the Wasserstein space on a graph if it satisfies:
\begin{enumerate}
\item[(i)] $w \in \mathcal{C}^2({\mathcal{P}})$ and $w(\xi) > 0$ for all $\xi \in {\mathcal{P}}^\circ$;
\item[(ii)] $w(\xi) = 0$ for all $\xi \in \partial {\mathcal{P}}$. 
\end{enumerate}
\end{definition}

The following examples satisfy Definition~\ref{def_truncation} and provide 
typical constructions of admissible  weight functions on $\mathcal{P}$.

\begin{example}
\begin{enumerate}
\item[(i)] Polynomial weight function:  
$
w(\xi) = \prod_{i=1}^d \xi_i ^\alpha, \; \alpha \geq 1, \; \xi \in \mathcal{P}^{\circ}.
$

\item[(ii)] Exponential weight function: $
w(\xi) = \exp( -\frac{\lambda}{\prod_{i=1}^d \xi_i } ), \; \lambda > 0,\;\xi \in \mathcal{P}^{\circ}.
$

\item[(iii)] Smooth mollifier weight function: $
w(\xi) = \prod_{i=1}^d \psi(\xi_i ),\xi \in \mathcal{P}^{\circ},$ where $\psi(s) = e^{-1/s}$ for $s>0$ and $\psi(s) = 0$ for $s \leq 0$.
\end{enumerate} 
\end{example}
For these examples, one has $\| w\|_{L^{\infty}(\widetilde{\mathcal P})}\vee \|\nabla w\|_{L^{\infty}(\widetilde{\mathcal P})}<\infty.$

 \subsection{Semi-discrete finite difference scheme and weighted adjoint equation}
 For each $h>0$, we introduce the following semi-discrete finite difference equation
on the domain $\mathcal{P}$:
\begin{align}
\label{spatial_eq}
\begin{cases}
\partial_t u^h(t,\xi)+\mathcal G(\xi, [D^\pm u^h])+\mathcal F(\xi)=0,
\quad t>0,\;\xi\in\mathcal P^{\circ}, \\
u^h(0,\xi)=\mathcal U_0(\xi),
\end{cases}
\end{align} where the difference matrices $[D^{\pm}u^h]$ are given in \eqref{differ_matrix}. 
For each $h>0$, $\xi_0\in\mathcal{P}^\circ$, and $T>0$, we define the adjoint 
variable $\sigma(t,\xi):=\sigma^{h,\xi_0,T}(t,\xi)$ associated with the  
scheme \eqref{spatial_eq} 
 as the solution to the following equation:
\begin{align}\label{adjoint_exact}
\begin{cases}
\partial_s \sigma(s,\xi)
+ \dfrac{1}{w(\xi)}
  \displaystyle\sum_{(i,j) \in E} \dfrac{\sqrt{\omega_{i,j}}}{h}
  \Big(
    D^-_{e_{i,j}} \big[ \sigma(s,\xi) \mathcal{A}_{i,j}(\xi) w(\xi) \big]
    + D^+_{e_{i,j}} \big[ \sigma(s,\xi) \mathcal{B}_{i,j}(\xi) w(\xi) \big]
  \Big) = 0, \\
\sigma(T, \xi) = \dfrac{\delta_{\xi_0}(\xi)}{w(\xi)},\quad \xi\in\mathcal P^{\circ},
\end{cases}
\end{align}
where $\mathcal{A}_{i,j}(\xi) = \partial_{p_{i,j}} \mathcal{G}(\xi, [D^\pm u^h]),$ 
$\mathcal{B}_{i,j}(\xi) = \partial_{q_{i,j}} \mathcal{G}(\xi, [D^\pm u^h])$, and $\delta_{\xi_0}$ denotes the Dirac delta measure concentrated at 
$\xi_0 \in \mathcal{P}^\circ$, so that the terminal condition 
$\sigma(T,\xi) = \delta_{\xi_0}(\xi)/w(\xi)$ is understood in the distributional sense, 
with $\int_{\widetilde{\mathcal{P}}} \sigma(T,x) w(x)\,\mathrm{d}x = 1$. 
To make the structure of \eqref{adjoint_exact} more transparent, we expand the
discrete product rule identities
\[
  D^-_{e_{i,j}} \big( \sigma(s,\xi) \mathcal{A}_{i,j}(\xi) w(\xi) \big)
  = w(\xi)\, D^-_{e_{i,j}} \big( \sigma(s,\xi) \mathcal{A}_{i,j}(\xi) \big)
    + \sigma(s,\xi-he_{i,j})\, \mathcal{A}_{i,j}(\xi-he_{i,j})\, D^-_{e_{i,j}} w(\xi),
\]
\[
  D^+_{e_{i,j}} \big( \sigma(s,\xi) \mathcal{B}_{i,j}(\xi) w(\xi) \big)
  = w(\xi)\, D^+_{e_{i,j}} \big( \sigma(s,\xi) \mathcal{B}_{i,j}(\xi) \big)
    + \sigma(s,\xi+he_{i,j})\, \mathcal{B}_{i,j}(\xi+he_{i,j})\, D^+_{e_{i,j}} w(\xi),
\]
and substitute these expressions into \eqref{adjoint_exact}. Dividing through by $w(x)$, we arrive at the following equivalent decomposed form: 
\begin{align}\label{adjoint1}\partial_t &\sigma^{h,\xi_0,T} + \sum_{(i,j)\in E} \frac{\sqrt{\omega_{i,j}}}{h} \Big( D^-_{e_{i,j}} \big(\sigma^{h,\xi_0,T} \partial_{p_{i,j}}\mathcal G\big) + D^+_{e_{i,j}} \big(\sigma^{h,\xi_0,T} \partial_{q_{i,j}}\mathcal G\big) \Big) \notag\\&+ {\mathcal{S}(\sigma^{h,\xi_0,T}, \mathcal{G},w)} = 0,\quad \text{\big(written as }\partial_t\sigma + \mathcal{L}_h^*\sigma = 0\big)
\end{align}
subject to the terminal condition $\sigma^{h,\xi_0,T}(T,\xi) =\frac{ \delta_{\xi_0}(\xi)} { w(\xi)}$, where $\mathcal{G}$ is evaluated at $(\xi,[D^\pm u^h])$.
 The summation term in \eqref{adjoint1} represents the discrete divergence, 
which is the formal adjoint of the upwind discretization of $\mathcal{G}$. When $w \equiv 1$, as in a flat domain with Dirichlet or periodic boundary conditions, this term coincides with the full adjoint operator and 
$\mathcal{S}$ vanishes. The term $\mathcal{S}$, which we call the \emph{geometric 
drift}, is a correction that arises from the non-uniformity of the 
weight $w$ and the Wasserstein simplex. It is given by 
\begin{align}\label{drift_term}
&\quad\mathcal{S}(\sigma^{h,\xi_0,T} ,\mathcal{G},w)\notag
\\
& := \sum_{(i,j)\in E} \frac{\sqrt{\omega_{i,j}}}{h}
   \Big(
     \frac{D^+_{e_{i,j}} w}{w}
     \big(\sigma^{h,\xi_0,T} \partial_{p_{i,j}}\mathcal{G}\big)(\xi+he_{i,j})
     + \frac{D^-_{e_{i,j}} w}{w}
     \big(\sigma^{h,\xi_0,T} \partial_{q_{i,j}}\mathcal{G}\big)(\xi-he_{i,j})
   \Big). 
\end{align} 

 Here, to handle  the shifted grid points $\xi \pm h e_{i,j}$ lying outside $\mathcal{P}$, we extend
the Hamiltonian $\mathcal{G}$ to the complement $\mathcal{P}^{c}$ by setting
\begin{align}\label{zero_G}
  \mathcal{G}(\xi) = 0,
  \quad \forall\, \xi \notin \mathcal{P}.
\end{align}
This extension is both natural and compatible with the Wasserstein graph structure. Indeed, the Hamiltonian 
$H$ is degenerate at the boundary, reflecting the fact that no probability flux crosses
 the boundary. As we show below, this extension, together with the vanishing of the weight
function $w$ on $\partial\mathcal{P}$, guarantees that all boundary
terms in the discrete  integration-by-parts (IBP) formula vanish.

\begin{lemma}\label{remark_IBP}
 For any edge $(i,j)\in E$, test function $\varphi \in \mathcal{C}(\widetilde{\mathcal{P}})$, 
and fluxes $\mathcal{A} = \sigma \partial_{p_{i,j}}\mathcal{G}$,
$\mathcal{B}= \sigma \partial_{q_{i,j}}\mathcal{G}$ with $\sigma \in \mathcal{C}(\widetilde{\mathcal{P}})$, the following weighted
IBP identity holds: 
\begin{align}\label{IBP_identity}
&\int_{\widetilde{\mathcal{P}}} \varphi(x)
  \bigl( D^-_{e_{i,j}} \mathcal{A}(x) + D^+_{e_{i,j}} \mathcal{B}(x) \bigr)
  w(x)\,\mathrm{d}x \notag \\
= \,&- \int_{\widetilde{\mathcal{P}}}
  \bigl( \mathcal{A}(x)\, D^+_{e_{i,j}} \varphi(x)
       + \mathcal{B}(x)\, D^-_{e_{i,j}} \varphi(x) \bigr)
  w(x)\,\mathrm{d}x \notag\\
&- \int_{\widetilde{\mathcal{P}}}
  \mathcal{A}(x)\, \varphi(x+h\vec{m}_{i,j})\, D^+_{e_{i,j}} w(x)\,\mathrm{d}x \notag \\
&- \int_{\widetilde{\mathcal{P}}}
  \mathcal{B}(x)\, \varphi(x-h\vec{m}_{i,j})\, D^-_{e_{i,j}} w(x)\,\mathrm{d}x.
\end{align}
\end{lemma}
\begin{proof} 
We establish the identity for the term involving $D^-_{e_{i,j}}\mathcal{A}$; the
corresponding identity for $D^+_{e_{i,j}}\mathcal{B}$ follows by an
analogous argument. Using the change of variables $y = x - h\vec{m}_{i,j}$, we write
\begin{align*}
&\int_{\widetilde{\mathcal{P}}}
  \varphi(x)\,w(x)\,D^-_{e_{i,j}}\mathcal{A}(x)\,\mathrm{d}x
= \int_{\widetilde{\mathcal{P}}}
  \varphi(x)\,w(x)\,\mathcal{A}(x)\,\mathrm{d}x
- \int_{\widetilde{\mathcal{P}}-h\vec{m}_{i,j}}
  (\varphi w)(y+h\vec{m}_{i,j})\,\mathcal{A}(y)\,\mathrm{d}y \\
&= -\int_{\widetilde{\mathcal{P}}}
  \mathcal{A}(x)\,D^+_{e_{i,j}}(\varphi w)(x)\,\mathrm{d}x
+ \int_{\widetilde{\mathcal{P}}
        \setminus (\widetilde{\mathcal{P}}-h\vec{m}_{i,j})}
  \mathcal{A}(x)\,(\varphi w)(x+h\vec{m}_{i,j})\,\mathrm{d}x \\
&\quad
- \int_{(\widetilde{\mathcal{P}}-h\vec{m}_{i,j})
        \setminus \widetilde{\mathcal{P}}}
  \mathcal{A}(x)\,(\varphi w)(x+h\vec{m}_{i,j})\,\mathrm{d}x.
\end{align*}
Here, the last boundary integral is zero because
$\partial_{p_{i,j}}\mathcal{G}(\xi)=0$ for all 
$\xi\notin\mathcal{P}$ 
by the zero extension~\eqref{zero_G}. 
The preceding boundary integral is also  zero because
$w(\xi)= 0$ for $\xi\notin\mathcal{P}$ (see Definition~\ref{def_truncation}), 
and hence
\[
  \int_{\widetilde{\mathcal{P}}}
  \varphi(x)\,w(x)\,D^-_{e_{i,j}}\mathcal{A}(x)\,\mathrm{d}x
  = -\int_{\widetilde{\mathcal{P}}}
  \mathcal{A}(x)\,D^+_{e_{i,j}}(\varphi w)(x)\,\mathrm{d}x.
\]
Applying the discrete product rule $D^+_{e_{i,j}}(\varphi w) = w\,D^+_{e_{i,j}}\varphi 
+ \varphi (\cdot+h\vec{m}_{i,j})\,D^+_{e_{i,j}}w$ yields
\[
  \int_{\widetilde{\mathcal{P}}}
  \varphi(x)\,w(x)\,D^-_{e_{i,j}}\mathcal{A}(x)\,\mathrm{d}x
  = -\int_{\widetilde{\mathcal{P}}}
  \mathcal{A}(x)\,D^+_{e_{i,j}}\varphi(x)\,w(x)\,\mathrm{d}x
  - \int_{\widetilde{\mathcal{P}}}
  \mathcal{A}(x)\,\varphi(x+h\vec{m}_{i,j})\,D^+_{e_{i,j}}w(x)\,\mathrm{d}x.
\]
Combining this with the analogous identity for the $D^+_{e_{i,j}}\mathcal{B}$
term completes the proof.
\end{proof}

The identity \eqref{IBP_identity} is the cornerstone of our convergence
analysis. It allows discrete derivatives to be transferred from the numerical flux to the test function $\varphi$, and it plays a central role in establishing the $L^1$
convergence rate.

\subsection{Main results}
Let the Hamiltonian $\mathcal{H}$ be defined as a sum of local contributions 
associated with each edge, reflecting the fact that the evolution of the 
probability measure $\xi$ is driven by local mass transfer between adjacent 
vertices. Specifically, we assume that the continuous Hamiltonian 
$\mathcal{H} \colon \mathcal{P}\times \mathbb{R}^d \to \mathbb{R}$ 
admits the edge-wise decomposition
\begin{equation}\label{eq:H_sum}
  \mathcal{H}(\xi, P) 
  = \sum_{(i,j) \in E} \mathcal{H}_{i,j}(\xi, P),
  \quad \xi \in \mathcal{P},\; P \in \mathbb{S}^{d\times d},
\end{equation}
where $\mathcal{H}_{i,j} \colon \mathcal{P}_\epsilon(G) \times 
\mathbb{S}^{d\times d} \to \mathbb{R}$ is the local Hamiltonian associated 
with the edge $(i,j)$. As a concrete example, in the setting of 
Example~\ref{exam1} with $\kappa = 2$, the local Hamiltonian takes the form 
$\mathcal{H}_{i,j}(\xi, P) = \frac{1}{2}\mathcal{I}^{-2}(\xi)\,
g_{i,j}(\xi)\,p^2_{i,j}$ for each $(i,j) \in E$. We next impose regularity and structural assumptions on the numerical Hamiltonian.

\begin{assumption}\label{ass_G}  Suppose that the discrete Hamiltonian $\mathcal{G}(\xi, P, Q)$ is twice 
continuously differentiable in $\xi \in \mathcal{P}$ for each fixed 
$(P, Q) \in \mathbb{R}^{\frac{d^2-d}{2}} \times 
\mathbb{R}^{\frac{d^2-d}{2}}$, and twice differentiable in 
$(P, Q)$ for almost every 
$(P, Q) \in \mathbb{R}^{\frac{d^2-d}{2}} \times 
\mathbb{R}^{\frac{d^2-d}{2}}$, for each fixed $\xi \in \mathcal{P}$.  
Let $\mathcal G$  
take the form of 
\begin{align}\label{H_special1}
  \mathcal{G}(\xi, P, Q) 
  = \sum_{(i,j)\in E} \mathcal{G}_{i,j}\big(\xi,P,Q\big),
  \quad \xi \in \mathcal{P}(G).
\end{align}  
Fix $R_0>0$ as in Assumption ~\ref{ass_H2}, and impose the following conditions.

\begin{enumerate}
\item[(i)] (Growth conditions) For each $R\in(0,R_0]$, there exists $C:=C(R)>0$ such that for all $\xi \in \mathcal{P}$ and $P,Q$ with $\|P\|_{\infty}\vee \|Q\|_{\infty} \leq R$,\begin{align}
 &\sup_{(i,j)\in E}\sup_{\xi\in\mathcal P}|\partial_{p_{i,j}}\mathcal G_{i,j}|\leq C,\label{growth2}\\
 &\sup_{i=1,\ldots,d}\sup_{(k,l)\in E}\Big(\Big|\frac{\partial^2\mathcal{G}}{\partial\xi_i\partial p_{k,l}}\Big| + \Big|\frac{\partial^2\mathcal{G}}{\partial\xi_i\partial q_{k,l}}\Big|\Big) 
+ \sup_{i,j=1,\ldots,d} \Big|\frac{\partial^2\mathcal{G}}{\partial\xi_i\partial\xi_j}\Big| \leq C\label{growth3}.\end{align}
\item[(ii)] (Bounded non-negative second derivatives)
For each $R\in(0,R_0]$, there exists $C:=C(R)>0$ such that for all edge $(i,j) \in E$, 
 $\xi \in \mathcal{P}$ and $P,Q$ with $\|P\|_{\infty}\vee \|Q\|_{\infty} \leq R$, 
\begin{align}\label{Gbounded}
0\leq \partial^2_{p_{i,j}p_{i,j}}\mathcal{G},\,   \partial^2_{p_{i,j}q_{i,j}}\mathcal{G},\,\partial^2_{q_{i,j}q_{i,j}}\mathcal{G}\leq Cg_{i,j}(\xi).
\end{align} 
\end{enumerate}
\end{assumption}

Under the above structural assumptions on the numerical Hamiltonian, 
we now address the well-posedness and regularity of the semi-discrete scheme \eqref{spatial_eq}, stated in the following proposition. The proof is postponed to Appendix~\ref{app1}.
Define the nested sublevel sets: for $(i,j)\in E,$
\begin{align}\label{nested}
  \mathcal{P}_{h,e_{i,j}}
  &:= \bigl\{\xi\in\mathcal{P}:
       \xi\pm he_{i,j}\in\mathcal{P}\bigr\},\notag\\
  \mathcal{P}_{2h,e_{i,j}}
  &:= \bigl\{\xi\in\mathcal{P}:
       \xi\pm he_{i,j}\in\mathcal{P}_h\bigr\},
\end{align} and similarly for $\mathcal P_{3h,e_{i,j}}.$ Clearly, $\mathcal{P}_{2h,e_{i,j}}\subset\mathcal{P}_{h,e_{i,j}}\subset\mathcal{P}$. Morover, 
the boundaries $\partial\mathcal{P}_{h,e_{i,j}}$, $\partial\mathcal{P}_{2h,e_{i,j}}$
are piecewise-linear hypersurfaces of measure zero in the interior  
$\mathcal{P}^\circ$.
Define \begin{align*}
  P\mathcal C^0_{e_{i,j}}
  := \Bigl\{f:\mathcal{P}\to\mathbb{R}\;\Big|\;&
     f\in\mathcal{C}(\mathcal{P}_{2h,e_{i,j}}),\;
     f\in\mathcal{C}(\mathcal{P}_{h,e_{i,j}}\setminus\mathcal{P}_{2h,e_{i,j}}),\;
     f\in\mathcal{C}(\mathcal{P}\setminus\mathcal{P}_{h,e_{i,j}}),\\
     &
     \text{with only finite jumps across }
     \partial\mathcal{P}_{h,e_{i,j}}\cup\partial\mathcal{P}_{2h,e_{i,j}}
  \Bigr\}.
\end{align*}

\begin{proposition}\label{prop_regularity}
Let Assumptions~\ref{assumption_g}-\ref{ass_G} hold,
let $h \in (0, \frac1d)$, and let $\mathcal{U}_0, \mathcal{F} \in
\mathcal{C}^2(\mathcal{P})$. Then there exists a unique bounded solution
$u^h$ to \eqref{spatial_eq}. 
Moreover, $u^h\in\mathcal{C}\bigl((0,T)\times\mathcal{P}\bigr)$, $\nabla^{e_{i,j}}u^h,e_{i,j}^{\top}\nabla^2u^he_{i,j}\in P\mathcal C^0_{e_{i,j}}$ for $(i,j)\in E$, 
 and
\begin{align*}
  \nabla^{e_{i,j}} u^h 
  &\in \mathcal C\bigl((0,T)\times\mathcal{P}_{h,e_{i,j}}\bigr)
     \cap \mathcal C\bigl((0,T)\times(\mathcal{P}\setminus\mathcal{P}_{h,e_{i,j}})\bigr),\\
  e^{\top}_{i,j}\nabla^2 u^h e_{i,j}
  &\in \mathcal C\bigl((0,T)\times\mathcal{P}_{2h,e_{i,j}}\bigr)
     \cap \mathcal C\bigl((0,T)\times(\mathcal{P}_{h,e_{i,j}}\setminus\mathcal{P}_{2h,e_{i,j}})\bigr)
     \cap \mathcal C\bigl((0,T)\times(\mathcal{P}\setminus\mathcal{P}_{h,e_{i,j}})\bigr).
\end{align*}
\end{proposition}
\begin{remark}
The $\mathcal{C}^2$-regularity of $\mathcal{U}_0$ and $\mathcal{F}$ 
in Proposition~\ref{prop_regularity} in fact can be weakened, without changing the conclusion, to 
\begin{align*}
  \mathcal{U}_0,\,\mathcal{F}\in \bigl\{f\in W^{2,\infty}(\mathcal{P}) \,:\, 
  \nabla^{e_{i,j}}f,\, e^{\top}_{i,j}\nabla^2 f\, e_{i,j} 
  \in P\mathcal{C}^0_{e_{i,j}} \text{ for all } (i,j)\in E\bigr\},
\end{align*} 
as is clear from the proof of Proposition~\ref{prop_regularity}. Here $W^{2,\infty}(\mathcal P)$ denotes the usual Sobolev space with index $(2,\infty)$ on $\mathcal P.$ We nevertheless retain the assumption $\mathcal{U}_0,\mathcal{F}\in 
\mathcal{C}^2(\mathcal{P})$ for the following reasons. We want to emphasize that the piecewise nature of $\nabla_{\xi} u^h$ and $\nabla^2_{\xi} u^h$ at the artificial 
interfaces $\partial\mathcal{P}_{h,e_{i,j}}$ and 
$\partial\mathcal{P}_{2h,e_{i,j}}$ originates from the 
extrapolation convention \eqref{eq:const_ext}, rather than by any lack of regularity of 
$\mathcal{U}_0$ or $\mathcal F$. Moreover, the
assumption $\mathcal{U}_0,\,\mathcal{F}\in\mathcal{C}^2(\mathcal{P})$  is more transparent and avoids introducing  $h$-dependent interfaces into the formulation of the original HJE \eqref{HJeq}.
\end{remark}

We are now in a position to state the main convergence result. We obtain a 
first-order error estimate for the proposed scheme,  under the assumption that the numerical solution satisfies both the gradient bound and the semi-concavity bound in Assumption~\ref{ass:semi-concave}. 
Our analysis is carried out in the weighted space $L^1_w=L^1_w(\widetilde{\mathcal{P}})$. 
 The proof of Theorem \ref{thm1} is presented in Section~\ref{sec_5}.

\begin{assumption}\label{ass:semi-concave}
There exists  constants $h_0\in(0,\frac{1}{d})$ and $C>0$ such that, for any $h\in (0,h_0),$
\begin{enumerate}
 \item[\textup{(i)}] $\displaystyle
    \sup_{t\in[0,T],\xi\in\mathcal{P}}
    \max_{(i,j)\in E}|\nabla^{e_{i,j}}u^h(t,\xi)| \leq C;$ 
  \item[\textup{(ii)}] $\displaystyle
    \sup_{t\in[0,T]}\sup_{\mathbf{a}\in\mathbb{V}}
    \sup_{\xi\in\mathcal{P}}
    \mathbf{a}^\top\nabla^2 u^h(t,\xi)\,\mathbf{a} \leq C,$ where $\mathbb V$ is the set of unit tangent vectors in  \eqref{mathbbV}.  
    \end{enumerate}
\end{assumption}

\begin{theorem}\label{thm1}
Let Assumptions~\ref{assumption_g}-\ref{ass:semi-concave} hold, and let $\mathcal{U}_0, \mathcal{F} 
\in \mathcal{C}^2(\mathcal{P})$. For any $T>0,$ there exists a constant 
$C := C(\|\mathcal U_0\|_{\mathcal C^2(\mathcal P)},\|\mathcal F\|_{\mathcal C^2(\mathcal P)},T,G,g) > 0$, such that for all $h \in (0,h_0)$ with $h_0$ given in Assumption \ref{ass:semi-concave},
\begin{align*}
  \int_{\widetilde{\mathcal{P}}}
  \bigl|\tilde{u}^h(T,x) - \tilde{u}(T,x)\bigr|\,w(x)\,\mathrm{d}x 
  \leq C h.
\end{align*}
\end{theorem}
\begin{remark}
(i)~The first-order rate in Theorem~\ref{thm1} is sharp for this class of monotone schemes. Improving to higher order would require additional regularity assumptions on Hamiltonians and the use of high-order reconstructions such as ENO and WENO \cite{WENO}. We leave this for future work.  

(ii)~Theorem~\ref{thm1} concerns  only the spatial semi-discrete error.  
For fully discrete schemes with time step $\Delta t = \mathcal O(h)$ satisfying a CFL-type condition \cite{CDM25}, one expects the temporal discretization error to be of the same order.
\end{remark}

\section{$L^1$-convergence Analysis via Adjoint Method}\label{sec_5}
In this section, we prove Theorem~\ref{thm1} using the weighted adjoint 
method. The key idea is to represent the error $\tilde{u}^h - \tilde{u}$ as 
a duality pairing with the weighted adjoint variable $\sigma^{h,\xi_0,T}$, 
whose properties are established in Section \ref{subsec:4.1}.

\subsection{Properties of the weighted adjoint  variable}\label{subsec:4.1}
We begin by presenting the properties of the adjoint variable 
$\sigma^{h,\xi_0,T}$ defined in \eqref{adjoint1}. To this end, we introduce 
the formal linearized operator $L^h_t$ associated with the semi-discrete 
scheme \eqref{spatial_eq}.  For any function $\varphi:[0,T]\times\mathcal P\to\mathbb R$ that is differentiable in $t$ and bounded in $\xi$, set 
\begin{align}\label{linearized}
 L^h_t \varphi 
  := \partial_t \varphi 
  + \sum_{(i,j)\in E}\frac{\sqrt{\omega_{i,j}}}{h}
    \Bigl(
      \partial_{p_{i,j}}\mathcal{G}(\xi,[D^{\pm}u^h])\,D^+_{e_{i,j}} \varphi
      + \partial_{q_{i,j}}\mathcal{G}(\xi,[D^{\pm}u^h])\,D^-_{e_{i,j}} \varphi
    \Bigr)
  =: \partial_t \varphi + \mathcal{L}_h \varphi.
\end{align} 
The spatial operator $\mathcal L_h$ is a finite-difference operator with bounded coefficients (by Assumption~\ref{ass_G}(i) and the a~priori gradient bound on $u^h$ in Assumption \ref{ass:semi-concave}), so it is a bounded linear operator on $L^p(\widetilde{\mathcal P})$ for $p\in[1,+\infty]$ with norm at most $C/h$ for each fixed $h\in(0,h_0)$.
The operator $\mathcal{L}_h^*$, defined through  \eqref{adjoint1}, is the formal 
weighted adjoint of $\mathcal{L}_h$; 
indeed, the weighted IBP identity \eqref{IBP_identity} yields that 
\begin{align}\label{IBP2}
  \int_{\widetilde{\mathcal{P}}} (\mathcal{L}_h \varphi)\,\sigma\,w\,\mathrm{d}x 
  = -\int_{\widetilde{\mathcal{P}}} \varphi(\mathcal{L}_h^*\sigma)\,w\,\mathrm{d}x.
\end{align}
We emphasize that derivatives of the weight $w$ do not appear explicitly in \eqref{IBP2}, 
since the corresponding terms are absorbed into the geometric drift term 
$\mathcal S$ in the definition of $\mathcal L^*_h$.  
Let $\mathcal M(\widetilde{\mathcal P})$ denote the space of finite signed Radon measures on $\widetilde{\mathcal P}$, equipped with the total variation norm \begin{align}\label{TV_norm}
\|\mu\|_{\mathcal M} := \sup_{\phi}\Big\{\int_{\widetilde{\mathcal P}}\phi\,\mathrm d\mu : \phi\in\mathcal C(\widetilde{\mathcal P}),\ \|\phi\|_{L^\infty(\widetilde{\mathcal P})}\leq 1\Big\}.
\end{align}

\begin{proposition}\label{prop_sigma}
Let the conditions of Theorem~\ref{thm1} hold, $h\in(0,h_0)$, $\xi_0\in\mathcal P^\circ$, and $T>0$. Equation \eqref{adjoint1} admits a unique solution $\sigma^{h,\xi_0,T}$ in the following sense: the product $\rho(t,\cdot):=\sigma^{h,\xi_0,T}(t,\cdot)w$ defines a family of non-negative finite signed Radon measures $\rho\in\mathcal C^1([0,T];\mathcal M(\widetilde{\mathcal P}))$ with $\rho(T,\cdot)=\delta_{\Pi(\xi_0)}$, satisfying the weak formulation
\begin{align}\label{rho_weak}
\frac{\mathrm d}{\mathrm dt}\int_{\widetilde{\mathcal P}}\phi(x)\,\rho(t,x)\mathrm dx = \int_{\widetilde{\mathcal P}}(\mathcal L_h\phi)(x)\,\rho(t,x)\mathrm dx\quad\forall\phi\in\mathcal C(\widetilde{\mathcal P}),\ t\in[0,T).
\end{align} Moreover, for each $t\in[0,T]$, the weighted mass is conserved $$\int_{\widetilde{\mathcal P}}\sigma^{h,\xi_0,T}(t,x)w(x)\,\mathrm dx = 1.$$
\end{proposition}

\begin{proof}
We first prove the existence of
$\rho$ via the forward linearized equation and the Riesz representation 
theorem, then prove its uniqueness via a duality argument, and finally verify the mass conservation property. 

\textit{Existence of $\rho$ via forward duality.}
We introduce the 
forward linearized equation that is  dual to \eqref{adjoint1}.  For $f\in\mathcal C(\mathcal P)$ and $t_0\in[0,T)$, consider \begin{align}\label{adjoint2}
\begin{cases}
\partial_t \varphi^{h,f,t_0}(t,\xi) 
+ \displaystyle\sum_{(i,j)\in E}\frac{\sqrt{\omega_{i,j}}}{h}
  \Bigl(
    \partial_{p_{i,j}}\mathcal{G}(\xi,[D^{\pm}u^h])\,
    D^+_{e_{i,j}} \varphi^{h,f,t_0}(t,\xi) \\
\qquad\qquad\qquad\qquad\qquad
    + \partial_{q_{i,j}}\mathcal{G}(\xi,[D^{\pm}u^h])\,
    D^-_{e_{i,j}} \varphi^{h,f,t_0}(t,\xi)
  \Bigr) = 0, 
  \quad t > t_0,\; \xi \in \mathcal{P}^{\circ}, \\\varphi^{h,f,t_0}(t_0,\xi) = f(\xi), 
\end{cases}
\end{align} 
where values outside $\mathcal{P}$ are defined by  constant extrapolation. Equivalently, \eqref{adjoint2} can be written as  
\[
\partial_t\varphi^{h,f,t_0} + \mathcal L_h\varphi^{h,f,t_0} = 0,\quad t\in(t_0,T],\qquad \varphi^{h,f,t_0}(t_0,\cdot)=f.
\]
The operator $\mathcal L_h$ is bounded on the Banach space $\mathcal C({\mathcal P})$ (equipped with the supremum norm). Indeed, $\|D^\pm_{e_{i,j}}\varphi\|_{L^\infty}\leq 2\|\varphi\|_{L^\infty}$, and the coefficients $\partial_{p_{i,j}}\mathcal G,\partial_{q_{i,j}}\mathcal G$ are bounded by Assumption~\ref{ass_G}(i), together with the a~priori gradient bound on $u^h$, Therefore, $
\|\mathcal L_h\|_{\mathcal L(\mathcal C({P}),\mathcal C({P}))}\leq C/h
$
for some $C$ depending on $\|\partial_{p_{i,j}}\mathcal G\|_{L^\infty}, \|\partial_{q_{i,j}}\mathcal G\|_{L^\infty}$ and $|E|$. Hence by the standard Cauchy--Lipschitz theory for the linear differential equation in Banach spaces,  \eqref{adjoint2} admits a unique solution $\varphi^{h,f,t_0}\in\mathcal C^1([t_0,T];\mathcal C({\mathcal P}))$, given explicitly by \begin{align}\label{varphi:explicit}
\varphi^{h,f,t_0}(t,\cdot) = e^{-(t-t_0)\mathcal L_h}f.
\end{align}

We first establish a discrete maximum principle: for $f\geq 0$,
\begin{align}\label{positivity_v}
\min_{(t,\xi)\in[t_0,T]\times\mathcal P}\varphi^{h,f,t_0}(t,\xi)\geq\min_{\xi\in\mathcal P}f(\xi)\geq 0.
\end{align}
Fix $t_1 \in (t_0, T]$, let $\xi_1 \in \mathcal{P}$ be a spatial 
minimum point of $\varphi^{h,f,t_0}(t_1,\cdot)$, so that 
$\varphi^{h,f,t_0}(t_1,\xi_1) = \min\limits_{\xi\in\mathcal P} \varphi^{h,f,t_0}(t_1,\xi)$. 
At this minimum point, the discrete differences satisfy
\[
  D^+_{e_{i,j}} \varphi^{h,f,t_0}(t_1,\xi_1) \geq 0
  \quad\text{and}\quad
  D^-_{e_{i,j}} \varphi^{h,f,t_0}(t_1,\xi_1) \leq 0
  \quad \forall\,(i,j) \in E.
\] 
Recalling that the monotonicity of $\mathcal{G}$ implies 
$\partial_{p_{i,j}}\mathcal{G} \leq 0$ and $\partial_{q_{i,j}}\mathcal{G} \geq 0$, 
we conclude from \eqref{adjoint2} that 
$\partial_t \varphi^{h,f,t_0}(t_1,\xi_1) \geq 0$.
Hence the spatial minimum of $\varphi^{h,f,t_0}$ is non-decreasing in time, 
which gives \eqref{positivity_v}.

Applying the same argument at a spatial maximum point  
we obtain that the spatial maximum of $\varphi^{h,f,t_0}$ is non-increasing in time:
\begin{align}\label{max_v}
\max_{(t,\xi)\in[t_0,T]\times\mathcal P}\varphi^{h,f,t_0}(t,\xi) \leq \max_{\xi\in\mathcal P}f(\xi).
\end{align}
Combining \eqref{positivity_v} and \eqref{max_v},
\[
\|\varphi^{h,f,t_0}(t,\cdot)\|_{L^\infty(\mathcal P)}\leq\|f\|_{L^\infty(\mathcal P)}\quad\text{for all } t\in[t_0,T].
\] 
 Hence the map 
\[
L_{t_0}:\mathcal C({\mathcal P})\to\mathbb R,\quad L_{t_0}(f) := \varphi^{h,f,t_0}(T,\xi_0)
\]
is bounded linear with operator norm $\|L_{t_0}\|_{\mathcal L(\mathcal C({\mathcal P});\mathbb R)}\leq 1$,  and non-negative in the sense that $L_{t_0}(f)\geq 0$ whenever $f\geq 0$. By the Riesz representation theorem, $L_{t_0}$ is identified with an element of the dual 
space $\mathcal{C}(\mathcal{P})^*$, which is isomorphic 
to the space of finite signed Radon 
measures on $\mathcal{P}$. Note that $\mathcal{P}$ and $\widetilde{\mathcal{P}}$ are isomorphic 
via the bijection $\Pi$. Thus, there exists a unique non-negative finite Radon measure $\rho(t_0,\cdot)\in\mathcal M(\widetilde{\mathcal P})$ such that
\begin{align}\label{rho_def}
\int_{\widetilde{\mathcal P}}f(x)\rho(t_0,x) \mathrm dx= \varphi^{h,f,t_0}(T,\xi_0)\quad\text{for all } f\in\mathcal C(\widetilde{\mathcal P}).
\end{align}
At $t_0 = T$, we have $\varphi^{h,f,T}(T,\xi_0) = f(\xi_0)$, so $\rho(T,\cdot) = \delta_{\Pi(\xi_0)}$. 
For fixed $f\in\mathcal C({\mathcal P})$, 
from \eqref{varphi:explicit}, we compute
\[
\frac{\partial}{\partial t_0}\varphi^{h,f,t_0}(T,\xi_0) = \big[e^{-(T-t_0)\mathcal L_h}\mathcal L_h f\big](\xi_0) = \int_{\widetilde{\mathcal P}}\mathcal L_h f(x)\rho(t_0,x)\mathrm dx,
\]
where the last equality uses \eqref{rho_def} with $f$ replaced by $\mathcal L_h f$, noting that $\mathcal L_h f\in \mathcal C({\mathcal P})$. Combining this with \eqref{rho_def}, we conclude that  $\rho$ satisfies the weak formulation \eqref{rho_weak}. 

\textit{Uniqueness of $\rho$.} 
Suppose $\tilde\rho\in\mathcal C^1([0,T];\mathcal M(\widetilde{\mathcal P}))$ is another  solution satisfying \eqref{rho_weak} with $\tilde\rho(T,\cdot)=\delta_{\Pi(\xi_0)}$. Denote $\tilde\sigma:=\tilde\rho/w$. For any $f\in\mathcal C({\mathcal P})$ and $t_0\in[0,T]$, consider
\[
g(t):=\int_{\widetilde{\mathcal P}}\varphi^{h,f,t_0}(t,x)\,\tilde\sigma(t,x)\,w(x)\,\mathrm dx,\quad t\in[t_0,T].
\]
By $\varphi^{h,f,t_0}\in\mathcal C^1([t_0,T];\mathcal C({\mathcal P}))$ and $\tilde\rho\in\mathcal C^1([t_0,T];\mathcal M(\widetilde{\mathcal P}))$, $g\in\mathcal C^1([t_0,T])$ with derivative
\begin{align*}
g'(t)&= \int_{\widetilde{\mathcal P}}(\partial_t\varphi^{h,f,t_0})\,\tilde\sigma\,w\,\mathrm dx + \int_{\widetilde{\mathcal P}}\mathcal L_h\varphi^{h,f,t_0}\,\tilde\sigma\,w\,\mathrm dx \\&= -\int_{\widetilde{\mathcal P}}\mathcal L_h\varphi^{h,f,t_0}\,\tilde\sigma\,w\,\mathrm dx + \int_{\widetilde{\mathcal P}}\mathcal L_h\varphi^{h,f,t_0}\,\tilde\sigma\,w\,\mathrm dx = 0,
\end{align*}
where the first equality uses \eqref{rho_weak} and the second equality uses \eqref{adjoint2}. Hence $g$ is constant on $[t_0,T]$. Evaluating at $t=T$ and $t=t_0$,
\[
\int_{\widetilde{\mathcal P}}f(x)\,\tilde\sigma(t_0,x)\,w(x)\,\mathrm dx = g(t_0) = g(T) = \varphi^{h,f,t_0}(T,\xi_0) = \int_{\widetilde{\mathcal P}}f(x)\,\sigma^{h,\xi_0,T}(t_0,x)\,w(x)\,\mathrm dx.
\]
By the  arbitrariness of $f\in\mathcal C({\mathcal P})$, we deduce the uniqueness $\tilde\rho(t_0,\cdot)=\rho(t_0,\cdot)$.

\textit{Weighted mass conservation.}
Take $f \equiv 1$ in \eqref{adjoint2}. Since $\mathcal{L}_h[1] = 0$, 
the unique solution of \eqref{adjoint2} is $\varphi^{h,1,t_0}(t,\xi) \equiv 1$ for 
$t \in [t_0,T]$. Substituting into \eqref{rho_def} gives
\[
\int_{\widetilde{\mathcal P}}\sigma^{h,\xi_0,T}(t_0,x)w(x)\,\mathrm dx = \varphi^{h,1,t_0}(T,\xi_0) = 1.
\]
This finishes the proof.
\end{proof}
In Proposition~\ref{prop_sigma}, we write $\rho(t,x)\,\mathrm dx$ to denote the measure $\rho(t,\cdot)$, and integrals of the form $\int_{\widetilde{\mathcal P}}\phi(x)\,\rho(t,x)\,\mathrm dx$ are understood as the dual pairing of $\phi\in\mathcal C(\widetilde{\mathcal P})$ with $\rho(t,\cdot)$. This notation aligns with the one when $\sigma$ is function-valued, as in the more regular setting \eqref{adjoint3}.

\begin{lemma}\label{IBP_identity2}
Let the conditions of Theorem~\ref{thm1} hold, $h \in(0,h_0)$, 
$\xi_0 \in \mathcal{P}^{\circ}$, and $T > 0$. Then for any function 
$\varphi \in L^{\infty}\bigl((0,T)\times\mathcal{P}\bigr)$ with $\partial_t\varphi\in L^1(0,T;L^{\infty}(\mathcal P))$, 
the following weighted duality identity holds$:$
\begin{align*}
  \int_0^T \int_{\widetilde{\mathcal{P}}}
  \sigma^{h,\xi_0,T}(t,x)\,(L^h_t\varphi)(t,x)\,w(x)\,\mathrm{d}x\,\mathrm{d}t
  = \varphi(T,\xi_0) 
  - \int_{\widetilde{\mathcal{P}}}
    \varphi(0,x)\,\sigma^{h,\xi_0,T}(0,x)\,w(x)\,\mathrm{d}x,
\end{align*}where the operator $L^h$ is defined in 
\eqref{linearized}. 
\end{lemma}

\begin{proof}
By Proposition~\ref{prop_sigma},  $\rho(t,\cdot) = \sigma(t,\cdot)w$ is a probability measure with $\rho\in\mathcal C^1([0,T];\mathcal M(\widetilde{\mathcal P}))$. The weak formulation \eqref{rho_weak} extends from test functions $\phi\in\mathcal C(\widetilde{\mathcal P})$ to bounded measurable $\phi\in L^\infty(\widetilde{\mathcal P})$ by standard density arguments. In particular, all dual pairings $\int\phi\rho\,\mathrm dx,\int\phi\,\partial_t\rho\,\mathrm dx$ in the proof below are well defined for $\phi\in L^\infty(\widetilde{\mathcal P})$. Moreover, integrals involving $\mathcal L_h^*\sigma$ are interpreted via the relation $\mathcal L_h^*\sigma\cdot w = -\partial_t\rho$,  which follows from \eqref{adjoint1}. Accordingly, we have that $\int(\mathcal L_h^*\sigma)\phi\, w\,\mathrm dx := -\int\phi\,\partial_t\rho\,\mathrm dx$.

Multiplying the weighted adjoint equation \eqref{adjoint1} by $\varphi(t,x)$ 
and integrating over $(t,x) \in [0,T] \times \widetilde{\mathcal{P}}$ 
with respect to the weighted measure $w(x)\,\mathrm{d}x$, we obtain
\begin{align*}
  \int_0^T \int_{\widetilde{\mathcal{P}}} 
  (\partial_t\sigma)\,\varphi\,w\,\mathrm{d}x\,\mathrm{d}t 
  + \int_0^T \int_{\widetilde{\mathcal{P}}} 
  (\mathcal{L}_h^*\sigma)\,\varphi\,w\,\mathrm{d}x\,\mathrm{d}t = 0.
\end{align*}
Applying the weighted IBP identity \eqref{IBP_identity} (equivalently, 
\eqref{IBP2}) to the second integral, we have
\begin{align*}
  \int_0^T \int_{\widetilde{\mathcal{P}}} 
  (\mathcal{L}_h^*\sigma)\,\varphi\,w\,\mathrm{d}x\,\mathrm{d}t 
  = -\int_0^T \int_{\widetilde{\mathcal{P}}} 
  \sigma\,(\mathcal{L}_h\varphi)\,w\,\mathrm{d}x\,\mathrm{d}t,
\end{align*}
where $\mathcal{L}_h$ is the spatial operator of $L^h_t$ defined in 
\eqref{linearized}. Substituting this back and integrating the time-derivative 
term by parts in $t$, we obtain
\begin{align*}
  &\quad \int_0^T \int_{\widetilde{\mathcal{P}}} 
  (\partial_t\sigma)\,\varphi\,w\,\mathrm{d}x\,\mathrm{d}t 
  - \int_0^T \int_{\widetilde{\mathcal{P}}} 
  \sigma\,(\mathcal{L}_h\varphi)\,w\,\mathrm{d}x\,\mathrm{d}t \\
  &= -\int_0^T \int_{\widetilde{\mathcal{P}}} 
  \sigma\,(\partial_t\varphi + \mathcal{L}_h\varphi)\,w\,\mathrm{d}x\,\mathrm{d}t 
  + \Bigl[
      \int_{\widetilde{\mathcal{P}}} 
      \sigma(t,x)\,\varphi(t,x)\,w(x)\,\mathrm{d}x
    \Bigr]_{t=0}^{t=T} = 0.
\end{align*}
Recognising $\partial_t\varphi + \mathcal{L}_h\varphi = L^h_t\varphi$, we arrive at
\begin{align*}
  \int_0^T \int_{\widetilde{\mathcal{P}}} 
  \sigma\,(L^h_t\varphi)\,w\,\mathrm{d}x\,\mathrm{d}t 
  = \int_{\widetilde{\mathcal{P}}} 
    \sigma(T,x)\,\varphi(T,x)\,w(x)\,\mathrm{d}x 
  - \int_{\widetilde{\mathcal{P}}} 
    \sigma(0,x)\,\varphi(0,x)\,w(x)\,\mathrm{d}x.
\end{align*}
Substituting the terminal condition 
$\sigma(T,x) = \delta_{\Pi(\xi_0)}(x)/w(x)$ into the first term on the 
right-hand side gives
\begin{align*}
  \int_{\widetilde{\mathcal{P}}} 
  \sigma(T,x)\,\varphi(T,x)\,w(x)\,\mathrm{d}x 
  = \int_{\widetilde{\mathcal{P}}} 
    \frac{\delta_{\Pi(\xi_0)}(x)}{w(x)}\,\varphi(T,x)\,w(x)\,\mathrm{d}x 
  = \varphi(T,\xi_0),
\end{align*}
which completes the proof.
\end{proof}

\subsection{\textit{A priori} estimates of the semi-discrete solution} 

Based on the gradient estimate and the semi-concavity estimate in Assumption \ref{ass:semi-concave}, we now derive several auxiliary  \textit{a priori}  estimates for the semi-discrete solution  
$u^h$. These estimates will 
play a central role in proving the first order  convergence  of the numerical method.

\begin{proposition}\label{prop_regu}Let the conditions of Theorem~\ref{thm1} hold and $h\in(0,h_0)$. Then for every $t \in [0,T]$, the following estimates hold with a constant $C>0$ 
independent of $h$:
\begin{enumerate}
   \item[\textup{(i)}] $\displaystyle
    \sup_{\xi\in\mathcal{P}}
    \max_{(i,j)\in E}
    \frac{1}{h}\bigl|D^\pm_{e_{i,j}}u^h(t,\xi)\bigr| \leq C;$
  \item[\textup{(ii)}] $\displaystyle
    \mathcal R^+_{i,j}(t,\xi):=\nabla^{e_{i,j}}u^h(t,\xi+he_{i,j}) 
    - \frac{1}{h}D^+_{e_{i,j}}u^h(t,\xi) \leq C h,\quad \xi\in\mathcal P_{2h,e_{i,j}};$
  \item[\textup{(iii)}] $\displaystyle
    -\mathcal R^-_{i,j}(t,\xi):=-\nabla^{e_{i,j}}u^h(t,\xi-he_{i,j}) 
    + \frac{1}{h}D^-_{e_{i,j}}u^h(t,\xi) \leq C h,\quad \xi\in\mathcal P_{2h,e_{i,j}}.$
\end{enumerate}
\end{proposition}
\begin{proof}
By Proposition~\ref{prop_regularity},  we have $u^h(t,\cdot)\in\mathcal C(\mathcal P)$ and $\nabla^{e_{i,j}}u^h(t,\cdot)\in L^{\infty}(\mathcal P).$ For $\xi\in\mathcal P^{\mathrm{Int}}_{e_{i,j}}$ (see \eqref{interior}), the path $[0,1]\ni \tau \mapsto \xi \pm \tau h e_{i,j}$ stays in $\mathcal P.$ Hence, by the fundamental theorem of calculus,   $\frac{1}{h}D^\pm_{e_{i,j}}u^h(\xi) = \int_0^1\nabla^{e_{i,j}}u^h(\xi\pm\tau he_{i,j})\,\mathrm d\tau$. For $\xi\in\mathcal P\backslash \mathcal P^{\mathrm{Int}}_{e_{i,j}},$ the constant extrapolation gives $\frac1h D^{\pm}_{e_{i,j}}u^h=0.$ In either case, \begin{align}\label{M3}
  \sup_{\xi\in\mathcal{P}}
  \max_{(i,j)\in E}
  \Big(
    \frac{|D^+_{e_{i,j}}u^h|}{h} \vee \frac{|D^-_{e_{i,j}}u^h|}{h}
  \Big)
  \leq \sup_{\xi\in\mathcal{P}}
             \max_{(i,j)\in E}|\nabla^{e_{i,j}}u^h(t,\xi)|.
\end{align}  
Hence, 
(i) follows immediately from Assumption \ref{ass:semi-concave}(i). 

For (ii), we estimate the difference between the directional derivative and the 
forward finite difference quotient via a second-order Taylor expansion.  
For fixed $(t,\xi)$ and edge $(i,j)\in E$, define the path 
$\xi_{\tau} := \xi + \tau he_{i,j}$ for $\tau \in[0,1]$ and set 
$f(\tau) := u^h(t,\xi_{\tau})$. 
By Proposition~\ref{prop_regularity}, 
the path $\xi_\tau=\xi+\tau he_{i,j}$ for $\xi\in\mathcal{P}_{2h,e_{i,j}}$ satisfies 
$\xi_\tau\in\mathcal{P}_h$ for all $\tau\in[0,1]$, so $f'\in\mathcal{C}([0,1])$ 
with $f''\in L^\infty(0,1)$ piecewise. Hence, $f$ and $f'$ are absolutely continuous and \begin{align*}
\frac{1}{h}D^+_{e_{i,j}}u^h(t,\xi) 
= \frac{f(1)-f(0)}{h} 
= \frac{1}{h}\int_0^1 f'(\tau)\,d\tau,
\end{align*}
This yields 
\begin{align*}
\mathcal{R}^+_{i,j}(t,\xi) 
&= \nabla^{e_{i,j}}u^h(t,\xi+he_{i,j}) - \frac{1}{h}\int_0^1 f'(\tau)\,d\tau\\
&= \frac{1}{h}\int_0^1\bigl(f'(1)-f'(s)\bigr)\,ds= \frac{1}{h}\int_0^1\int_\tau^1 f''(r)\,dr\,d\tau. 
\end{align*} 
A direct calculation shows that the second derivative along the path is 
given by the Hessian of $u^h$ in the direction $e_{i,j}$:
$
  f''(\tau) 
  = h^2\,e_{i,j}^\top\nabla^2 u^h(t,\xi_\tau)\,e_{i,j}.
$
By the semi-concavity estimate (see Assumption \ref{ass:semi-concave}(ii)), we have 
$f''(\tau) \leq C h^2$ for all $\tau \in [0,1]$. Substituting this into the 
integral representation yields 
$
  \mathcal{R}^+_{i,j}(t,\xi) 
  \leq C h,$
which establishes (ii). Estimate~(iii) follows by an analogous argument 
applied to the backward difference $D^-_{e_{i,j}}u^h$, replacing $\xi_\tau$ 
by $\xi - \tau he_{i,j}$ and using the same semi-concavity bound; we omit further  
details. This completes the proof.  
\end{proof}

Proposition \ref{prop_regu}(ii)--(iii) presents the upper bound for the consistency 
remainders $\mathcal{R}^+_{i,j}$ and $-\mathcal R^-_{i,j}$. The following proposition presents estimates for $|\mathcal R^{\pm}_{i,j}|$ in the $L^1(\widetilde{\mathcal P})$ sense.

\begin{proposition}\label{lemma5.8}Let the conditions of Theorem~\ref{thm1} hold and $h\in(0,h_0)$. Then there exists a constant $C > 0$ independent of $h$ such that
\begin{align*}
  \int_{\widetilde{\mathcal{P}}} 
    \bigl(|\mathcal{R}^+_{i,j}(t,x)| + |\mathcal{R}^-_{i,j}(t,x)|\bigr)
  \,\mathrm{d}x \leq Ch.
\end{align*}
\end{proposition}
\begin{proof}For brevity, we write $\widetilde{\mathcal{P}}_{h,e_{i,j}} := \Pi(\mathcal{P}_{h,e_{i,j}})$ 
and $\widetilde{\mathcal{P}}_{2h,e_{i,j}} := \Pi(\mathcal{P}_{2h,e_{i,j}})$. 
We decompose the domain $\widetilde{\mathcal{P}}=\widetilde{\mathcal P}_{2h,e_{i,j}}\cup\widetilde{\mathcal P}\backslash\widetilde{\mathcal P}_{2h,e_{i,j}}$, and let
\begin{align*}
  \mathcal{I} 
  &:= \int_{\widetilde{\mathcal{P}}_{2h,e_{i,j}}} 
           \bigl(|\mathcal{R}^+_{i,j}| + |\mathcal{R}^-_{i,j}|\bigr)
      \,\mathrm{d}x
    + \int_{\widetilde{\mathcal{P}}\backslash\widetilde{\mathcal P}_{2h,e_{i,j}}} 
           \bigl(|\mathcal{R}^+_{i,j}| + |\mathcal{R}^-_{i,j}|\bigr)
      \,\mathrm{d}x
  =: \mathcal{I}_{\mathrm{Int}} + \mathcal{I}_{\mathrm{Bound}}.
\end{align*}

\medskip
\textit{Estimation of the interior term $\mathcal{I}_{\mathrm{Int}}$.}
Using the identity $|a| = 2\max\{a,0\} - a,\,a\in\mathbb R,$ together with 
Proposition~\ref{prop_regu}(ii)--(iii), we estimate each term at a point 
$\xi \in {\mathcal{P}}_{2h,e_{i,j}}$:
\begin{align*}
 |\mathcal{R}^+_{i,j}(t,\xi)| 
  &= 2\max\{0,\mathcal{R}^+_{i,j}\} 
     - \mathcal{R}^+_{i,j} 
  \leq Ch - \mathcal{R}^+_{i,j}(t,\xi), \\
  |\mathcal{R}^-_{i,j}(t,\xi)| 
  &=2\max\{0,-\mathcal{R}^-_{i,j}\} 
     + \mathcal{R}^-_{i,j} 
  \leq Ch + \mathcal{R}^-_{i,j}(t,\xi).
\end{align*}
Integrating over $\widetilde{\mathcal{P}}_{2h,e_{i,j}}$ 
gives
\[
  \mathcal{I}_{\mathrm{Int}} 
  \leq Ch 
    +\Big[- \int_{\widetilde{\mathcal{P}}_{2h,e_{i,j}}} 
      \mathcal{R}^+_{i,j}(t,x)\,\mathrm{d}x
    + \int_{\widetilde{\mathcal{P}}_{2h,e_{i,j}}} 
      \mathcal{R}^-_{i,j}(t,x)\,\mathrm{d}x\Big]
  =: Ch + I_R,
\]
where $I_R:=- \int_{\widetilde{\mathcal{P}}_{2h,e_{i,j}}} 
      \mathcal{R}^+_{i,j}(t,x)\,\mathrm{d}x
    + \int_{\widetilde{\mathcal{P}}_{2h,e_{i,j}}} 
      \mathcal{R}^-_{i,j}(t,x)\,\mathrm{d}x$ denotes the integral involving $\mathcal R^{\pm}_{i,j}.$

To bound $I_R$, we use integral representations of $\mathcal{R}^\pm_{i,j}$. 
 By 
Proposition~\ref{prop_regularity}, we have $\nabla^{e_{i,j}}u^h\in\mathcal C(\mathcal P_{h,e_{i,j}})$ and is absolutely continuous with $e^{\top}_{i,j}\nabla^2u^he_{i,j}\in L^{\infty}(\mathcal P_{h,e_{i,j}})$.   Hence for $\xi \in \mathcal{P}_{2h,e_{i,j}}$, the fundamental theorem of calculus gives
\[
  D^+_{e_{i,j}}u^h(t,\xi) 
  = \int_0^h \nabla^{e_{i,j}}u^h(t,\xi+se_{i,j})\,\mathrm{d}s,
\]
so that
\begin{align*}
  \mathcal{R}^+_{i,j}(t,\xi) 
  &= \frac{1}{h}\int_0^h 
     \Bigl(\nabla^{e_{i,j}}u^h(t,\xi+he_{i,j}) 
           - \nabla^{e_{i,j}}u^h(t,\xi+se_{i,j})\Bigr)\,\mathrm{d}s\\
           &
  = \frac{1}{h}\int_0^h\int_s^h 
    \nabla^2_{e_{i,j}}u^h(t,\xi+\tau e_{i,j})\,\mathrm{d}\tau\,\mathrm{d}s.
\end{align*} 
Here and below, we use the shorthand notation for the second directional 
derivative
\[
  \nabla^2_{e_{i,j}}u^h(\xi) 
  := e_{i,j}^\top \nabla^2 u^h(\xi)\, e_{i,j} 
  = (\partial_{\xi_i} - \partial_{\xi_j})^2 u^h(\xi),
\]
in a manner consistent with the  notation 
$\nabla^{e_{i,j}}u^h = (\partial_{\xi_i}-\partial_{\xi_j})u^h$ 
used throughout the paper.
Similarly,
\begin{align*}
  \mathcal{R}^-_{i,j}(t,\xi) 
  &= -\frac{1}{h}\int_{-h}^0\int_{-h}^s 
    \nabla^2_{e_{i,j}}u^h(t,\xi+\tau e_{i,j})\,\mathrm{d}\tau\,\mathrm{d}s.
\end{align*}
As a result,
\begin{align*}
  |I_R| 
  &\leq \frac{1}{h}\int_0^h\int_s^h 
    \Big|\int_{\widetilde{\mathcal{P}}_{2h,e_{i,j}}} 
    \nabla^2_{e_{i,j}}u^h(t,x+\tau\vec{m}_{i,j})\,\mathrm{d}x\Big|
    \,\mathrm{d}\tau\,\mathrm{d}s \\
  &\quad + \frac{1}{h}\int_{-h}^0\int_{-h}^s 
    \Big|\int_{\widetilde{\mathcal{P}}_{2h,e_{i,j}}} 
    \nabla^2_{e_{i,j}}u^h(t,x+\tau\vec{m}_{i,j})\,\mathrm{d}x\Big|
    \,\mathrm{d}\tau\,\mathrm{d}s,
\end{align*}
where we applied Fubini's theorem to exchange the order of integration. 
Applying the Gauss--Green formula (see e.g. \cite[Chapter~XXIII]{Lang}) in the direction 
$e_{i,j}$,
\begin{align*}
  \Big|\int_{\widetilde{\mathcal{P}}_{2h,e_{i,j}}} 
  \nabla^2_{e_{i,j}}u^h(t,x+\tau\vec{m}_{i,j})\,\mathrm{d}x\Big| 
  = \Big|\oint_{\partial\widetilde{\mathcal{P}}_{2h,e_{i,j}}} 
    \nabla^{e_{i,j}}u^h(t,y+\tau\vec{m}_{i,j})\,(\mathbf{n}\cdot e_{i,j})\,
    \mathrm{d}S_y\Big|,
\end{align*}
where $\mathrm  dS_y$ denotes the  $(d-2)$-dimensional surface measure on $\partial\widetilde{\mathcal{P}}_{2h,e_{i,j}}$ and $\mathbf{n}$ is the outward unit normal to 
$\partial\widetilde{\mathcal{P}}_{2h,e_{i,j}}$ and we use the fact that this boundary is Lipschitz. By 
Proposition~\ref{prop_regu}(i), $\|\nabla u^h\|_{L^\infty(\mathcal P)} \leq C$, 
and the surface area of 
$\partial\widetilde{\mathcal{P}}_{2h,e_{i,j}}$ is 
bounded independently of $h$. Therefore,
\[
  \Big|\int_{\widetilde{\mathcal{P}}_{2h,e_{i,j}}} 
  \nabla^2_{e_{i,j}}u^h(t,x+\tau\vec{m}_{i,j})\,\mathrm{d}x\Big| \leq C.
\]
Since $\frac{1}{h}\int_0^h\int_s^h\mathrm{d}\tau\,\mathrm{d}s = \frac{1}{2}h$, 
we conclude $|I_R| \leq Ch$ and hence $\mathcal{I}_{\mathrm{Int}} \leq Ch$.

\medskip
\textit{Estimation of the boundary term $\mathcal{I}_{\mathrm{Bound}}$.}
On the boundary region $\widetilde{\mathcal{P}}\backslash \widetilde{\mathcal P}_{2h,e_{i,j}}$, 
the second-order Taylor expansion and the Gauss--Green formula used in the interior estimate may fail. 
We therefore use the first-order gradient bound to estimate the upper bound for $\mathcal I_{\rm Bound}$. By definition,
\[
\mathcal{R}^+_{i,j}(t,x) 
  = \nabla^{e_{i,j}}u^h(t,x+he_{i,j}) 
    - \frac{1}{h}D^+_{e_{i,j}}u^h(t,x),
\]
and an analogous expression holds for $\mathcal{R}^-_{i,j}$. 
By Proposition~\ref{prop_regu}(i) and Assumption \ref{ass:semi-concave}(i), both 
$\|\nabla u^h\|_{L^\infty(\mathcal P)}$ and $\|\frac{1}{h}D^\pm u^h\|_{L^\infty(\mathcal P)}$ 
are bounded by $C$, so
\[
    |\mathcal{R}^+_{i,j}(t,x)| + |\mathcal{R}^-_{i,j}(t,x)| 
  \leq {2C}
  \quad \forall\, x \in \widetilde{\mathcal{P}}\backslash\widetilde{\mathcal P}_{2h,e_{i,j}}.
\]
Notice that the Lebesgue measure of the boundary layer satisfies 
$\mathrm{Vol}(\widetilde{\mathcal{P}}\backslash\widetilde{\mathcal P}_{2h,e_{i,j}} ) \leq C_G h$, where 
$C_G$ depends only on the geometry of $G$ and $\mathcal{P}$ 
but not on $h$. Consequently,
\[
  \mathcal{I}_{\mathrm{Bound}}
  \leq \int_{\widetilde{\mathcal{P}}\backslash\widetilde{\mathcal P}_{2h,e_{i,j}} } 
      2C\,\mathrm{d}x 
  \leq {2C}C_G h.
\]

Combining the interior and boundary layer estimates yields
\[
  \mathcal{I} 
  = \mathcal{I}_{\mathrm{Int}} + \mathcal{I}_{\mathrm{Bound}} \leq Ch,
\]
where $C$ is independent of $h$. This completes the proof of Proposition~\ref{lemma5.8}.\end{proof}

\subsection{$L^1$-error estimate} The convergence analysis in this subsection requires a variant of the adjoint variable $\sigma^{h,\xi_0,T}$ in \eqref{adjoint1}, with the Dirac terminal datum $\delta_{\xi_0}/w$ replaced by a general bounded measurable function $\nu\in L^\infty(\mathcal P)$. Concretely, for each fixed $\nu\in L^\infty(\mathcal P)$, let $\sigma^{h,\nu,T}$ be the solution to the same backward adjoint equation \eqref{adjoint1}, but with terminal datum  $\nu$ \begin{align}\label{adjoint3}
\begin{cases}
\partial_t\sigma^{h,\nu,T}(t,\xi) 
+ \displaystyle\sum_{(i,j)\in E}\frac{\sqrt{\omega_{i,j}}}{h}
  \Bigl(
    D^-_{e_{i,j}}\bigl(\sigma^{h,\nu,T} \partial_{p_{i,j}}\mathcal{G}(\xi,[D^\pm u^h])\bigr) \\
\quad + D^+_{e_{i,j}}\bigl(\sigma^{h,\nu,T} \partial_{q_{i,j}}\mathcal{G}(\xi,[D^\pm u^h])\bigr)
  \Bigr)
+ \mathcal{S}(\sigma^{h,\nu,T},\mathcal{G},w) = 0, 
\quad t \in (0,T),\; \xi \in \mathcal{P}^{\circ}, \\[4pt]
\sigma^{h,\nu,T}(T,\xi) = \nu(\xi), \quad \xi \in \mathcal{P}^{\circ},
\end{cases}
\end{align}where $\mathcal S$ is given in \eqref{drift_term}. With this bounded terminal datum, $\sigma^{h,\nu,T}$ is a  function rather than a measure-valued solution. For each fixed $h\in(0,h_0)$, existence and uniqueness in $\mathcal C^1([0,T];L^1_w)$ follow from standard linear evolution theory applied to \eqref{adjoint3}, since  the operator $\mathcal L_h^*$ is bounded linear on $L^1_w$, with operator norm depending on $h$. 
The next lemma provides an $h$-independent $L^\infty$-bound for $w\sigma^{h,\nu,T}$, 
which relies crucially on the semi-concavity estimate.

\begin{proposition}\label{bound_sigma}
Let the conditions of Theorem~\ref{thm1} hold and $h\in(0,h_0)$. There exists a constant 
$C > 0$ independent of $h$ such that
\[
  \sup_{t\in[0,T]}\sup_{\xi\in\mathcal{P}}
  |w(\xi)\,\sigma^{h,\nu,T}(t,\xi)| \leq C\,\|\nu\|_{L^\infty(\mathcal{P})}.
\] 
\end{proposition}
\begin{proof}
\textit{Step 1: Reduction to an evolution equation for $\Phi(t,\xi) := w(\xi)\,\sigma^{h,\nu,T}(t,\xi)$.}
Substituting $\sigma = \Phi/w$ into \eqref{adjoint_exact} and 
multiplying by $w(\xi)$ (which is independent of time), we obtain
\begin{align*}
  \partial_t\Phi 
  + \sum_{(i,j)\in E}\frac{\sqrt{\omega_{i,j}}}{h}
    \Bigl(
      D^-_{e_{i,j}}(\Phi\,\mathcal{A}_{i,j}) 
      + D^+_{e_{i,j}}(\Phi\,\mathcal{B}_{i,j})
    \Bigr) = 0
\end{align*}
on $[0,T]\times \mathcal P,$ where $\mathcal{A}_{i,j}(\xi) := \partial_{p_{i,j}}\mathcal{G}(\xi,[D^\pm u^h])$ 
and $\mathcal{B}_{i,j}(\xi) := \partial_{q_{i,j}}\mathcal{G}(\xi,[D^\pm u^h])$.
Using the discrete product rules for functions $\varphi_1,\varphi_2:\mathcal P\to\mathbb R$
 \begin{align*}&\quad D^+_{e_{i,j}}(\varphi_1\varphi_2)(\xi)= (\varphi_1\varphi_2)(\xi+he_{i,j})-(\varphi_1\varphi_2)(\xi)\\
&=\varphi_2(\xi+he_{i,j})(\varphi_1(\xi+he_{i,j})-\varphi_1(\xi))+\varphi_1(\xi)(\varphi_2(\xi+he_{i,j})-\varphi_2(\xi))\\&=
\varphi_2(\xi+he_{i,j})D^+_{e_{i,j}}\varphi_1(\xi)+\varphi_1(\xi)D^+_{e_{i,j}}\varphi_2(\xi),
\end{align*}
and 
\begin{align*}&\quad D^-_{e_{i,j}}(\varphi_1\varphi_2)(\xi)= (\varphi_1\varphi_2)(\xi)-(\varphi_1\varphi_2)(\xi-he_{i,j})\\
&=\varphi_2(\xi-he_{i,j})(\varphi_1(\xi)-\varphi_1(\xi-he_{i,j}))+\varphi_1(\xi)(\varphi_2(\xi)-\varphi_2(\xi-he_{i,j}))\\&=
\varphi_2(\xi-he_{i,j})D^-_{e_{i,j}}\varphi_1(\xi)+\varphi_1(\xi)D^-_{e_{i,j}}\varphi_2(\xi),
\end{align*}
we obtain the evolution equation for $\Phi$:
\begin{align}\label{Phi_eq1}
  \partial_t\Phi 
  + \sum_{(i,j)\in E}\frac{\sqrt{\omega_{i,j}}}{h}
    \Bigl[
      &\mathcal{A}_{i,j}(\xi-he_{i,j})\,D^-_{e_{i,j}}\Phi 
      + \mathcal{B}_{i,j}(\xi+he_{i,j})\,D^+_{e_{i,j}}\Phi \notag\\
      &+ \bigl(D^-_{e_{i,j}}\mathcal{A}_{i,j} (\xi)
               + D^+_{e_{i,j}}\mathcal{B}_{i,j}(\xi)\bigr)\Phi
    \Bigr] = 0.
\end{align}

\textit{Step 2: Upper bound for zeroth-order coefficient $\frac{\sqrt{\omega_{i,j}}}{h} (D^-_{e_{i,j}}\mathcal{A}_{i,j} + D^+_{e_{i,j}}\mathcal{B}_{i,j})$.}  
By the definition of the forward and backward differences, we have 
\begin{align}\label{differ1}
  &D^-_{e_{i,j}}\mathcal{A}_{i,j} + D^+_{e_{i,j}}\mathcal{B}_{i,j}
  = \phi(1) - \phi(0),
\end{align} 
where 
\begin{align}\label{function_phi}
\phi(\tau):=\mathcal A_{i,j}(\Xi_1(\tau))+\mathcal B_{i,j}(\Xi_2(\tau)),\;\tau\in[0,1],
\end{align}
and \begin{align*}&\Xi_1(\tau):=(\xi+(\tau-1)he_{i,j}, \tau P(\xi)+(1-\tau)P(\xi-he_{i,j}), 
\tau Q(\xi)+(1-\tau)Q(\xi-he_{i,j})),\\
&\Xi_2(\tau):=(\xi+\tau he_{i,j}, (1-\tau)P(\xi)+\tau P(\xi+he_{i,j}), 
(1-\tau )Q(\xi)+\tau Q(\xi+he_{i,j})).
\end{align*}
Here, for notational simplicity, we write \begin{align*}
&P(\xi)=[D^+u^h](\xi),\quad Q(\xi)=[D^-u^h](\xi). 
  \end{align*}
Set the increment vectors
\begin{align*}
&
  \Delta P^{\pm}_{i,j} :=(P(\xi\pm he_{i,j})-P(\xi))_{i,j}=\frac{\sqrt{\omega_{i,j}}}{h}(D^+_{e_{i,j}}u^h(\xi\pm he_{i,j})-D^+_{e_{i,j}}u^h(\xi)),\\
  &
  \Delta Q^{\pm}_{i,j} := (Q(\xi\pm he_{i,j})-Q(\xi))_{i,j}= \frac{\sqrt{\omega_{i,j}}}{h}(D^-_{e_{i,j}}u^h(\xi\pm he_{i,j})-D^-_{e_{i,j}}u^h(\xi)).
\end{align*} 
By the mean value theorem, 
\begin{align*}
\phi(1)-\phi(0)&=\Big[h\int_0^1\nabla^{e_{i,j}}\mathcal A_{i,j}(\Xi_1(\tau ))\mathrm d\tau+h\int_0^1\nabla^{e_{i,j}}\mathcal B_{i,j}(\Xi_2(\tau ))\mathrm d\tau \Big]\\
&\quad +\Big[\int_0^1\partial_{p_{i,j}}\mathcal A_{i,j}(\Xi_1(\tau))\mathrm d\tau (-\Delta P^{-}_{i,j})
+\int_0^1\partial_{q_{i,j}}\mathcal A_{i,j}(\Xi_1(\tau))\mathrm d\tau (-\Delta Q^{-}_{i,j})\\&
\qquad +\int_0^1\partial_{p_{i,j}}\mathcal B_{i,j}(\Xi_2(\tau))\mathrm d\tau (\Delta P^{+}_{i,j})+\int_0^1\partial_{q_{i,j}}\mathcal B_{i,j}(\Xi_2(\tau))\mathrm d\tau (\Delta Q^{+}_{i,j})\Big]\\
&=:J_1+J_2. 
\end{align*}

We now bound each term. Since 
$\frac{1}{h}|D^\pm_{e_{i,j}}u^h| \leq C$ by Proposition~\ref{prop_regu}(i), 
the mixed-derivative condition \eqref{growth3} yields
\begin{align}\label{eq1}
  \sup_{\xi\in\mathcal{P}}
  \Bigl(
    |\nabla^{e_{i,j}}\mathcal{A}_{i,j}(\Xi_1(\tau))|(\xi) 
    + |\nabla^{e_{i,j}}\mathcal{B}_{i,j}(\Xi_2(\tau))|(\xi)
  \Bigr) \leq C, \text{ and thus } J_1\leq Ch.
\end{align} 
For the term $J_2,$ we distinguish two cases according to where $\xi$ sits.

\underline{Case A: $\xi\in\mathcal{P}_{3h,e_{i,j}}$.} In this case,  
$\xi,\xi\pm he_{i,j}, \xi\pm 2he_{i,j}\in\mathcal{P}_{h,e_{i,j}}$, and by Proposition \ref{prop_regularity}, we have $\nabla^{e_{i,j}}u^h\in\mathcal C(\mathcal P_{h,e_{i,j}})$ and is absolutely continuous with $e^{\top}_{i,j}\nabla^2u^he_{i,j}\in L^{\infty}(\mathcal P_{h,e_{i,j}})$.  
By Taylor's theorem and the semi-concavity 
estimate in Assumption \ref{ass:semi-concave},  we have 
\begin{align*}
-\Delta P^-_{i,j}&=\frac1h\Big(D^+_{e_{i,j}} u^h(\xi)-D^+_{e_{i,j}}u^h(\xi-he_{i,j})\Big)\\
&= \frac1h(u^h(\xi+he_{i,j})-2u^h(\xi)+u^h(\xi-he_{i,j}))\\
&=\int_0^1\nabla^{e_{i,j}}(u^h(\xi+\tau he_{i,j})-u^h(\xi+(\tau-1)he_{i,j}))\mathrm d\tau\\
&=h\int_0^1\int_0^1{e_{i,j}}^{\top}\nabla^2u^h(\xi+(s-1+\tau)he_{i,j})\,{e_{i,j}}\mathrm ds\mathrm d\tau
\leq Ch,
\end{align*}
and similarly, $\max\{-\Delta Q^-_{i,j},\Delta P^+_{i,j},\Delta Q^+_{i,j}\}\leq Ch$ 
on $\mathcal{P}_{3h,e_{i,j}}$.  
Combining this with the positivity condition on second order derivatives of $\mathcal G$ with $P,Q$ variables (see \eqref{Gbounded}), we derive 
\begin{align}\label{eq2bound}
&J_2\leq Ch\text{ on }\mathcal P_{3h,e_{i,j}}.
\end{align}

\underline{Case B: $\xi\in\mathcal{P}\setminus\mathcal{P}_{3h,e_{i,j}}$.}
By the gradient bound Proposition~\ref{prop_regu}(i), the increments are uniformly bounded but only of size $\mathcal O(1)$:
\begin{align}\label{DeltaPQ_layer}
  |\Delta P^{\pm}_{i,j}|,\ |\Delta Q^{\pm}_{i,j}| \leq C
  \quad\text{on } \mathcal{P}.
\end{align}
The Taylor route used in Case A is therefore unavailable here, but the loss is compensated by the boundary vanishing of the Hamiltonian. Indeed, 
by \eqref{Gbounded} and Assumption~\ref{assumption_g}(g-ii), we have 
\begin{align}\label{eq:G_boundary}
  0 \leq \partial^2_{p_{i,j} p_{i,j}}\mathcal{G},\ \partial^2_{p_{i,j} q_{i,j}}\mathcal{G},\ \partial^2_{q_{i,j} q_{i,j}}\mathcal{G} 
  \leq C\, g_{i,j}(\xi)\leq C\xi_i\wedge\xi_j
  \quad\text{for }\xi\in\mathcal{P}.
\end{align}
Now $\xi\in\mathcal{P}\setminus\mathcal{P}_{3h,e_{i,j}}$ means that $\xi_i<3h$ or $\xi_j<3h$ for $(i,j)\in E$. 
Along the paths $\Xi_1(\tau)_\xi = \xi+(\tau-1)he_{i,j}$ and $\Xi_2(\tau)_\xi = \xi+\tau he_{i,j}$, the components $\Xi_k(t)_{\xi,i}$ and $\Xi_k(t)_{\xi,j}$ remain within $h$ of $\xi_i,\xi_j$, so by the continuity (Assumption~\ref{assumption_g}(g-i)) and 1-homogeneity (Assumption~\ref{assumption_g}(g-iii)) of $g$,
\begin{align*}
  g_{i,j}(\Xi_k(\tau)_\xi) \leq C h \quad \text{whenever } \Xi_k(\tau)_\xi\in\mathcal{P},\,\tau\in[0,1];
\end{align*}
for $\Xi_k(\tau)_\xi\notin\mathcal{P}$ the zero extension \eqref{zero_G} gives $\partial^2\mathcal{G}=0$ at that point. Combining this with \eqref{eq:G_boundary} and \eqref{DeltaPQ_layer}, we obtain 
\begin{align*}
  \int_0^1 \partial^2_{p_{i,j} p_{i,j}}\mathcal{G}(\Xi_1(\tau ))\,\mathrm d\tau \cdot \bigl(-\Delta P^-_{i,j}\bigr)
  &\leq Ch,
\end{align*}
and similarly for the other three terms in $J_2$. Therefore
\begin{align}\label{I2_layer}
  J_2 \leq Ch \quad \text{on } \mathcal{P}\setminus\mathcal{P}_{3h,e_{i,j}}.
\end{align}

Combining \eqref{eq2bound} and \eqref{I2_layer}, we obtain $J_2 \leq Ch$ on  $\mathcal{P}$. Together with \eqref{eq1} and \eqref{differ1}, we derive \begin{align*}
\frac{\sqrt{\omega_{i,j}}}{h}\bigl(D^-_{e_{i,j}}\mathcal{A}_{i,j} + D^+_{e_{i,j}}\mathcal{B}_{i,j}\bigr)
  = \frac{\sqrt{\omega_{i,j}}}{h}\bigl(\phi(1)-\phi(0)\bigr)
  \leq C
\end{align*}
uniformly in $\xi\in\mathcal{P}$.

\textit{Step 3: Gronwall argument.} Let $\xi^*(t) \in \mathcal{P}$ be a spatial maximum point of 
$\Phi(t,\cdot)$, and define $\beta(t) := \Phi(t,\xi^*(t)) = 
\|w\sigma^{h,\nu,T}(t,\cdot)\|_{L^\infty(\mathcal{P})}$. At the 
maximum point $\xi^*(t)$, the discrete differences satisfy
\[
  D^-_{e_{i,j}}\Phi(\xi^*) \geq 0 
  \quad\text{and}\quad 
  D^+_{e_{i,j}}\Phi(\xi^*) \leq 0.
\]
By the monotonicity of $\mathcal{G}$, we have $\mathcal{A}_{i,j} \leq 0$ 
and $\mathcal{B}_{i,j} \geq 0$, so the transport terms in \eqref{Phi_eq1} 
satisfy
\[
  \mathcal{A}_{i,j}(\xi^*-he_{i,j})\,D^-_{e_{i,j}}\Phi(\xi^*) 
  + \mathcal{B}_{i,j}(\xi^*+he_{i,j})\,D^+_{e_{i,j}}\Phi(\xi^*) \leq 0.
\]
Evaluating \eqref{Phi_eq1} at $\xi^*(t)$ and discarding the non-positive transport terms, we obtain
\begin{align*}
  0 \leq \partial_t\Phi(\xi^*,t) 
  + \sum_{(i,j)\in E}\frac{\sqrt{\omega_{i,j}}}{h}
    \bigl(D^-_{e_{i,j}}\mathcal{A}_{i,j} + D^+_{e_{i,j}}\mathcal{B}_{i,j}\bigr)
    \Phi(\xi^*,t)
  \leq \beta'(t) + C\,\beta(t),
\end{align*}
where the last inequality uses the bound in \textit{Step 2}. 
Applying Gronwall's inequality backward in time from $t = T$ gives
\[
  \beta(t) 
  \leq e^{C(T-t)}\,\beta(T) 
  = e^{C(T-t)}\,\|w\,\nu\|_{L^\infty(\mathcal{P})} 
  \leq C\,\|\nu\|_{L^\infty(\mathcal{P})},
  \quad t \in [0,T],
\]
which finishes the proof. 
\end{proof}

\begin{remark}
The specific structure of the Hamiltonian in
\eqref{H_special1}
is essential in the above proof. 
For a general discrete Hamiltonian $\mathcal{G}$, one computes that for $\phi$ given in \eqref{function_phi}, 
\begin{align*}
\phi'(\tau) &= h  \Big[ \nabla^{e_{i,j}}_\xi  \partial_{p_{i,j}} \mathcal{G} + \sum_{(k,l)\in E}\frac{\sqrt{\omega_{k,l}}}{h}\Big((\partial_{p_{k,l}}\partial_{p_{i,j}} \mathcal{G}) \nabla^{e_{i,j}}_\xi D^+_{e_{k,l}}u^h + (\partial_{q_{k,l}} \partial_{p_{i,j}} \mathcal{G}) \nabla^{e_{i,j}}_\xi D^-_{e_{k,l}}u^h\Big) \Big]_ {\Xi_1(\tau)}\\
& + h\Big[ \nabla^{e_{i,j}}_\xi \partial_{q_{i,j}} \mathcal{G} + \sum_{(k,l)\in E}\frac{\sqrt{\omega_{k,l}}}{h}\Big((\partial_{p_{k,l}}\partial_{q_{i,j}} \mathcal{G}) \nabla^{e_{i,j}}_\xi D^+_{e_{k,l}}u^h + (\partial_{q_{k,l}}\partial_{q_{i,j}} \mathcal{G}) \nabla^{e_{i,j}}_\xi D^-_{e_{k,l}}u^h \Big)\Big] _{\Xi_2(\tau)}.\end{align*}For the cross-edge terms $(k,l) \neq (i,j)$, controlling  the mixed second 
derivatives of $\mathcal{G}$ requires a lower bound on $\nabla^2_{\xi} u^h$, 
which is not available in general. 
This is precisely where the special structure of the discrete Hamiltonian \eqref{H_special1} is used: it decouples the edges and eliminates those cross-edge contributions, making it possible to obtain the estimate in Step 2.
\end{remark}

We next present a useful identity that will be used in the convergence proof.

\begin{lemma}\label{lemma_R+}Let the conditions of Theorem~\ref{thm1} hold, $h\in(0,h_0)$, and let $L^h_t$ be defined in 
\eqref{linearized}. Then the derivative $\partial_h u^h\in\mathcal C((0,T)\times X)$ for  $X\in\{\mathcal{P}_{2h,e_{i,j}},
\mathcal{P}_{h,e_{i,j}}\setminus\mathcal{P}_{2h,e_{i,j}},\mathcal{P}\setminus\mathcal{P}_{h,e_{i,j}}\}$ for any $(i,j)\in E$ with $\partial_h u^h(0,\xi) = 0$, and it satisfies \begin{align}\label{third_eq}
  L^h_t(\partial_h u^h) 
  + \sum_{(i,j)\in E}\frac{\sqrt{\omega_{i,j}}}{h}
    \Bigl(
      \partial_{p_{i,j}}\mathcal{G}\cdot\mathcal{R}^+_{i,j}
      + \partial_{q_{i,j}}\mathcal{G}\cdot\mathcal{R}^-_{i,j}    \Bigr) = 0, 
\end{align}where the consistency remainders $\mathcal{R}^{\pm}_{i,j}$ are given in Proposition \ref{prop_regu}, and $\partial_{p_{i,j}}\mathcal G,\partial_{q_{i,j}}\mathcal G$ are evaluated at $(\cdot,[D^+u^h],[D^-u^h])(t,\xi).$ 
\end{lemma}

\begin{proof} \textit{Step 1: Differentiability of $u^h$ with respect to $h$.}  
By Proposition~\ref{prop_regularity}, $u^h$
is spatially $\mathcal C^2$ on each of the piecewise region $X$. 
On each such region, the map $(h,z)\mapsto \Gamma_h(z):=-\mathcal{G}(\cdot,[D^+z],[D^-z])$ 
is $\mathcal{C}^1$ from $(0,h_0)\times\mathcal{C}^2$ to $\mathcal{C}$,  
since the $h$-derivatives of
$D^+_{e_{i,j}}z(\xi)=z(\xi+he_{i,j})-z(\xi)$ and $D^-_{e_{i,j}}z(\xi)=z(\xi)-z(\xi-he_{i,j})$
exist whenever $z$ is spatially $\mathcal C^1$.
 Hence on each piecewise region 
$X $, we have $\partial_h u^h\in\mathcal C((0,T)\times X),$
with $\partial_h u^h(0,\xi)=0$. 

\textit{Step 2: Derivation of~\eqref{third_eq}.} 
Using the total derivative formula for the forward difference quotient, 
we compute
\begin{align*}
&\quad \frac{\partial}{\partial h}\frac{u^h(\xi+he_{i,j})-u^h(\xi)}{h}\\
&= \frac{1}{h}\bigl[\partial_h u^h(\xi+he_{i,j}) - \partial_h u^h(\xi)\bigr]
 + \frac{\nabla^{e_{i,j}} u^h(\xi+he_{i,j})}{h}
 - \frac{u^h(\xi+he_{i,j})-u^h(\xi)}{h^2}\\
&= \frac{1}{h}D^+_{e_{i,j}}(\partial_h u^h)(\xi) + \frac{1}{h}\mathcal{R}^+_{i,j}(t,\xi),
\end{align*}
where $\mathcal{R}^+_{i,j} 
= \nabla^{e_{i,j}}u^h(\xi+he_{i,j}) - \frac{1}{h}D^+_{e_{i,j}}u^h(\xi)$ and  we have used Lemma~\ref{lemma1} and 
\eqref{def_DtildeU}. 
Differentiating \eqref{spatial_eq} with respect to $h$ and substituting 
the above expression for each edge $(i,j)\in E$ yields
\begin{align*}
  &\partial_h\partial_t u^h(t,\xi) 
  + \sum_{(i,j)\in E}\frac{\sqrt{\omega_{i,j}}}{h}
    \Bigl(
      \partial_{p_{i,j}}\mathcal{G}(\xi,[D^\pm u^h])
      \bigl(D^+_{e_{i,j}}(\partial_h u^h) + \mathcal{R}^+_{i,j}\bigr) \\
  &\hspace{1cm}
      + \partial_{q_{i,j}}\mathcal{G}(\xi,[D^\pm u^h])
      \bigl(D^-_{e_{i,j}}(\partial_h u^h) + \mathcal{R}^-_{i,j}\bigr)
    \Bigr) = 0,
\end{align*}
which is precisely \eqref{third_eq}.
\end{proof}

With Propositions~\ref{lemma5.8} and \ref{bound_sigma}, and Lemma \ref{lemma_R+} in hand, we are now in 
a position to prove the main convergence result.

\begin{proof}[Proof of Theorem \ref{thm1}]
\textit{Step 1: Weighted $L^1$-bound on $\partial_h u^h$.}
We multiply \eqref{third_eq} by $w\,\sigma^{h,\nu,T}$ and integrate over 
$\widetilde{\mathcal{P}} \times [0,T]$ using Lemma~\ref{IBP_identity2}. 
This yields
\begin{align*}
  &\quad \int_{\widetilde{\mathcal{P}}} 
    \partial_h u^h(x,T)\,\nu(x)\,w(x)\,\mathrm{d}x \\
  &= \int_0^T\int_{\widetilde{\mathcal{P}}} 
     \sum_{(i,j)\in E}\frac{\sqrt{\omega_{i,j}}}{h}
     \Bigl(
       \partial_{p_{i,j}}\mathcal{G}(\xi,[D^\pm u^h])\,\mathcal{R}^+_{i,j}(t,\xi) 
       + \partial_{q_{i,j}}\mathcal{G}(\xi,[D^\pm u^h])\,\mathcal{R}^-_{i,j}(t,\xi)
     \Bigr)\times\\
     &\quad 
     \sigma^{h,\nu,T}(x,t)\,w(x)\,\mathrm{d}x\,\mathrm{d}t.
\end{align*}
Since $L^\infty(\widetilde{\mathcal{P}})$  is the dual of 
$L^1(\widetilde{\mathcal{P}})$, we have the duality representation
\[
  \int_{\widetilde{\mathcal{P}}}
  |\partial_h u^h(x,T)|\,w(x)\,\mathrm{d}x 
  = \sup_{\substack{\nu\in L^\infty(\widetilde{\mathcal{P}})\\ \|\nu\|_{L^\infty(\widetilde{\mathcal{P}})}=1}}
    \int_{\widetilde{\mathcal{P}}}
    \partial_h u^h(x,T)\,\nu(x)\,w(x)\,\mathrm{d}x.
\]
By Proposition~\ref{bound_sigma}, 
$\sup_{\|\nu\|_{L^\infty(\mathcal P)}=1}\sup_{t\in[0,T]}\|w\,\sigma^{h,\nu,T}(t,\cdot)\|_{L^\infty(\mathcal P)}
\leq C$. This, together with the gradient bounds from Proposition \ref{prop_regu}, the condition \eqref{growth2},  
and Lemma~\ref{lemma5.8}, we obtain
\begin{align}\label{partial_h_bound}
  \int_{\widetilde{\mathcal{P}}}
  |\partial_h u^h(\xi,T)|\,w(\xi)\,\mathrm{d}\xi \leq C,
\end{align}
where $C$ is independent of $h\in(0,h_0)$ with $h_0$ given in Assumption \ref{ass:semi-concave}.

\textit{Step 2: Cauchy property and $L^1$-convergence.}
For any $0 < h_1 < h_2 \leq h_0$, the fundamental theorem of calculus gives
\[
  u^{h_2}(\xi,T) - u^{h_1}(\xi,T) 
  = \int_{h_1}^{h_2} \partial_s u^s(\xi,T)\,\mathrm{d}s.
\]
Integrating this over $\widetilde{\mathcal{P}}$ with weight $w$ and 
applying \eqref{partial_h_bound}, we obtain
\begin{align}\label{Cauchy_est}
  \int_{\widetilde{\mathcal{P}}}
  |u^{h_2}(\xi,T) - u^{h_1}(\xi,T)|\,w(\xi)\,\mathrm{d}\xi 
  \leq \int_{h_1}^{h_2}
    \int_{\widetilde{\mathcal{P}}}
    |\partial_s u^s(\xi,T)|\,w(\xi)\,\mathrm{d}\xi\,\mathrm{d}s 
  \leq C|h_2-h_1|.
\end{align} 
Hence $\{u^h(T)\}_{h>0}$ is a Cauchy family in 
$L^1_w$, and there exists a unique limit 
$u \in L^1_w$ such that 
$u^h(\cdot,T) \to u(\cdot,T)$ in $L^1_w$ as $h \to 0$.

\textit{Step 3: Identification of the limit as the viscosity solution.}
Since $\mathcal{G}$ satisfies the consistency condition 
$\mathcal{G}(\xi,P,P) = \mathcal{H}(\xi,P)$ for all $\xi \in \mathcal{P}^{\circ}$ 
(see Assumption~\ref{ass_H2}), the scheme $\partial_t u^h + \mathcal{G}(\cdot,[D^\pm u^h]) +\mathcal F= 0$ 
is a consistent, monotone finite difference approximation of the HJE 
\eqref{HJeq}. By the the arguments and uniform convergence result in \cite{CDM25}, one can derive $
\sup_{\xi \in \mathcal{P}^{\circ}}|u^h(T,\xi) - u(T,\xi)|\to0
$ as $h\to0,$ which identifies $u$ as the unique viscosity 
solution of the original HJE.

\textit{Step 4: Error estimate.}
Taking $h_1 \to 0$ in the Cauchy estimate \eqref{Cauchy_est} and letting $h_2=h$,  we obtain 
the desired first-order error bound:
\[
  \int_{\widetilde{\mathcal{P}}}
  |u^h(\xi,T) - u(\xi,T)|\,w(\xi)\,\mathrm{d}\xi \leq Ch,
\]
where $C>0$ is independent of $h$. This completes the proof 
of Theorem~\ref{thm1}.
\end{proof}

\section{Gradient and Semi-concavity Estimates for  Numerical Hamiltonians}\label{sec:examples}

In this section, we first present two families of numerical Hamiltonians
that satisfy Assumptions~\ref{ass_H2}, \ref{ass_G},
and~\ref{ass:semi-concave}. These are all associated with the local Hamiltonian 
$\mathcal{H}_{i,j}(\xi,P) = \frac{1}{2}\mathcal{I}^{-2}(\xi)\,
g_{i,j}(\xi)\,p_{i,j}^2$ for $(i,j)\in E$
(see  Example~\ref{exam1} with $\kappa=2$).
We then verify that these numerical Hamiltonians satisfy the required gradient bounds and semi-concavity estimates in Assumption \ref{ass:semi-concave}.

\subsection{Numerical Hamiltonians}\label{exam:LF}

\textbf{(I) Lax--Friedrichs type.} The Lax--Friedrichs numerical Hamiltonian is defined by
\begin{align}\label{H_special}
  \mathcal{G}^{LF}(\xi,P,Q)
  = \frac{1}{2}\mathcal{I}^{-2}(\xi)
    \sum_{(i,j)\in E} g_{i,j}(\xi)
    \Bigl(\frac{p_{i,j}^2+q_{i,j}^2}{2}
    - \gamma_{i,j}(p_{i,j}-q_{i,j})\Bigr),
\end{align}
where the viscosity coefficients $\gamma_{i,j}>0$ are free parameters of the scheme, to be specified in the verification of Assumption~\ref{ass_H2}(i) below.
Compared with the classical flat-space Lax--Friedrichs
scheme, the viscosity term $-\gamma_{i,j}(p_{i,j}-q_{i,j})$
is placed {inside} the factor
$\mathcal{I}^{-2}(\xi)g_{i,j}(\xi)$, so that
$\mathcal{G}^{LF}(\xi,\cdot,\cdot)$ vanishes naturally
on $\partial\mathcal{P}$ at the same rate as the
continuous Hamiltonian $\mathcal{H}$.
In particular, the zero extension~\eqref{zero_G}
is consistent with \eqref{H_special}.

 \textit{Verification of Assumption~\ref{ass_H2}.} 
The Hamiltonian $\mathcal{G}^{LF}$ is $\mathcal{C}^2$ in all
arguments, and consistent ($\mathcal{G}^{LF}(\xi,P,P)=\mathcal{H}(\xi,P)$). For the monotonicity, take $R_0$ to be the  gradient bound of numerical solution (see Proposition~\ref{prop_regu2}), and choose the 
viscosity coefficient $\gamma_{i,j} = \gamma_{i,j}(R_0) = 2R_0$ for all $(i,j)\in E$. Then for any fixed $R\in(0,R_0]$,  for all $\xi\in\mathcal{P}$ 
and all $(P,Q)$ with $\|P\|_{\infty}\vee\|Q\|_{\infty}\leq R$,
\[
\partial_{p_{i,j}} \mathcal G^{LF} 
= \tfrac{1}{2}\mathcal I^{-2}\, g_{i,j}\,(p_{i,j} - \gamma_{i,j}) \leq 0,
\quad
\partial_{q_{i,j}} \mathcal G^{LF} 
= \tfrac{1}{2}\mathcal I^{-2}\, g_{i,j}\,(\gamma_{i,j}- q_{i,j}) \geq 0,
\]
which verifies Assumption~\ref{ass_H2}(i).


\textit{Verification of Assumption~\ref{ass_G}.}
A direct computation gives that for any fixed $R>0,$
\begin{align*}
  \sup_{(k,l)\in E}
  \Bigl(|\partial_{p_{k,l}}\mathcal{G}^{LF}|
       +\Bigl|\frac{\partial^2\mathcal{G}^{LF}}{\partial\xi_i\partial p_{k,l}}\Bigr|
       +\Bigl|\frac{\partial^2\mathcal{G}^{LF}}{\partial\xi_i\partial q_{k,l}}\Bigr|
  \Bigr)
  &\leq C(1+R),\\
  \sup_{1\leq i,j\leq d}\Bigl|\frac{\partial^2\mathcal{G}^{LF}}{\partial\xi_i\partial\xi_j}\Bigr|
  &\leq C(1+R^2),
\end{align*} 
so \eqref{growth2}--\eqref{growth3} hold with $C(R)=C(1+R^2)$.
Furthermore,
\[\partial^2_{p_{i,j}q_{i,j}}\mathcal{G}^{LF}=0,\quad 
  \partial^2_{p_{i,j}p_{i,j}}\mathcal{G}^{LF}
    = \partial^2_{q_{i,j}q_{i,j}}\mathcal{G}^{LF}
  = \tfrac{1}{2}\mathcal{I}^{-2}(\xi)\,g_{i,j}(\xi),
\]
which verifies \eqref{Gbounded} with $C=\tfrac{1}{2} \sup_{\xi\in\mathcal{P}}
    \mathcal{I}^{-2}(\xi)<\infty.$

\textit{Verification of Assumption \ref{ass:semi-concave}.} 
The gradient and semi-concavity bounds require a more delicate 
argument, as they depend on the structure of numerical schemes and are established simultaneously via 
a bootstrap argument in Proposition~\ref{prop_regu2}. 

\textbf{(II) Osher--Sethian type.}
\label{exam:EO}
The Osher--Sethian numerical Hamiltonian is defined by
\begin{align}\label{H_special3}
  \mathcal{G}^{OS}(\xi,P,Q)
  = \frac12\mathcal{I}^{-2}(\xi)
    \sum_{(i,j)\in E} g_{i,j}(\xi)
    \Bigl(\max(q_{i,j},0)^2 + \min(p_{i,j},0)^2\Bigr).
\end{align}
Unlike the Lax--Friedrichs numerical Hamiltonian \eqref{H_special}, $\mathcal{G}^{OS}$ carries no
explicit viscosity term; dissipation is instead provided intrinsically
by the upwind selection of characteristics. Indeed, when $q_{i,j}\geq 0 
\geq p_{i,j}$, the dissipation per 
edge $(i,j)\in E$ is 
\[
\mathcal{G}^{OS}_{i,j} - \mathcal{H}_{i,j}\big(\xi,\tfrac{p_{i,j}+q_{i,j}}{2}\big)
= \tfrac{1}{2}\mathcal I^{-2}g_{i,j}(\xi)\Big(p_{i,j}^2+q_{i,j}^2 
- \tfrac{(p_{i,j}+q_{i,j})^2}{4}\Big)
= \tfrac{\mathcal I^{-2}(\xi)g_{i,j}(\xi)}{8}(p_{i,j}-q_{i,j})^2 \geq 0,
\]
while the dissipation vanishes when $p_{i,j}$ and $q_{i,j}$ 
have the same sign. 
 
 \textit{Verification of Assumption~\ref{ass_H2}.}
The consistency and monotonicity properties follow directly from the definition. 
The regularity of $\mathcal G^{OS}$ is $\mathcal{C}^2$ in $\xi$ and almost everywhere
$\mathcal{C}^2$ in $(P,Q)$, with the exceptional set
$\{p_{i,j}=0\}\cup\{q_{i,j}=0\}$ being a finite union of
hyperplanes of measure zero.

\textit{Verification of Assumption~\ref{ass_G}.}
Since $\mathcal{G}^{OS}(\xi,\cdot,\cdot)=0$ for
$\xi\in\partial\mathcal{P}$ and $\mathcal{G}^{OS}\in\mathcal{C}^2$
in $\xi$, the bounds \eqref{growth2}--\eqref{growth3} and
\eqref{Gbounded} follow by the same computation as in
Section~\ref{exam:LF}, noting
\[
  \partial^2_{p_{i,j}p_{i,j}}\mathcal{G}^{OS}
  = 2\mathcal{I}^{-2}g_{i,j}\cdot\mathbf{1}_{p_{i,j}<0},
  \quad
  \partial^2_{q_{i,j}q_{i,j}}\mathcal{G}^{OS}
  = 2\mathcal{I}^{-2}g_{i,j}\cdot\mathbf{1}_{q_{i,j}>0},
  \quad
  \partial^2_{p_{i,j}q_{i,j}}\mathcal{G}^{OS} = 0,
\]
all lying in $[0,Cg_{i,j}]$ with $C=2\sup_{\xi\in\mathcal P}\mathcal I^{-2}(\xi)<\infty$ almost everywhere. 

\textit{Verification of Assumption~\ref{ass:semi-concave}.} The verification of Assumption \ref{ass:semi-concave} is given in Proposition \ref{prop_regu2}.

\subsection{Gradient and semi-concavity  estimates}
This subsection aims to prove the following proposition, which establishes gradient and semi-concavity bounds
for each of the two numerical Hamiltonians introduced in
Sections~\ref{exam:LF}. 
\begin{proposition}\label{prop_regu2}Let $\mathcal{G}$ be one of the numerical Hamiltonians defined in
\eqref{H_special} or \eqref{H_special3}, and let $g$ be given
by Example~\ref{ex_g}. Then there exist constants $h_0\in(0,\frac1d)$, $R_0>0$, 
$K_0>0$, depending only on 
$T, \|\mathcal{U}_0\|_{\mathcal{C}^2(\mathcal{P})},\|\mathcal F\|_{\mathcal C^2(\mathcal P)}$ and $G$, such that 
for all $h\in(0,h_0)$ and $t\in[0,T]$,
\begin{enumerate}
\item[(i)] $\displaystyle 
   \sup_{\xi\in\mathcal{P}}\max_{(i,j)\in E}
   |\nabla^{e_{i,j}}u^h(t,\xi)| \leq R_0$;
\item[(ii)] $\displaystyle 
   \sup_{\mathbf{a}\in\mathbb{V}}\sup_{\xi\in\mathcal{P}}
   \mathbf{a}^\top\nabla^2 u^h(t,\xi)\,\mathbf{a} \leq K_0$.
\end{enumerate}
\end{proposition}
As a result, Assumption~\ref{ass:semi-concave} holds with 
$C = \max(R_0, K_0)$. Moreover, the 
gradient bound (i) confines the difference matrices 
$[D^{\pm}u^h]$ to the ball $\{\|P\|_{\infty}\leq R_0\}$, 
on which the monotonicity of $\mathcal{G}$ required in  Assumption \ref{ass_H2}(i) 
holds.  
The proof of Proposition~\ref{prop_regu2} relies on a duality argument and on the following differential identities.

\begin{lemma}\label{lemma_first}Let $L^h_t$ be defined in 
\eqref{linearized}. Then the following identities hold for a.e.\ $\xi\in\mathcal{P}$:  
\begin{align}
 \mathrm{(i)} & \quad
  L^h_t\Big(\frac{\partial u^h}{\partial\xi_k}\Big) 
  + \frac{\partial\mathcal{G}}{\partial\xi_k} +\frac{\partial\mathcal{F}}{\partial\xi_k} = 0,
  \quad k = 1,2,\ldots,d; 
  \label{first_eq}\\
  \mathrm{(ii)} & \quad
  L^h_t\Big(\frac{1}{2}\big(\frac{\partial u^h}{\partial\xi_k}\big)^{\!2}\Big) 
  + \sum_{(i,j)\in E}\frac{\sqrt{\omega_{i,j}}}{2h}
    \Big(
      -\partial_{p_{i,j}}\mathcal{G}\big(D^+_{e_{i,j}}\frac{\partial u^h}{\partial\xi_k}\big)^{\!2}
      + \partial_{q_{i,j}}\mathcal{G}\big(D^-_{e_{i,j}}\frac{\partial  u^h}{\partial \xi_k}\big)^{\!2}
    \Big)\notag\\
    &\quad 
  + \frac{\partial \mathcal{G}}{\partial \xi_k}\frac{\partial  u^h}{\partial \xi_k}+\frac{\partial \mathcal{F}}{\partial \xi_k}\frac{\partial  u^h}{\partial \xi_k} = 0.
  \label{fourth_eq}\\
   \mathrm{(iii)}&\quad   L^h_t\mathcal{V}_{i,j} 
  + \mathcal{Q}_{i,j} + \mathcal{M}_{i,j} 
  + \frac{\partial ^2\mathcal{G}}{\partial \xi_i\partial \xi_j} = 0,\text{ with Hessian component }\mathcal{V}_{i,j} := \frac{\partial ^2 u^h}{\partial \xi_i\partial \xi_j},\label{hessian_evo}
\end{align}where the quadratic term $\mathcal{Q}_{i,j}$ 
\begin{align}\label{def_Q}
  \mathcal{Q}_{i,j} 
  &:= \sum_{(k,l)\in E}\frac{\omega_{k,l}}{h^2}
    \Bigl[
      \partial^2_{p_{k,l}}\mathcal{G}\,
      (D^+_{e_{k,l}}\partial_{\xi_i}u^h)(D^+_{e_{k,l}}\partial_{\xi_j}u^h) 
      + \partial^2_{q_{k,l}}\mathcal{G}\,
      (D^-_{e_{k,l}}\partial_{\xi_i}u^h)(D^-_{e_{k,l}}\partial_{\xi_j}u^h) \notag\\
    &\hspace{2.5cm}
      + 2\partial_{p_{k,l}}\partial_{q_{k,l}}\mathcal{G}\,
      (D^+_{e_{k,l}}\partial_{\xi_i}u^h)(D^-_{e_{k,l}}\partial_{\xi_j}u^h)
    \Bigr],
\end{align}
and the mixed term $\mathcal{M}_{i,j}$
\begin{align}\label{def_M}
  \mathcal{M}_{i,j} 
  &:= \sum_{(k,l)\in E}\frac{\sqrt{\omega_{k,l}}}{h}
    \Bigl[
      \frac{\partial ^2\mathcal{G}}{\partial \xi_i\,\partial p_{k,l}}
      D^+_{e_{k,l}}\partial_{\xi_j}u^h 
      + \frac{\partial ^2\mathcal{G}}{\partial \xi_j\,\partial  p_{k,l}}
      D^+_{e_{k,l}}\partial_{\xi_i}u^h \notag\\
    &\hspace{2.5cm}
      + \frac{\partial ^2\mathcal{G}}{\partial \xi_i\,\partial  q_{k,l}}
      D^-_{e_{k,l}}\partial_{\xi_j}u^h 
      + \frac{\partial ^2\mathcal{G}}{\partial \xi_j\,\partial  q_{k,l}}
      D^-_{e_{k,l}}\partial_{\xi_i}u^h
    \Bigr].
\end{align}
\end{lemma}

\begin{proof}
\textit{(i) Proof of \eqref{first_eq}.} 
Differentiating \eqref{spatial_eq} with respect to $\xi_k$ directly yields 
\eqref{first_eq}.

\medskip
\textit{(ii) Proof of \eqref{fourth_eq}.}
Multiplying \eqref{first_eq} by $\phi:=\frac{\partial}{\partial \xi} u^h$ and taking into account of 
\begin{align*}
&\quad D^+_{e_{i,j}}\phi^2(\xi)=(\phi^2(\xi+he_{i,j})-\phi^2(\xi))\\&=2\phi(\xi)(\phi(\xi+he_{i,j})-\phi(\xi))+\big(\phi(\xi+he_{i,j})-\phi(\xi)\big)^2\\&
=2 \phi(\xi) D^+_{e_{i,j}}\phi(\xi)+\big(D^+_{e_{i,j}}\phi(\xi)\big)^2,
\end{align*} and similarly, 
\begin{align*}
&\quad D^-_{e_{i,j}}\phi^2(\xi)
=2 \phi(\xi) D^-_{e_{i,j}}\phi(\xi)-\big(D^-_{e_{i,j}}\phi(\xi)\big)^2,
\end{align*}  
we derive \eqref{fourth_eq}.

\textit{Proof of \eqref{hessian_evo}.}
 Differentiating \eqref{spatial_eq} twice with respect to $\xi_i$ and $\xi_j$ 
successively yields the equation \eqref{hessian_evo}. 
The proof is finished. 
\end{proof}
The following lower bound estimate plays an important role in verifying the semi-concavity lower bound.  
The proof is postponed to Appendix \ref{app1}. 
\begin{lemma}\label{lemma_G}Let $\mathcal{G}$ be one of the numerical Hamiltonians defined in \eqref{H_special} or \eqref{H_special3}, let $g$ be given by 
Example \ref{ex_g}, let
$\mathbf{a}\in\mathbb{V}$ be a unit tangent vector, and let $R>0$. 
Then there exists a constant $C>0$ depending only on the graph $G$ and metric tensor $g$ such that,
for almost every $(\xi,P,Q)\in
\mathcal{P}\times\mathbb{R}^{\frac{d^2-d}{2}}\times\mathbb{R}^{\frac{d^2-d}{2}}$, when $\|P\|_{\infty}\vee \|Q\|_{\infty}\leq R$,   
\[
\mathbf a^{\top}(\mathcal Q+\mathcal M+\nabla^2_{\xi}\mathcal G)\mathbf a\ge -CR^2,
\]where  $\mathcal{Q}=(\mathcal Q_{i,j})_{1\leq i,j\leq d},\mathcal{M}=(\mathcal M_{i,j})_{1\leq i,j\leq d}$ are given in \eqref{def_Q} and \eqref{def_M}, respectively. 
\end{lemma}

With these preliminaries in place, we now prove Proposition \ref{prop_regu2}.

\begin{proof}[Proof of Proposition \ref{prop_regu2}]
We first present the  the proof in detail for the Lax--Friedrichs Hamiltonian
\eqref{H_special}. 
(i) and (ii) are established jointly via a bootstrap argument in 
Steps~0--3.  
Throughout the proof, we assume  $\mathcal{F} = 0$ for notational simplicity. The general case  $\mathcal{F} \in \mathcal{C}^2(\mathcal{P})$ follows by a similar argument since $\mathcal{F}$ is independent of $u^h$.  

\medskip
\textbf{Step 0: Bootstrap setup.} Since $u^h \in \mathcal{C}([0,T]\times\mathcal{P}(G))$ with piecewise 
spatial $\mathcal{C}^2$-regularity by Proposition~\ref{prop_regularity}, 
the quantities
\begin{align*}
  M_1(t) &:= \sup_{\xi\in\mathcal{P}}
             \max_{(i,j)\in E}|\nabla^{e_{i,j}}u^h(t,\xi)|, \quad   M_2(t) := \sup_{\mathbf{a}\in\mathbb{V}}
             \sup_{\xi\in\mathcal{P}}
             \mathbf{a}^\top\nabla^2 u^h(t,\xi)\,\mathbf{a}
\end{align*}
are continuous in $t$, with $M_1(0) \vee M_2(0) \leq C_0$ for some constant 
$C_0 = C_0(\|\mathcal{U}_0\|_{\mathcal{C}^2(\mathcal{P})})$.
Moreover, it follows from \eqref{M3} that   
\begin{align*}  \sup_{\xi\in\mathcal{P}}
  \max_{(i,j)\in E}
  \Big(
    \frac{|D^+_{e_{i,j}}u^h|}{h} \vee \frac{|D^-_{e_{i,j}}u^h|}{h}
  \Big)
  \leq M_1(t).
\end{align*}

Choose constants $R_0, K_0$, depending only on 
$ T, \|\mathcal{U}_0\|_{\mathcal{C}^2(\mathcal{P})}$ (their precise values will be fixed in Step 3), such that  $R_0 > C_0$ and $K_0 > C_0$. 
In view of \eqref{growth2} and the definition of $\gamma_{i,j}$ in 
\eqref{H_special1}, we set the viscosity coefficients
\begin{align}\label{gammaij}
  \gamma_{i,j} := 2 R_0, 
  \quad \forall\,(i,j) \in E,
\end{align} which ensures the monotonicity of $\mathcal{G}$ in the ball 
$\{\|P\|_\infty \vee \|Q\|_\infty \leq R_0\}$. Define 
\begin{align}\label{T^*}
  T^* := \sup\bigl\{t \in [0,T] : 
    M_1(s) \leq R_0 \text{ and } M_2(s) \leq K_0 
    \text{ for all } s \in [0,t]\bigr\}.
\end{align}
By the continuity of $M_1$ and $M_2$, we have $T^* > 0$. The aim is to show $T^* = T$.  This will be proved 
 once we establish the strict inequalities $M_1(t) < R_0$ and 
$M_2(t) < K_0$ on $[0, T^*]$ for $h \in (0, h_0]$, where 
$h_0$ will be determined in Step 3 below.

\medskip
\noindent\textbf{Step 1: Gradient bound --- Part~(i).}
Differentiating \eqref{spatial_eq} with respect to $t$ gives 
$L^h_t(\partial_t u^h) = 0$. Setting $\varphi := \frac{1}{2}(\partial_t u^h)^2$ 
and applying $L^h_t$ to $\phi$, we obtain
\begin{align}\label{eqq1}
  L^h_t\varphi 
  = \frac{1}{2}\sum_{(i,j)\in E}\frac{\sqrt{\omega_{i,j}}}{h}
    \Bigl[
      \partial_{p_{i,j}}\mathcal{G}(D^+_{e_{i,j}}\partial_t u^h)^2 
      - \partial_{q_{i,j}}\mathcal{G}(D^-_{e_{i,j}}\partial_t u^h)^2
    \Bigr] \leq 0
\end{align}
on $[0, T^*]$, by the monotonicity of $\mathcal{G}$ 
(i.e., $\partial_{p_{i,j}}\mathcal{G} \leq 0$ and $\partial_{q_{i,j}}\mathcal{G} \geq 0$). 
For a fixed point $(t_1, \xi_1) \in (0,T] \times \mathcal{P}$, let 
$\sigma := \sigma^{h,\xi_1,t_1}$ denote the solution to the adjoint equation 
with terminal condition $\sigma(t_1,\cdot) = \delta_{\xi_1}/w$. 
Applying Lemma~\ref{IBP_identity2} and using \eqref{eqq1} together with 
$w, \sigma \geq 0$, we obtain 
\begin{align*}
  \varphi(t_1,\xi_1) 
  - \int_{\widetilde{\mathcal{P}}} 
    \varphi(0,x)\,\sigma(0,x)\,w(x)\,\mathrm{d}x 
  &= \int_0^{t_1}\int_{\widetilde{\mathcal{P}}} 
     \sigma\,(L^h\varphi)\,w\,\mathrm{d}x\,\mathrm{d}t \leq 0.
\end{align*}
Hence, for any $(t_1,\xi_1) \in [0,T] \times \mathcal{P}$,
\begin{align*}
  |\partial_t u^h(t_1,\xi_1)| 
  &\leq \Big(\int_{\widetilde{\mathcal{P}}} 
    |\partial_t u^h(0,x)|^2\,w(x)\,\sigma(0,x)\,\mathrm{d}x\Big)^{\frac12} \\
  &\leq \sup_{\xi\in\mathcal{P}}|\partial_t u^h(0,\xi)| 
  \leq \sup_{\xi\in\mathcal{P}}
    |\mathcal{G}(\xi,[D^\pm\mathcal{U}_0(\xi)])| 
  =: C_1,
\end{align*} 
where $C_1 $ depends on $\|\mathcal{U}_0\|_{\mathcal C^1(\mathcal P)}$, and we have used the 
mass conservation $\int\sigma\,w\,\mathrm{d}x = 1$ from 
Proposition~\ref{prop_sigma}. In particular,
\begin{align}\label{eq:G_bd}
  |\mathcal{G}(\xi,[D^\pm u^h(t,\xi)])| \leq C_1, 
  \quad (t,\xi) \in [0,T^*]\times\mathcal{P}.
\end{align}

Denote $s_{k,l} := \frac{p_{k,l}-q_{k,l}}{2}  
=  \sqrt{\omega_{k,l}}\frac{D^+_{e_{k,l}}u^h - D^-_{e_{k,l}}u^h}{2h}$.
Recall that the definition of the area $ \mathcal P_{2h,e_{i,j}}$ is given in \eqref{nested}.  
For $\xi \in \mathcal{P}_{2h,e_{k,l}}$, the semi-concavity bound 
$M_2 \leq K_0$,  $e_{k,l}^\top\nabla^2 u^he_{k,l}\in\mathcal{C}(\mathcal{P}_{2h,e_{k,l}})$ (see Proposition~\ref{prop_regularity}), and Taylor's formula yield that for some $\tau_0\in(0,1)$,
\begin{align*}
s_{k,l} &= \tfrac{h\sqrt{\omega_{k,l}}}{2}\,e_{k,l}^\top 
    \frac{1}{h^2}(u^h(\xi+he_{k,l})-2u^h(\xi)+u^h(\xi-he_{k,l}))\,e_{k,l} \notag\\
    &= \tfrac{h\sqrt{\omega_{k,l}}}{2}\,e_{k,l}^\top\nabla^2 u^h(\xi+\tau_0 h e_{k,l})\,e_{k,l} \leq CK_0 h.\end{align*}
   For $\xi \in \mathcal{P}\backslash\mathcal P_{2h,e_{k,l}}$, we instead 
use the gradient bound $M_1 \leq R_0$ to deduce 
$|s_{k,l}| \leq \frac{h\sqrt{\omega_{k,l}}}{2}\cdot\frac{|D^\pm u^h|}{h} \leq C R_0$. 
As a result, 
we have that for some $\tau_0\in(0,1)$ and for all 
$(t,\xi) \in [0,T^*]\times\mathcal{P}$,
\begin{align}\label{eq:s_upper}
s_{k,l} \leq C_2(K_0 h\mathbf 1_{\mathcal P_{2h,e_{k,l}}}+R_0\mathbf 1_{\mathcal P\backslash\mathcal P_{2h,e_{k,l}}}), 
\end{align} 

Using \eqref{eq:s_upper} and the definition of $\gamma_{i,j}$ in  $\mathcal{G} \leq C_1$, we deduce, 
for all $t\in[0,T^*],$
\begin{align}\label{C_5}
  \mathcal I^{-2}g_{k,l}|p^2_{k,l}+q^2_{k,l}| 
  \leq C_3(
    1 + \mathcal I^{-2}R_0
      (K_0 h\mathbf 1_{\mathcal P_{2h,e_{k,l}}}+R_0\mathbf 1_{\mathcal P\backslash\mathcal P_{2h,e_{k,l}}})
  ),
\end{align}where $C_3 := C_3(g,\mathcal{U}_0,E)>0.$

To bound $M_1$, we apply the adjoint method to the equation 
\eqref{fourth_eq} for $\varphi := \frac{1}{2}(\frac{\partial u^h}{\partial\xi_k})^2$. 
Let $(t_1,\xi_1)$ be a maximum point of $\varphi$ on 
$\mathcal{P} \times [0,T^*]$. Using
Lemma~\ref{IBP_identity2}  as in Step 1,   we obtain
\begin{align*}
  \frac{1}{2}\big(\frac{\partial u^h}{\partial\xi_k}\big)^{\!2}(t_1,\xi_1) 
  \leq \int_{\widetilde{\mathcal{P}}} 
    \varphi(0,x)\,\sigma(0,x)\,w(x)\,\mathrm{d}x 
  - \frac{1}{2}\int_0^{t_1}\int_{\widetilde{\mathcal{P}}} 
    \sigma\,w\,\frac{\partial\mathcal{G}}{\partial\xi_k}
    \frac{\partial u^h}{\partial\xi_k}\,\mathrm{d}x\,\mathrm{d}t.
\end{align*}
Applying Young's inequality to the last term yields
\begin{align}\label{integralG}
  \frac{1}{4}\big(\frac{\partial u^h}{\partial\xi_k}\big)^{\!2}(t_1,\xi_1) 
  \leq C + \left[
    \frac{1}{2}\int_0^{t_1}\int_{\widetilde{\mathcal{P}}} 
    \sigma\,w\,\left|\frac{\partial\mathcal{G}}{\partial\xi_k}\right|
    \mathrm{d}x\,\mathrm{d}t
  \right]^{\!2}.
\end{align}
By the bound \eqref{C_5},
\begin{align*}
 \Big|\frac{\partial\mathcal{G}}{\partial\xi_k}\Big| &\leq C\Big(1+\sup_{(i,j)\in E}(\mathcal I^{-3}\xi_k^{-2}g_{i,j}|p_{i,j}^{2}+q^2_{i,j}|+\mathcal I^{-2}\partial_{\xi_i}g_{i,j}|p_{i,j}^2+q^2_{i,j}|)\Big)\\
  &\leq C\Bigl(
    1 + R_0
        \sup_{(i,j)\in E}\bigl(
      K_0 h\,\mathbf{1}_{\mathcal{P}_{2h,e_{i,j}}} 
      + R_0\,\mathbf{1}_{\mathcal{P}\backslash\mathcal P_{2h,e_{i,j}}}
    \bigr)(\mathcal I^{-3}\xi_k^{-2}+\mathcal I^{-2}\partial_{\xi_i}g_{i,j} g^{-1}_{i,j})
  \Bigr)\\
  &\leq C\Bigl(
    1 + R_0\sup_{(i,j)\in E}
        \bigl(
      K_0 h\,\mathbf{1}_{\mathcal{P}_{2h,e_{i,j}}} 
      + R_0\,\mathbf{1}_{\mathcal{P}\backslash\mathcal P_{2h,e_{i,j}}}
    \bigr)\Big),
\end{align*} where we use the boundedness of $|\mathcal I^{-3}\xi_k^{-2}|+|\mathcal I^{-2}\partial_{\xi_i}g_{i,j} g^{-1}_{i,j}|\leq C.$  
Substituting this into \eqref{integralG} and using the mass conservation 
$\int\sigma\,w\,\mathrm{d}x = 1$, we arrive at
\begin{align*}
  \frac{1}{4}\big(\frac{\partial u^h}{\partial\xi_k}\big)^{\!2}(t_1,\xi_1) 
  \leq C_4\bigl(1 + R_0K_0 h + R_0^2 h\bigr)^2,
\end{align*} where $C_4 >0$ is a constant depending on $g,\mathcal{U}_0,E$. 
Choose $R_0$ large and require $h$ to satisfy
\begin{align}\label{h_choice1}
  R_0K_0 h \vee R_0^2h \leq R_0^{1-\zeta}, 
  \quad\text{i.e.,}\quad 
  R_0^{\zeta}K_0 h \vee R_0^{1+\zeta}h \leq 1,
\end{align}
for some $\zeta \in (0,1]$, so that
\begin{align}\label{R_0-determ}
  C_4\bigl(1 + R_0K_0 h + R_0^2 h\bigr)^2 
  < \tfrac{1}{16}R_0^2.
\end{align}
This yields
\begin{align}\label{gradient1}
  \sup_{\xi\in\mathcal{P}_\epsilon}
  \sup_{(i,j)\in E}|\nabla^{e_{i,j}}u^h(t,\xi)| 
  \leq 2\sup_{\xi,k}\,|\partial_{\xi_k}u^h(t_1,\xi_1)| 
  < R_0,
\end{align}
where we use $\nabla^{e_{i,j}}u^h = \partial_{\xi_i}u^h - \partial_{\xi_j}u^h$.

\medskip
\textbf{Step 2: Semi-concavity --- Part~(ii).}  
To control the Hessian $\mathcal V=(\mathcal V_{i,j})_{1\leq i,j\leq d}$ (see \eqref{hessian_evo} for the definition of $\mathcal V_{i,j}$), let $\mathbf{a} \in \mathbb{V}$ be a unit vector 
with $\sum_{k=1}^d a_k = 0$, and define the directional Hessian
$Z_{\mathbf{a}}(t,\xi) := \sum_{i,j=1}^d a_i a_j\mathcal{V}_{i,j}$, 
the directional derivative $\nabla_{\mathbf{a}}v := \sum_{i=1}^d a_i\partial_{\xi_i}v$, 
and the contracted forms 
$\mathcal{Q}_{\mathbf{a}} := \mathbf{a}^\top\mathcal{Q}\,\mathbf{a}$, 
$\mathcal{M}_{\mathbf{a}} := \mathbf{a}^\top\mathcal{M}\,\mathbf{a}$, where  $\mathcal{Q}=(\mathcal Q_{i,j})_{1\leq i,j\leq d},\mathcal{M}=(\mathcal M_{i,j})_{1\leq i,j\leq d}$ are given in \eqref{def_Q} and \eqref{def_M}, respectively.  
Contracting \eqref{hessian_evo} with vector $\mathbf{a}$ gives
\begin{align}\label{Za_evo}
 (L^h Z_{\mathbf{a}} 
  + \mathcal{Q}_{\mathbf{a}} + \mathcal{M}_{\mathbf{a}} 
  + \mathbf{a}^\top\nabla^2_\xi\mathcal{G}\,\mathbf{a})[u^h](t,\xi) = 0.
\end{align}
By applying Lemma \ref{lemma_G} with $R=R_0$, we obtain 
\begin{align}\label{LhZa_bound}
  L^h Z_{\mathbf{a}}[u^h] \leq C(R_0^2 + 1) 
  \quad\text{on } [0,T^*].
\end{align}

We now use \eqref{LhZa_bound} and Lemma~\ref{IBP_identity2} to establish 
$M_2(t) < K_0$ on $[0,T^*]$. Recall that 
$M_2(t) = \sup_{\xi\in\mathcal{P}}\lambda_{\max}(t,\xi)$, where 
$\lambda_{\max}(t,\xi)$ denotes the largest eigenvalue of the Hessian matrix 
$\mathcal{V}(t,\xi) = (\mathcal{V}_{i,j})_{d\times d}$.
Fix $t_1 \in (0,T^*]$. Let $\xi^* \in \mathcal{P}$ be a point at 
which $\lambda_{\max}(t_1,\cdot)$ attains its supremum, and let 
$\mathbf{a}^* \in \mathbb{V}$ be the unit eigenvector of $\mathcal{V}(t_1,\xi^*)$ 
associated with $\lambda_{\max}(t_1,\xi^*)$. Then,  by construction,
\[
  M_2(t_1) = \lambda_{\max}(t_1,\xi^*) 
  = Z_{\mathbf{a}^*}(t_1,\xi^*).
\]Let $\sigma := \sigma^{h,\xi^*,t_1}$ be the adjoint solution with terminal 
condition $\sigma(t_1,\cdot) = \delta_{\xi^*}/w(\xi^*)$. Since 
$\int_{\widetilde{\mathcal{P}}}\sigma(t_1,x)\,w(x)\,\mathrm{d}x = 1$ 
by Proposition~\ref{prop_sigma}, we have
\begin{align*}
  Z_{\mathbf{a}^*}(t_1,\xi^*) 
  &= \int_{\widetilde{\mathcal{P}}} 
    Z_{\mathbf{a}^*}(t_1,x)\,\sigma(t_1,x)\,w(x)\,\mathrm{d}x.
 \end{align*}
Applying Lemma~\ref{IBP_identity2} to $Z_{\mathbf{a}^*}$,   and then using \eqref{LhZa_bound} and the mass conservation, gives
\begin{align*}
  M_2(t_1) 
  &= \int_{\widetilde{\mathcal{P}}} 
    Z_{\mathbf{a}^*}(0,x)\,\sigma(0,x)\,w(x)\,\mathrm{d}x 
  + \int_0^{t_1}\int_{\widetilde{\mathcal{P}}} 
    \sigma\,(L^h Z_{\mathbf{a}^*})\,w\,\mathrm{d}x\,\mathrm{d}t\\
   & \leq \sup_{\xi\in\mathcal{P}}\|\nabla^2\mathcal{U}_0(\xi)\| 
  + C(R_0^2+1)\,T 
  \quad\text{on } [0,T^*].
\end{align*} 
We therefore choose
\begin{align}\label{K_0}
  K_0 > \sup_{\xi\in\mathcal{P}}|\nabla^2\mathcal{U}_0(\xi)| 
  + C(R_0^2+1)\,T,
\end{align}
so that $M_2(t) < K_0
$ holds strictly  on $[0,T^*]$.

\medskip
\textbf{Step 3: Closing the bootstrap --- Parts~(i) and~(ii).}
It remains to choose $R_0,K_0$ and $h_0$ so that the conditions 
\eqref{h_choice1}, $M_1(t) < K_0$ and 
$M_2(t) < K_0$ are valid simultaneously on $[0,T^*]$. For fixed $\zeta \in (0,1]$, set
\[
  R_0 > (\frac{1}{8\sqrt{C_4}})^{\frac{1}{1-\zeta}} \vee 1,
\]
so that
\begin{align}\label{C41}
  \sqrt{C_4}(1+R_0^{1-\zeta}) \leq 2\sqrt{C_4} R_0^{1-\zeta} \leq \frac14
  < \frac 14R_0.
\end{align}
For this fixed $R_0$, we choose $K_0$ satisfying \eqref{K_0} and set 
$h_0 := h_0(R_0, \zeta)$ small enough that \eqref{h_choice1} holds for all $h\in(0,h_0)$.
This ensures 
\[
  \sqrt{C_4}(1+R_0K_0 h+R_0^2h)  \leq \sqrt{C_4}(1+R_0^{1-\zeta}).
\]
Combining this with \eqref{C41} yields \eqref{R_0-determ}. Consequently, both
$M_1(t) < R_0$ and $M_2(t) < K_0$ 
 hold strictly on $[0,T^*]$. By the definition of $T^*$, this forces 
$T^* = T$, completing the proof of Parts~(i) and~(ii).

For the Osher--Sethian scheme~\eqref{H_special3}, the proof follows 
the same bootstrap structure as above. The only change is in the estimate corresponding to ~\eqref{C_5}: the upwind selection yields 
the estimate
\begin{align*}
  \mathcal{I}^{-2}(\xi)
  \sup_{(k,l)\in E}
  g_{k,l}(\xi)
  \bigl(|\min(p_{k,l},0)|^2\vee|\max(q_{k,l},0)|^2\bigr)
  \leq C_3,
\end{align*}
which does not depend on  $R_0$. Consequently, the auxiliary 
conditions~\eqref{h_choice1}--\eqref{R_0-determ} are no longer needed in the Osher--Sethian case, and the rest of the argument (Steps 0–3) carries over with only minor changes based on the estimate above.
\end{proof}

\section{Appendix}\label{app1}
This appendix collects the proofs of Proposition \ref{prop_regularity} and Lemma \ref{lemma_G}.

\begin{proof}[Proof of Proposition \ref{prop_regularity}]Throughout this proof, we assume $\mathcal{F}=0$ for simplicity. 
The general case $\mathcal{F}\in\mathcal{C}^2(\mathcal{P})$ 
follows by the 
same argument, since $\mathcal{F}$ is independent of $u^h$, and $\mathcal{C}^2$-regularity is sufficient to obtain the gradient and Hessian estimates of $u^h$ in \textit{Steps 4--5}.

\textit{Step 1: Local existence and uniqueness.} 
We reformulate \eqref{spatial_eq} as a differential equation in the Banach space
$\mathcal{C}(\mathcal{P})$:
\begin{align}\label{eq:ODE_zh}
\begin{cases}
\dot{z}^h(t) = \Gamma_h(z^h(t)), \quad t \in (0,\infty), \\
z^h(0) = \mathcal{U}_0,
\end{cases}
\end{align}
where $\Gamma_h : \mathcal{C}(\mathcal{P}) \to \mathcal{C}(\mathcal{P})$
is defined by
\begin{align}\label{Gammah}
  \Gamma_h(z)(\xi) := -\mathcal{G}\bigl(\xi,[D^+z](\xi),[D^-z](\xi)\bigr),
\end{align}
with the constant extrapolation convention
\begin{align}\label{eq:const_ext}
  D^\pm_{e_{i,j}}z(\xi) := 0
  \quad\text{whenever } \xi \pm he_{i,j} \notin \mathcal{P}.
\end{align}
For each fixed $h>0$, the difference operators $[D^\pm\cdot]$ are 
bounded linear operators (with operator norm $\mathcal O(\frac1h)$) on $\mathcal{C}(\mathcal{P})$. Indeed, we have 
\[
  \bigl\|[D^\pm z]\bigr\|_{\mathcal{C}(\mathcal{P};\,\mathbb{S}^{d\times d})}
  \leq \frac{C\sqrt{\omega_{\max}}}{h}\,\|z\|_{\mathcal{C}(\mathcal{P})},
  \qquad \omega_{\max}:=\max_{(i,j)\in E}\omega_{i,j}.
\]
Since $\mathcal{G}(\xi,P,Q)$ is locally Lipschitz in $(P,Q)$ for each
$\xi\in\mathcal{P}$, the operator $\Gamma_h$ is locally Lipschitz on
$\mathcal{C}(\mathcal{P})$. By the 
Picard--Lindel\"of theorem in Banach spaces, there exist $\tau > 0$ and 
a unique solution $z^h \in \mathcal{C}^1([0,\tau);\mathcal{C}(\mathcal{P}))$ 
to \eqref{eq:ODE_zh}. Setting $u^h := z^h$ yields the unique local 
solution to \eqref{spatial_eq}.

\textit{Step 2: Global existence and uniqueness.} We claim that $u^h$ satisfies the a priori bound
\begin{align}\label{eq:Linfty_bound}
  \|u^h(\cdot,t)\|_{L^\infty(\mathcal P)} 
  \leq \|\mathcal{U}_0\|_{L^\infty(\mathcal P)} + \|\mathcal{G}(\cdot,0,0)\|_{L^\infty(\mathcal P)}\,t, 
  \quad t \in (0,\infty),
\end{align}
which prevents finite-time blowup and thus extends the solution globally. 
To prove the upper bound, fix $t_1 > 0$, choose any constant 
$c_1 > \|\mathcal{G}(\cdot,0,0)\|_{L^\infty(\mathcal P)}$, and set 
$v^h(\xi,t) := u^h(\xi,t) - c_1 t$. Let $(\xi_0,t_0)$ be a 
maximizer of $v^h$ on $\mathcal{P}\times[0,t_1]$. 
Suppose for contradiction that $t_0 \in (0,t_1]$. Then 
$\partial_t v^h(\xi_0,t_0) \geq 0$ with $\partial_t v^h(\xi_0, t_0) =0$ when $t_0\in(0,t_1)$.  Since $\xi_0$ is a spatial maximizer  of $u^h(\cdot,t_0)$, we have 
\[
  [D^+u^h]_{k,l}(\xi_0,t_0)\leq 0
  \quad\text{and}\quad
  [D^-u^h]_{k,l}(\xi_0,t_0)\geq 0
  \quad\forall\,(k,l)\in E.
\]
By the monotonicity of 
$\mathcal{G}$ (non-increasing in $P$, non-decreasing in $Q$),
\[
  \mathcal{G}\bigl(\xi_0,[D^\pm u^h]\bigr) 
  \geq \mathcal{G}(\xi_0,0,0).
\]
Hence
\[
  \partial_t v^h(\xi_0,t_0) 
  = -\mathcal{G}\bigl(\xi_0,[D^\pm u^h]\bigr) - c_1 
  \leq -\mathcal{G}(\xi_0,0,0) - c_1 < 0
\] due to the selected constant $c_1 > \|\mathcal{G}(\cdot,0,0)\|_{L^\infty(\mathcal P)},$ a contradiction. Thus $t_0 = 0$, and the maximum of $v^h$ is attained 
at the initial time. A symmetric argument applied to $w^h = u^h + c_1 t$ 
yields  the corresponding lower bound. The uniform bound \eqref{eq:Linfty_bound} 
rules out finite-time blow-up of the solution,  and therefore the local solution extends uniquely to all  $t \in (0,\infty)$.

\medskip
\textit{Step 3: Piecewise structure of the domain.} 
Under the constant extrapolation convention \eqref{eq:const_ext}:
\begin{itemize}
\item[(1)] For $\xi\in\mathcal{P}_{h,e_{i,j}}$, all stencil points
$\xi\pm he_{i,j}$ lie in $\mathcal{P}$, so $[D^\pm u^h](\xi)$
depends smoothly on $\xi$ through $u^h$ alone.
\item[(2)] As $\xi$ crosses $\partial\mathcal{P}_{h,e_{i,j}}$, the values
$[D^\pm u^h]$ undergo a jump  (some components switch between
$\frac{u^h(\xi\pm he_{i,j})-u^h(\xi)}{h}$ and $0$), creating
a finite jump discontinuity in $[D^\pm u^h]$ at $\partial\mathcal{P}_{h,e_{i,j}}$.
\item[(3)] As $\xi$ crosses $\partial\mathcal{P}_{2h,e_{i,j}}$, the
{derivative} of $[D^\pm u^h]$ (which involves $\nabla^{e_{i,j}} u^h$
at shifted points) jumps, since those shifted points cross
$\partial\mathcal{P}_{h,e_{i,j}}$ where $\nabla^{e_{i,j}}u^h$ is already
discontinuous.
\end{itemize}
The analysis below is carried out on each region $\{\mathcal{P}_{2h,e_{i,j}},
\mathcal{P}_{h,e_{i,j}}\setminus\mathcal{P}_{2h,e_{i,j}},\mathcal{P}\setminus\mathcal{P}_{h,e_{i,j}}\}$  separately.

\medskip
\textit{Step 4: First-order spatial regularity: $\nabla^{e_{i,j}} u^h \in
P\mathcal C^0_{e_{i,j}}$.}

\textit{(a) Continuity on $\mathcal{P}_{h,e_{i,j}}$.} We proceed in two steps. We first formally differentiate \eqref{spatial_eq} in $\xi$ to obtain a linearized differential equation (see \eqref{eq:linear_V}) and show the well-posedness of the solution $V^h$. Since the existence of the Euclidean gradient $\nabla_\xi u^h$ 
has not yet been established, 
$V^h$ is introduced as a candidate for this gradient.
Next, We show that the difference quotient $\frac{u^h(t,\xi+\tau e_{i,j})-u^h(t,\xi)}{\|\tau e_{i,j}\|_{l^2}} 
 - V^h\cdot\frac{ e_{i,j}}{\|e_{i,j}\|_{l^2}}$ vanishes as $\tau\to 0$. This proves at the same time the existence of $\nabla^{e_{i,j}} u^h$ and its identification with $V^h\cdot e_{i,j}$.  The claimed continuity then follows from the continuity of  $V^h$. 

For $\xi\in\mathcal{P}_{h,e_{i,j}}$, no extrapolation is involved in the direction $e_{i,j}$ and
$\Gamma_h(u^h)(\xi)$ is $\mathcal{C}^1$ in $\xi$.
Consider the linearized differential equation on
$\mathcal{P}$ 
\begin{equation}\label{eq:linear_V}
\begin{cases}
  \dot{V}^h(t) = \mathcal{L}^h(t,V^h(t)), \quad t\in(0,\infty),\\
  V^h(0) = \nabla_\xi\mathcal{U}_0.
\end{cases}
\end{equation}
where for $W\in\mathcal{C}(\mathcal{P};\mathbb{R}^d)$, the operator
$\mathcal{L}^h(t,\cdot)$ acts componentwise. More precisely, its $k$-th 
component is given by 
\[
  \bigl(\mathcal{L}^h(t,W)\bigr)_k(\xi)
  := -\partial_{\xi_k}\mathcal{G}\big|_*
     - \sum_{(i,j)\in E}\partial_{p_{i,j}}\mathcal{G}\big|_*
       [D^+W_k]_{i,j}(\xi)
     - \sum_{(i,j)\in E}\partial_{q_{i,j}}\mathcal{G}\big|_*
       [D^-W_k]_{i,j}(\xi),
\]
for $k=1,\dots,d$, with $*=(\xi,[D^\pm u^h(t,\xi)])$.
Here, the differences $[D^\pm W_k]_{i,j}$ are computed under the same
constant extrapolation convention \eqref{eq:const_ext} as for $u^h$. 
The coefficients $\partial_{p_{i,j}}\mathcal{G}|_*$, 
$\partial_{q_{i,j}}\mathcal{G}|_*$, $\nabla_\xi\mathcal{G}|_*$ are 
bounded measurable functions of $\xi$ on $\mathcal{P}$ for fixed $h>0$, 
and are continuous on each region 
$\{\mathcal{P}_{h,e_{i,j}},\mathcal{P}\setminus\mathcal{P}_{h,e_{i,j}}\}$, and  $[D^\pm \cdot]$ are bounded linear 
operators on $L^\infty(\mathcal{P})$ with operator norm $\mathcal O(1/h)$. 
By the Picard--Lindel\"of theorem in Banach spaces, for each fixed $h,$ there exists a unique global 
solution $V^h\in\mathcal{C}([0,T];L^\infty(\mathcal{P};\mathbb{R}^d))\cap \mathcal C([0,T];\mathcal C(\mathcal P_{h,e_{i,j}};\mathbb R^d))$.
In particular, $\mathcal{L}^h(t,V^h(t))\cdot e_{i,j}$ is well defined as an 
element of $L^\infty(\mathcal{P})$.

To identify $V^h\cdot e_{i,j} = \nabla^{e_{i,j}}u^h$, fix 
a tangent direction $y=\tau e_{i,j}\in\mathbb{R}^d$ with $\tau>0$ sufficiently small so that $\xi$ and $\xi+y$ 
lie in the same open region $
\mathcal{P}_{h,e_{i,j}}$. Set 
$
  \delta_y u^h(t,\xi) 
  := \frac{u^h(t,\xi+y)-u^h(t,\xi)}{\|y\|_{l^2}}.
$
Differencing \eqref{spatial_eq} at points $\xi+y$ and $\xi$, and using the local Lipschitz 
property of $\mathcal{G}$ (Assumption~\ref{ass_H2}(iii)) and 
its $\mathcal{C}^1$ regularity in $\xi$ (Assumption~\ref{ass_G}) 
at $\xi_\theta:=\xi+\theta y$, $\theta\in(0,1),$ we obtain 
\begin{align*}
\partial_t({\delta}_y u^h)(t,\xi) 
  &= -\Bigl(
      \nabla_\xi\mathcal{G}\big|_\theta\cdot\tfrac{y}{\|y\|_{l^2}} 
      + \sum_{(k,l)\in E}\partial_{p_{k,l}}\mathcal{G}\big|_\theta
        [D^+\delta_y u^h]_{k,l}(\xi) 
      + \sum_{(k,l)\in E}\partial_{q_{k,l}}\mathcal{G}\big|_\theta
        [D^-\delta_y u^h]_{k,l}(\xi)
    \Bigr)\\
    &
  =: \mathcal{R}^h(t,\delta_y u^h(t)),
\end{align*}
where $|_\theta$ denotes evaluation of the derivatives 
at $(\xi_\theta, [D^\pm u^h(\xi_\theta)])$ 
with $\xi_\theta := \xi+\theta y$ for some $\theta\in(0,1)$, 
given by
\[
\nabla_\xi\mathcal{G}|_\theta := \int_0^1\nabla_\xi\mathcal{G}
(\xi+\theta y,[D^\pm u^h(\xi+\theta y)])\,d\theta,
\]
and similarly for $\partial_{p_{k,l}}\mathcal{G}|_\theta$ and 
$\partial_{q_{k,l}}\mathcal{G}|_\theta$. 
Setting $e_y(t):=\delta_y u^h(t) - V^h(t)\cdot\frac{y}{\|y\|_{l^2}}$, 
we have 
\[
  \dot{e}_y(t) = \mathcal{R}^h(t,e_y(t)) + \phi_{h,y}(t),
\]
where 
$\phi_{h,y}:=\mathcal{R}^h(t,V^h\cdot\frac{y}{\|y\|_{l^2}}) 
- \mathcal{L}^h(t,V^h)\cdot\frac{y}{\|y\|_{l^2}}$. 
By the $\mathcal{C}^1$-regularity 
of $\mathcal{G}$ and $\xi_\theta\to\xi$ as $y\to 0$, we have 
$\|\phi_{h,y}(t)\|_{\mathcal{C}(\mathcal{P}_{h,e_{i,j}})}\to 0$ 
as $y\to 0$, uniformly in $t\in[0,T]$, 
Moreover, $\|\phi_{h,y}(\cdot)\|_{\mathcal{C}(\mathcal{P}_{h,e_{i,j}})}$ 
is uniformly bounded in $s \in [0,T]$ by the boundedness of $V^h$. By 
Gronwall's inequality, 
\[
  \|e_y(t)\|_{\mathcal{C}(\mathcal{P}_{h,e_{i,j}})} 
  \leq e^{C_1 t}\|e_y(0)\|_{\mathcal{C}(\mathcal{P}_{h,e_{i,j}})} 
  + \int_0^t e^{C_1(t-s)}
    \|\phi_{h,y}(s)\|_{\mathcal{C}(\mathcal{P}_{h,e_{i,j}})}\,\mathrm{d}s.
\]
Taking $y\to 0$ and applying the dominated convergence theorem, 
both terms vanish, proving
$\nabla^{e_{i,j}} u^h = V^h\cdot e_{i,j}\in\mathcal{C}((0,T)\times\mathcal{P}_{h,e_{i,j}})$.

\textit{(b) Continuity on $\mathcal{P}\setminus\mathcal{P}_{h,e_{i,j}}$.}
For $\xi\in\mathcal{P}\setminus\mathcal{P}_{h,e_{i,j}}$, whenever
$\xi+he_{i,j}\notin\mathcal{P}$, the constant extrapolation sets
$[D^+_{e_{i,j}}u^h](\xi)\equiv 0$ as a function of $\xi$, so
$\partial_{\xi_k}[D^+_{e_{i,j}}u^h](\xi)=0$ for all $k=1,\ldots,d.$  Hence the
corresponding terms in $\mathcal{L}^h$ and $\mathcal R^h$ vanish, and
$\mathcal{L}^h,\mathcal R^h$ reduce to sums over only those edges $(i,j)$
for which $\xi\pm he_{i,j}\in\mathcal{P}$.  The coefficients
of this reduced operator are continuous on
$\mathcal{P}\setminus\mathcal{P}_{h,e_{i,j}}$.
Applying the same linearisation argument as in part (a) 
 gives $\nabla_\xi u^h\in\mathcal{C}
((0,T)\times(\mathcal{P}\setminus\mathcal{P}_{h,e_{i,j}}))$.

\textit{(c) Finite jump across $\partial\mathcal{P}_{h,e_{i,j}}$.}
As $\xi$ approaches $\partial\mathcal{P}_{h,e_{i,j}}$ from inside
$\mathcal{P}_{h,e_{i,j}}$, the difference
$[D^+_{e_{i,j}}u^h](\xi) = (u^h(\xi+he_{i,j})-u^h(\xi))/h$
is bounded by
$2\|u^h\|_{L^\infty(\mathcal{P})}/h$, whereas from outside
$\mathcal{P}_{h,e_{i,j}}$ it is zero by the constant extrapolation
convention~\eqref{eq:const_ext}.
Hence the coefficient
$\partial_{p_{i,j}}\mathcal{G}|_{*}$
of $\mathcal{L}^h$ jumps by some constant $C_h,$ which depends on $h$ but is finite for each 
fixed $h>0$.
By the Gronwall's inequality argument, the jump in $\nabla^{e_{i,j}}u^h$
can be bounded by some constant $\tilde C_h$ depending on $C_h$:
\[
  \bigl\|\nabla^{e_{i,j}}u^h\big|_{\partial\mathcal{P}_{h,e_{i,j}}^+}
  - \nabla^{e_{i,j}}u^h\big|_{\partial\mathcal{P}_{h,e_{i,j}}^-}\bigr\|_{l^2}
  \leq \tilde C_h < \infty.
\]

Combining parts~(a)--(c) with the $L^\infty$ bound of \textit{Step~2}, we conclude that 
\[
  \nabla^{e_{i,j}} u^h \in P\mathcal C^0_{e_{i,j}}\bigl((0,T)\times\mathcal{P}\bigr) \quad\text{with finite jumps on } \partial\mathcal{P}_{h,e_{i,j}},
  \qquad
  \|\nabla_\xi u^h(\cdot,t)\|_{L^\infty(\mathcal{P})} \leq C_h<\infty,
\]where $C_h>0$ depends on $h$ that may be different at each appearance.

\medskip
\textit{Step 5: Second-order spatial regularity: 
$e^{\top}_{i,j}\nabla^2 u^h e_{i,j} \in P\mathcal C^0_{e_{i,j}}$.}

\textit{(a) Continuity on $\mathcal{P}_{2h,e_{i,j}}$.} 
For $\xi\in\mathcal{P}_{2h,e_{i,j}}$, every shifted point
$\xi\pm he_{i,j}$ lies in $\mathcal{P}_{h,e_{i,j}}$, so
$\nabla_\xi[D^\pm u^h](\xi) = [D^\pm\nabla_\xi u^h](\xi)$ 
without any extrapolation, and
$\nabla_\xi u^h\in\mathcal{C}(\mathcal{P}_{h,e_{i,j}})$ from Step~4(a).
Differentiating \eqref{eq:linear_V} with respect to $\xi$ yields
the linearized differential equation for the candidate Hessian
$W^h(t):=\nabla^2u^h(t)$ satisfying 
\[
  \dot{W}^h(t) = \mathcal{L}^h_2(t,W^h(t)),
  \quad W^h(0) = \nabla^2\mathcal{U}_0,
\]
where $\mathcal{L}^h_2$ depends on the second derivatives of
$\mathcal{G}$ and on $V^h=\nabla_\xi u^h$. Both are continuous and
bounded on $\mathcal{P}_{2h,e_{i,j}}$ for fixed $h>0$.
The same Gronwall argument as in \textit{Step 4 (a)} gives
\[
  e^{\top}_{i,j}\nabla^2 u^he_{i,j} \in
  \mathcal{C}\bigl((0,T)\times\mathcal{P}_{2h,e_{i,j}}\bigr).
\]

\textit{(b) Continuity on $\mathcal{P}_{h,e_{i,j}}\setminus\mathcal{P}_{2h,e_{i,j}}$
and $\mathcal{P}\setminus\mathcal{P}_{h,e_{i,j}}$.}
For $\xi\in\mathcal{P}_{h,e_{i,j}}\setminus\mathcal{P}_{2h,e_{i,j}}$, some shifted
points $\xi\pm he_{i,j}$ lie in $\mathcal{P}\setminus\mathcal{P}_{h,e_{i,j}}$,
where $\nabla^{e_{i,j}} u^h$ is continuous by \textit{Step 4(b)}. As $\xi$ varies
within $\mathcal{P}_{h,e_{i,j}}\setminus\mathcal{P}_{2h,e_{i,j}}$, the shifted points
remain in $\mathcal{P}\setminus\mathcal{P}_{h,e_{i,j}}$ without crossing
$\partial\mathcal{P}_{h,e_{i,j}}$. Thus the coefficients of 
$\mathcal{L}^h_2$
are continuous on $\mathcal{P}_{h,e_{i,j}}\setminus\mathcal{P}_{2h,e_{i,j}}$ and the
same reasoning as in \textit{Step 5 (a)} shows
$e_{i,j}^{\top}\nabla^2u^he_{i,j}\in\mathcal{C}((0,T)\times
(\mathcal{P}_{h,e_{i,j}}\setminus\mathcal{P}_{2h,e_{i,j}}))$.
An analogous argument on $\mathcal{P}\setminus\mathcal{P}_{h,e_{i,j}}$
 gives
$e_{i,j}^{\top}\nabla^2u^he_{i,j}\in\mathcal{C}((0,T)\times
(\mathcal{P}\setminus\mathcal{P}_{h,e_{i,j}}))$.

\textit{(c) Finite jumps across $\partial\mathcal{P}_{2h,e_{i,j}}$ and
$\partial\mathcal{P}_{h,e_{i,j}}$.}
As $\xi$ crosses $\partial\mathcal{P}_{2h,e_{i,j}}$, the shifted point
$\xi\pm he_{i,j}$ crosses $\partial\mathcal{P}_{h,e_{i,j}}$, where
$\nabla^{e_{i,j}} u^h$ has a finite jump of size $\mathcal O(C_h)$.
This induces a finite jump in the coefficients of $\mathcal{L}^h_2$,
and therefore
\[
  \bigl\|e^{\top}_{i,j}(\nabla^2 u^h\big|_{\partial\mathcal{P}_{2h,e_{i,j}}^+}
  - \nabla^2u^h\big|_{\partial\mathcal{P}_{2h,e_{i,j}}^-})e_{i,j}\bigr\|_{l^2}
  \leq C_h < \infty.
\] A second finite jump of the same form occurs at
$\partial\mathcal{P}_{h,e_{i,j}}$ for the same reason as in \textit{Step~4(c)}.
Hence
\[
  e^{\top}_{i,j}\nabla^2 u^he_{i,j} \in P\mathcal C^0_{e_{i,j}}\bigl((0,T)\times\mathcal{P}\bigr),
  \qquad
  \|e^{\top}_{i,j}\nabla^2u^h(\cdot,t)e_{i,j}\|_{L^\infty(\mathcal{P})} \leq C_h,
\]
with discontinuities confined to the measure-zero hypersurfaces
$\partial\mathcal{P}_{2h,e_{i,j}}$ and $\partial\mathcal{P}_{h,e_{i,j}}$.

Combining Steps~1--5 completes the proof. 
\end{proof}

\begin{proof}[Proof of Lemma \ref{lemma_G}]

\textit{Verification for \eqref{H_special}.}
Let $R>0$ be fixed and $\|P\|_{\infty}\vee\|Q\|_{\infty}\leq R.$ By the  {definitions} \eqref{def_Q} and \eqref{def_M}, one can calculate that 
\[
\mathbf a^{\top}\mathcal{Q} \mathbf a = \mathcal{I}^{-2}(\xi)\sum_{(k,l)\in E}g_{k,l}(\xi)\frac{\omega_{k,l}}{h^2}\Big[\big(D^+_{e_{k,l}}\nabla_{\mathbf a} u^h \big)^2+ \big(D^-_{e_{k,l}}\nabla_{\mathbf a} u^h\big)^2\Big].
\]
 For the mixed term $
\mathcal{M}$, it follows from \eqref{H_special} that 
\begin{align*}
&\frac{\partial ^2\mathcal{G}}{\partial \xi_i\partial  p_{k,l}} = \partial_{\xi_i}(\tfrac{1}{2}\mathcal{I}^{-2} g_{k,l})\cdot p_{k,l}-\partial_{\xi_i}(\tfrac12\mathcal I^{-2}g_{k,l}),\\
&\frac{\partial ^2\mathcal{G}}{\partial \xi_i\partial  q_{k,l}} = \partial_{\xi_i}(\tfrac{1}{2}\mathcal{I}^{-2} g_{k,l})\cdot q_{k,l}+\partial_{\xi_i}(\tfrac12\mathcal I^{-2}g_{k,l}). 
\end{align*}
By using the 
 Young inequality with a small parameter 
$\gamma_0\in(0,1)$, we obtain 
\begin{align*}
\mathbf a^{\top}\mathcal M \mathbf a
&
\ge  -\sum_{(k,l)\in E}\Big[C(\gamma_0)|\nabla_{\mathbf a}(\tfrac12\mathcal I^{-2}g_{k,l})|^2 (1+|p_{k,l}|^2+|q_{k,l}|^2)(\mathcal{I}^{-2}g_{k,l})^{-1}
\Big]\\
&\quad -\gamma_0\mathcal{I}^{-2}\sum_{(k,l)\in E}\frac{\omega_{k,l}}{h^2}g_{k,l}\Big[(D^+_{e_{k,l}}\nabla_{\mathbf a}u^h)^2+(D^-_{e_{k,l}}\nabla_{\mathbf a}u^h)^2\Big]\\
&\ge -CR^2-
\gamma_0\mathcal{I}^{-2}\sum_{(k,l)\in E}\frac{\omega_{k,l}}{h^2}g_{k,l}\Big[(D^+_{e_{k,l}}\nabla_{\mathbf a}u^h)^2+(D^-_{e_{k,l}}\nabla_{\mathbf a}u^h)^2\Big], 
\end{align*} where the last term can be absorbed by 
$\mathbf a^{\top}\mathcal Q\mathbf a$, yielding $\mathbf a^{\top}(\mathcal Q+\mathcal M)\mathbf a\ge -CR^2.$ 
Similarly, 
since $\mathcal{G}$ is quadratic in $(P,Q)$ and its 
$\xi$-coefficients are smooth and bounded on $\mathcal{P}$, a direct computation 
using~\eqref{H_special}  gives 
$\mathbf a^\top\nabla^2_\xi\mathcal{G}\,\mathbf a \geq -C\sup_{(k,l)}(|p_{k,l}|^2+|q_{k,l}|^2) 
\geq -CR^2$. 
Thus we obtain that 
\[\mathbf a^{\top}(\mathcal Q+\mathcal M)\mathbf a +\mathbf a^{\top}\nabla^2_{\xi}\mathcal G\mathbf a\ge -CR^2.
\]

\textit{Verification for \eqref{H_special3}.} The proof proceeds analogously to the verification of  \eqref{H_special},  and thus is omitted. 
\end{proof}

\bibliographystyle{plain}
\bibliography{references.bib}

@article {WENO,
    AUTHOR = {Shu, C.-W.},
     TITLE = {High order weighted essentially nonoscillatory schemes for
              convection dominated problems},
   JOURNAL = {SIAM Rev.},
  FJOURNAL = {SIAM Review},
    VOLUME = {51},
      YEAR = {2009},
    NUMBER = {1},
     PAGES = {82--126},
      ISSN = {1095-7200,0036-1445},
   MRCLASS = {65M06},
  MRNUMBER = {2481112},
MRREVIEWER = {M.\ K.\ Kadalbajoo},
       DOI = {10.1137/070679065},
       URL = {https://doi-org.ezproxy.lb.polyu.edu.hk/10.1137/070679065},
}

@article {Cagnetti13,
    AUTHOR = {Cagnetti, F. and Gomes, D. and Tran, H. V.},
     TITLE = {Convergence of a semi-discretization scheme for the
              {H}amilton-{J}acobi equation: a new approach with the adjoint
              method},
   JOURNAL = {Appl. Numer. Math.},
  FJOURNAL = {Applied Numerical Mathematics. An IMACS Journal},
    VOLUME = {73},
      YEAR = {2013},
     PAGES = {2--15},
      ISSN = {0168-9274,1873-5460},
   MRCLASS = {65M06},
  MRNUMBER = {3146863},
MRREVIEWER = {K.\ Kanakadurga},
       DOI = {10.1016/j.apnum.2013.05.004},
       URL = {https://doi-org.ezproxy.lb.polyu.edu.hk/10.1016/j.apnum.2013.05.004},
}

@article {Tadmor2,
    AUTHOR = {Tadmor, E.},
     TITLE = {Local error estimates for discontinuous solutions of nonlinear
              hyperbolic equations},
   JOURNAL = {SIAM J. Numer. Anal.},
  FJOURNAL = {SIAM Journal on Numerical Analysis},
    VOLUME = {28},
      YEAR = {1991},
    NUMBER = {4},
     PAGES = {891--906},
      ISSN = {0036-1429},
   MRCLASS = {35L65 (65M12 65M15)},
  MRNUMBER = {1111445},
MRREVIEWER = {Benoit\ Perthame},
       DOI = {10.1137/0728048},
       URL = {https://doi-org.ezproxy.lb.polyu.edu.hk/10.1137/0728048},
}

@book {Lang,
    AUTHOR = {Lang, S.},
     TITLE = {Real and {F}unctional {A}nalysis},
    SERIES = {Graduate Texts in Mathematics},
    VOLUME = {142},
   EDITION = {Third},
 PUBLISHER = {Springer-Verlag, New York},
      YEAR = {1993},
     PAGES = {xiv+580},
      ISBN = {0-387-94001-4},
   MRCLASS = {00A05 (26-01 28-01 46-01 47-01 58-01)},
  MRNUMBER = {1216137},
       DOI = {10.1007/978-1-4612-0897-6},
       URL = {https://doi-org.ezproxy.lb.polyu.edu.hk/10.1007/978-1-4612-0897-6},
}

@article {Barles91,
    AUTHOR = {Barles, G. and Souganidis, P. E.},
     TITLE = {Convergence of approximation schemes for fully nonlinear
              second order equations},
   JOURNAL = {Asymptotic Anal.},
  FJOURNAL = {Asymptotic Analysis},
    VOLUME = {4},
      YEAR = {1991},
    NUMBER = {3},
     PAGES = {271--283},
      ISSN = {0921-7134},
   MRCLASS = {35K55 (35A40 35J60 65M12)},
  MRNUMBER = {1115933},
}

@article {Crandall84,
    AUTHOR = {Crandall, M. G. and Lions, P.-L.},
     TITLE = {Two approximations of solutions of {H}amilton-{J}acobi
              equations},
   JOURNAL = {Math. Comp.},
  FJOURNAL = {Mathematics of Computation},
    VOLUME = {43},
      YEAR = {1984},
    NUMBER = {167},
     PAGES = {1--19},
      ISSN = {0025-5718,1088-6842},
   MRCLASS = {65M10 (35F20 49C10)},
  MRNUMBER = {744921},
MRREVIEWER = {Steven\ M.\ Serbin},
       DOI = {10.2307/2007396},
       URL = {https://doi-org.ezproxy.lb.polyu.edu.hk/10.2307/2007396},
}

@article {Souganidis85,
    AUTHOR = {Souganidis, P.},
     TITLE = {Approximation schemes for viscosity solutions of
              {H}amilton-{J}acobi equations},
   JOURNAL = {J. Differential Equations},
  FJOURNAL = {Journal of Differential Equations},
    VOLUME = {59},
      YEAR = {1985},
    NUMBER = {1},
     PAGES = {1--43},
      ISSN = {0022-0396,1090-2732},
   MRCLASS = {35F20 (35L60 65M10)},
  MRNUMBER = {803085},
MRREVIEWER = {E.\ N.\ Barron},
       DOI = {10.1016/0022-0396(85)90136-6},
       URL = {https://doi-org.ezproxy.lb.polyu.edu.hk/10.1016/0022-0396(85)90136-6},
}

@misc{XiangY1,
      title={Discrete Mean Field Games on Finite Graphs as Initial Value Optimization}, 
      author={Y. Feng and Y. Xiang and H. Zhou},
      year={2026},
      eprint={2604.05685},
      archivePrefix={arXiv},
      primaryClass={math.NA},
      url={https://arxiv.org/abs/2604.05685}, 
}

@article {LT2001,
    AUTHOR = {Lin, C.-T. and Tadmor, E.},
     TITLE = {{$L^1$}-stability and error estimates for approximate
              {H}amilton-{J}acobi solutions},
   JOURNAL = {Numer. Math.},
  FJOURNAL = {Numerische Mathematik},
    VOLUME = {87},
      YEAR = {2001},
    NUMBER = {4},
     PAGES = {701--735},
      ISSN = {0029-599X,0945-3245},
   MRCLASS = {65M12 (35A35 35F25 49L20)},
  MRNUMBER = {1815732},
MRREVIEWER = {Alexander\ Kurganov},
       DOI = {10.1007/PL00005430},
       URL = {https://doi-org.ezproxy.lb.polyu.edu.hk/10.1007/PL00005430},
}

@article{CDM25,
      title={Finite difference schemes for {H}amilton--{J}acobi equation on Wasserstein space on graphs}, 
      author={J. Cui and T. Dang and C. Mou},
      year={2026},
      eprint={2504.13463},
      JOURNAL={To appear in SIAM Journal on Numerical Analysis},
      primaryClass={math.NA},
     }

@article {Mielke,
    AUTHOR = {Mielke, A.},
     TITLE = {Geodesic convexity of the relative entropy in reversible
              {M}arkov chains},
   JOURNAL = {Calc. Var. Partial Differential Equations},
  FJOURNAL = {Calculus of Variations and Partial Differential Equations},
    VOLUME = {48},
      YEAR = {2013},
    NUMBER = {1-2},
     PAGES = {1--31},
      ISSN = {0944-2669,1432-0835},
   MRCLASS = {60J27 (53C21 53C23 82B35)},
  MRNUMBER = {3090532},
MRREVIEWER = {Roman\ Urban},
       DOI = {10.1007/s00526-012-0538-8},
       URL = {https://doi.org/10.1007/s00526-012-0538-8},
}

@book {HJSK,
    AUTHOR = {Hofbauer, J. and Sigmund, K.},
     TITLE = {The {T}heory of {E}volution and {D}ynamical {S}ystems: {M}athematical {A}spects of {S}election},
    SERIES = {London Mathematical Society Student Texts},
    VOLUME = {7},
 PUBLISHER = {Cambridge University Press, Cambridge},
      YEAR = {1988},
     PAGES = {viii+341},
      ISBN = {0-521-35288-6; 0-521-35838-8},
   MRCLASS = {92D15 (58F10 58F21 58F40 92-01)},
  MRNUMBER = {1071180},
MRREVIEWER = {Gabriela\ Schranz-Kirlinger},
}

@article {CGKPR,
    AUTHOR = {Cosso, A. and Gozzi, F. and Kharroubi, I. and Pham,
              H. and Rosestolato, M.},
     TITLE = {Master {B}ellman equation in the {W}asserstein space:
              uniqueness of viscosity solutions},
   JOURNAL = {Trans. Amer. Math. Soc.},
  FJOURNAL = {Transactions of the American Mathematical Society},
    VOLUME = {377},
      YEAR = {2024},
    NUMBER = {1},
     PAGES = {31--83},
      ISSN = {0002-9947,1088-6850},
   MRCLASS = {49L25 (35B51 35Q89 49N80)},
  MRNUMBER = {4684588},
MRREVIEWER = {Benjamin\ Seeger},
       DOI = {10.1090/tran/8986},
       URL = {https://doi.org/10.1090/tran/8986},
}

@article {CDJS,
    AUTHOR = {Cardaliaguet, P. and Daudin, S. and Jackson, J. and
              Souganidis, P. E.},
     TITLE = {An algebraic convergence rate for the optimal control of
              {M}c{K}ean-{V}lasov dynamics},
   JOURNAL = {SIAM J. Control Optim.},
  FJOURNAL = {SIAM Journal on Control and Optimization},
    VOLUME = {61},
      YEAR = {2023},
    NUMBER = {6},
     PAGES = {3341--3369},
      ISSN = {0363-0129,1095-7138},
   MRCLASS = {93E20 (49N80 60H30)},
  MRNUMBER = {4665660},
MRREVIEWER = {Agamirza\ E.\ Bashirov},
       DOI = {10.1137/22M1486789},
       URL = {https://doi.org/10.1137/22M1486789},
}

@article {BFY,
    AUTHOR = {Bensoussan, A. and Frehse, J. and Yam, S. C.
              P.},
     TITLE = {The master equation in mean field theory},
   JOURNAL = {J. Math. Pures Appl. (9)},
  FJOURNAL = {Journal de Math\'ematiques Pures et Appliqu\'ees. Neuvi\`eme
              S\'erie},
    VOLUME = {103},
      YEAR = {2015},
    NUMBER = {6},
     PAGES = {1441--1474},
      ISSN = {0021-7824,1776-3371},
   MRCLASS = {35Q93 (49N70 60H15 91A06)},
  MRNUMBER = {3343705},
MRREVIEWER = {Caterina\ Sartori},
       DOI = {10.1016/j.matpur.2014.11.005},
       URL = {https://doi.org/10.1016/j.matpur.2014.11.005},
}

@article {BCP,
    AUTHOR = {Bayraktar, E. and Cosso, A. and Pham, H.},
     TITLE = {Randomized dynamic programming principle and {F}eynman-{K}ac
              representation for optimal control of {M}c{K}ean-{V}lasov
              dynamics},
   JOURNAL = {Trans. Amer. Math. Soc.},
  FJOURNAL = {Transactions of the American Mathematical Society},
    VOLUME = {370},
      YEAR = {2018},
    NUMBER = {3},
     PAGES = {2115--2160},
      ISSN = {0002-9947,1088-6850},
   MRCLASS = {49L20 (60H10 60H30 60K35 93E20)},
  MRNUMBER = {3739204},
MRREVIEWER = {Salvatore\ Federico},
       DOI = {10.1090/tran/7118},
       URL = {https://doi.org/10.1090/tran/7118},
}

@article {GT,
    AUTHOR = {Gangbo, W. and Nguyen, T. and Tudorascu, A.},
     TITLE = {Hamilton-{J}acobi equations in the {W}asserstein space},
   JOURNAL = {Methods Appl. Anal.},
  FJOURNAL = {Methods and Applications of Analysis},
    VOLUME = {15},
      YEAR = {2008},
    NUMBER = {2},
     PAGES = {155--183},
      ISSN = {1073-2772,1945-0001},
   MRCLASS = {49L25 (35F21 47J30 82C40)},
  MRNUMBER = {2481677},
MRREVIEWER = {Fabio\ Camilli},
       DOI = {10.4310/MAA.2008.v15.n2.a4},
       URL = {https://doi.org/10.4310/MAA.2008.v15.n2.a4},
}

@article {GNT,
    AUTHOR = {Gangbo, W. and Tudorascu, A.},
     TITLE = {On differentiability in the {W}asserstein space and
              well-posedness for {H}amilton-{J}acobi equations},
   JOURNAL = {J. Math. Pures Appl. (9)},
  FJOURNAL = {Journal de Math\'ematiques Pures et Appliqu\'ees. Neuvi\`eme
              S\'erie},
    VOLUME = {125},
      YEAR = {2019},
     PAGES = {119--174},
      ISSN = {0021-7824,1776-3371},
   MRCLASS = {58D25 (35B30 35D40 35F21 46G05)},
  MRNUMBER = {3944201},
MRREVIEWER = {Asuka\ Takatsu},
       DOI = {10.1016/j.matpur.2018.09.003},
       URL = {https://doi.org/10.1016/j.matpur.2018.09.003},
}

@book {CarmonaDelarue2,
    AUTHOR = {Carmona, R. and Delarue, F.},
     TITLE = {Probabilistic {T}heory of {M}ean {F}ield {G}ames with {A}pplications 
              {II}: {M}ean {F}ield {G}ames with {C}ommon {N}oise and {M}aster {E}quations},
    SERIES = {Probability Theory and Stochastic Modelling},
    VOLUME = {84}, 
 PUBLISHER = {Springer, Cham},
      YEAR = {2018},
     PAGES = {xxiv+697},
      ISBN = {978-3-319-56435-7; 978-3-319-56436-4},
   MRCLASS = {60-02 (35R60 49L20 60G55 60H10 60H30 91A13 91A15)},
  MRNUMBER = {3753660},
MRREVIEWER = {Vassili\ N.\ Kolokol\cprime tsov},
}

@book {CarmonaDelarue1,
    AUTHOR = {Carmona, R. and Delarue, F.},
     TITLE = {Probabilistic {T}heory of {M}ean {F}ield {G}ames with {A}pplications 
              {I}: {M}ean {F}ield {FBSDE}s, {C}ontrol, and {G}ames},
    SERIES = {Probability Theory and Stochastic Modelling},
    VOLUME = {83}, 
 PUBLISHER = {Springer, Cham},
      YEAR = {2018},
     PAGES = {xxv+713},
      ISBN = {978-3-319-56437-1; 978-3-319-58920-6},
   MRCLASS = {60-02 (35R60 49N70 49N90 60H15 60H30 91A15 93E20)},
  MRNUMBER = {3752669},
MRREVIEWER = {Vassili\ N.\ Kolokol\cprime tsov},
}

@incollection {CardaliaguetPorretta,
    AUTHOR = {Cardaliaguet, P. and Porretta, A.},
     TITLE = {An introduction to mean field game theory},
 BOOKTITLE = {Mean {F}ield {G}ames},
    SERIES = {Lecture Notes in Math.},
    VOLUME = {2281},
     PAGES = {1--158},
 PUBLISHER = {Springer, Cham},
      YEAR = {2020},
      ISBN = {978-3-030-59837-2; 978-3-030-59836-5},
   MRCLASS = {35Q91 (49N80 60H30 60K35 91A16 93E20)},
  MRNUMBER = {4214774},
MRREVIEWER = {John\ G.\ O'Hara},
       DOI = {10.1007/978-3-030-59837-2\_1},
       URL = {https://doi.org/10.1007/978-3-030-59837-2_1},
}

@article {LasryLions,
    AUTHOR = {Lasry, J.-M. and Lions, P.-L.},
     TITLE = {Mean field games},
   JOURNAL = {Jpn. J. Math.},
  FJOURNAL = {Japanese Journal of Mathematics},
    VOLUME = {2},
      YEAR = {2007},
    NUMBER = {1},
     PAGES = {229--260},
      ISSN = {0289-2316,1861-3624},
   MRCLASS = {91A23 (82B05 91B28)},
  MRNUMBER = {2295621},
       DOI = {10.1007/s11537-007-0657-8},
       URL = {https://doi.org/10.1007/s11537-007-0657-8},
}

@article {MFG_Caines,
    AUTHOR = {Huang, M. and Malham\'e, R. P. and Caines, P. E.},
     TITLE = {Large population stochastic dynamic games: closed-loop
              {M}c{K}ean-{V}lasov systems and the {N}ash certainty
              equivalence principle},
   JOURNAL = {Commun. Inf. Syst.},
  FJOURNAL = {Communications in Information and Systems},
    VOLUME = {6},
      YEAR = {2006},
    NUMBER = {3},
     PAGES = {221--251},
      ISSN = {1526-7555,2163-4548},
   MRCLASS = {91A15 (49L20 91A23)},
  MRNUMBER = {2346927},
       DOI = {10.4310/cis.2006.v6.n3.a5},
       URL = {https://doi.org/10.4310/cis.2006.v6.n3.a5},
}

@article {Daudin2,
    AUTHOR = {Daudin, S. and Jackson, J. and Seeger, B.},
     TITLE = {Well-posedness of {H}amilton-{J}acobi equations in the
              {W}asserstein space: non-convex {H}amiltonians and common
              noise},
   JOURNAL = {Comm. Partial Differential Equations},
  FJOURNAL = {Communications in Partial Differential Equations},
    VOLUME = {50},
      YEAR = {2025},
    NUMBER = {1-2},
     PAGES = {1--52},
      ISSN = {0360-5302,1532-4133},
   MRCLASS = {35F21 (35D40 49L25 49N80)},
  MRNUMBER = {4858218},
       DOI = {10.1080/03605302.2024.2432678},
       URL = {https://doi.org/10.1080/03605302.2024.2432678},
}

@article {Daudin,
    AUTHOR = {Daudin, S. and Delarue, F. and Jackson, J.},
     TITLE = {On the optimal rate for the convergence problem in mean field
              control},
   JOURNAL = {J. Funct. Anal.},
  FJOURNAL = {Journal of Functional Analysis},
    VOLUME = {287},
      YEAR = {2024},
    NUMBER = {12},
     PAGES = {Paper No. 110660, 94},
      ISSN = {0022-1236,1096-0783},
   MRCLASS = {49N80 (49L25 65C35)},
  MRNUMBER = {4794919},
       DOI = {10.1016/j.jfa.2024.110660},
       URL = {https://doi.org/10.1016/j.jfa.2024.110660},
}

@article {approximation_HJB2,
    AUTHOR = {Mayorga, S. and \'Swi\k{e}ch, A.},
     TITLE = {Finite dimensional approximations of
              {H}amilton-{J}acobi-{B}ellman equations for stochastic
              particle systems with common noise},
   JOURNAL = {SIAM J. Control Optim.},
  FJOURNAL = {SIAM Journal on Control and Optimization},
    VOLUME = {61},
      YEAR = {2023},
    NUMBER = {2},
     PAGES = {820--851},
      ISSN = {0363-0129,1095-7138},
   MRCLASS = {35D40 (35F21 35R15 49L25 49N80 93E20)},
  MRNUMBER = {4580743},
       DOI = {10.1137/22M1489186},
       URL = {https://doi.org/10.1137/22M1489186},
}

@article {approximation_HJB,
    AUTHOR = {Gangbo, W. and Mayorga, S. and \'Swi\k{e}ch,
              A.},
     TITLE = {Finite dimensional approximations of
              {H}amilton-{J}acobi-{B}ellman equations in spaces of
              probability measures},
   JOURNAL = {SIAM J. Math. Anal.},
  FJOURNAL = {SIAM Journal on Mathematical Analysis},
    VOLUME = {53},
      YEAR = {2021},
    NUMBER = {2},
     PAGES = {1320--1356},
      ISSN = {0036-1410,1095-7154},
   MRCLASS = {35D40 (26E15 35F21 35R15 49L25 49N80)},
  MRNUMBER = {4226237},
       DOI = {10.1137/20M1331135},
       URL = {https://doi.org/10.1137/20M1331135},
}

@article {Cui_Hamiltonian,
    AUTHOR = {Cui, J. and Liu, S. and Zhou, H.},
     TITLE = {Wasserstein {H}amiltonian flow with common noise on graph},
   JOURNAL = {SIAM J. Appl. Math.},
  FJOURNAL = {SIAM Journal on Applied Mathematics},
    VOLUME = {83},
      YEAR = {2023},
    NUMBER = {2},
     PAGES = {484--509},
      ISSN = {0036-1399,1095-712X},
   MRCLASS = {58B20 (35Q41 49Q20 58J65)},
  MRNUMBER = {4577196},
MRREVIEWER = {Jialin\ Hong},
       DOI = {10.1137/22M1490697},
       URL = {https://doi.org/10.1137/22M1490697},
}

@article {Chow2,
    AUTHOR = {Chow, S.-N. and Huang, W. and Li, Y. and Zhou, H.},
     TITLE = {Fokker-{P}lanck equations for a free energy functional or
              {M}arkov process on a graph},
   JOURNAL = {Arch. Ration. Mech. Anal.},
  FJOURNAL = {Archive for Rational Mechanics and Analysis},
    VOLUME = {203},
      YEAR = {2012},
    NUMBER = {3},
     PAGES = {969--1008},
      ISSN = {0003-9527,1432-0673},
   MRCLASS = {35R02 (60H30 60J25)},
  MRNUMBER = {2928139},
MRREVIEWER = {Sergey\ Dashkovskiy},
       DOI = {10.1007/s00205-011-0471-6},
       URL = {https://doi.org/10.1007/s00205-011-0471-6},
}

@article {Chow1,
    AUTHOR = {Chow, S.-N. and Li, W. and Zhou, H.},
     TITLE = {A discrete {S}chr\"odinger equation via optimal transport on
              graphs},
   JOURNAL = {J. Funct. Anal.},
  FJOURNAL = {Journal of Functional Analysis},
    VOLUME = {276},
      YEAR = {2019},
    NUMBER = {8},
     PAGES = {2440--2469},
      ISSN = {0022-1236,1096-0783},
   MRCLASS = {35R02 (05C90 35Q55 49Q20)},
  MRNUMBER = {3926122},
       DOI = {10.1016/j.jfa.2019.02.005},
       URL = {https://doi.org/10.1016/j.jfa.2019.02.005},
}

@article {Maas,
    AUTHOR = {Maas, J.},
     TITLE = {Gradient flows of the entropy for finite {M}arkov chains},
   JOURNAL = {J. Funct. Anal.},
  FJOURNAL = {Journal of Functional Analysis},
    VOLUME = {261},
      YEAR = {2011},
    NUMBER = {8},
     PAGES = {2250--2292},
      ISSN = {0022-1236,1096-0783},
   MRCLASS = {49Q20 (49J10 60J27)},
  MRNUMBER = {2824578},
MRREVIEWER = {Nung\ Kwan\ Yip},
       DOI = {10.1016/j.jfa.2011.06.009},
       URL = {https://doi.org/10.1016/j.jfa.2011.06.009},
}

@article {Liwuchen,
    AUTHOR = {Gangbo, W. and Li, W. and Mou, C.},
     TITLE = {Geodesics of minimal length in the set of probability measures
              on graphs},
   JOURNAL = {ESAIM Control Optim. Calc. Var.},
  FJOURNAL = {ESAIM. Control, Optimisation and Calculus of Variations},
    VOLUME = {25},
      YEAR = {2019},
     PAGES = {Paper No. 78, 36},
      ISSN = {1292-8119,1262-3377},
   MRCLASS = {49Q20 (28A33 49K35 60J20)},
  MRNUMBER = {4039140},
MRREVIEWER = {Erik\ J.\ Balder},
       DOI = {10.1051/cocv/2018052},
       URL = {https://doi.org/10.1051/cocv/2018052},
}

@article {MCC,
    AUTHOR = {Gangbo, W. and Mou, C. and \'Swi\k{e}ch,
              A.},
     TITLE = {Well-posedness for {H}amilton-{J}acobi equations on the
              {W}asserstein space on graphs},
   JOURNAL = {Calc. Var. Partial Differential Equations},
  FJOURNAL = {Calculus of Variations and Partial Differential Equations},
    VOLUME = {63},
      YEAR = {2024},
    NUMBER = {7},
     PAGES = {Paper No. 160, 41},
      ISSN = {0944-2669,1432-0835},
   MRCLASS = {35D40 (35F21 35R15 49L25 49Q20 60H30)},
  MRNUMBER = {4768504},
       DOI = {10.1007/s00526-024-02758-w},
       URL = {https://doi.org/10.1007/s00526-024-02758-w},
}

\end{document}